\documentclass[reqno]{amsart}

\usepackage{amsfonts}
\usepackage{amsaddr}
\usepackage{graphicx}
\usepackage{epstopdf}
\usepackage{array}
\usepackage{subcaption}
\usepackage{colonequals}
\usepackage{bm}
\usepackage{bbm}
\usepackage{multirow}
\usepackage{booktabs}
\usepackage{mathrsfs}
\usepackage{enumitem}
\ifpdf
  \DeclareGraphicsExtensions{.eps,.pdf,.png,.jpg}
\else
  \DeclareGraphicsExtensions{.eps}
\fi
\usepackage{hyperref}
\hypersetup
{
	colorlinks,
	citecolor = black,
	filecolor = black,
	linkcolor = black,
	urlcolor  = black
}

\RequirePackage[capitalize,nameinlink]{cleveref}[0.19]

\crefname{section}{section}{sections}
\crefname{subsection}{subsection}{subsections}
\Crefname{section}{Section}{Sections}
\Crefname{subsection}{Subsection}{Subsections}

\Crefname{figure}{Figure}{Figures}

\crefformat{equation}{\textup{#2(#1)#3}}
\crefrangeformat{equation}{\textup{#3(#1)#4--#5(#2)#6}}
\crefmultiformat{equation}{\textup{#2(#1)#3}}{ and \textup{#2(#1)#3}}
{, \textup{#2(#1)#3}}{, and \textup{#2(#1)#3}}
\crefrangemultiformat{equation}{\textup{#3(#1)#4--#5(#2)#6}}%
{ and \textup{#3(#1)#4--#5(#2)#6}}{, \textup{#3(#1)#4--#5(#2)#6}}{, and \textup{#3(#1)#4--#5(#2)#6}}

\Crefformat{equation}{#2Equation~\textup{(#1)}#3}
\Crefrangeformat{equation}{Equations~\textup{#3(#1)#4--#5(#2)#6}}                               
\Crefmultiformat{equation}{Equations~\textup{#2(#1)#3}}{ and \textup{#2(#1)#3}}                 
{, \textup{#2(#1)#3}}{, and \textup{#2(#1)#3}}                                                  
\Crefrangemultiformat{equation}{Equations~\textup{#3(#1)#4--#5(#2)#6}}%
{ and \textup{#3(#1)#4--#5(#2)#6}}{, \textup{#3(#1)#4--#5(#2)#6}}{, and \textup{#3(#1)#4--#5(#2)#6}}
                                                                                                
\crefdefaultlabelformat{#2\textup{#1}#3}

\newcommand{\scp}[2]{{\left\langle {#1}\, , \, {#2}\right\rangle}}
\newcommand{\mR}{{\mathbb{R}}}
\newcommand{\mB}{{\mathbb{B}}}

\newcommand{\Diffid}[2]{\mathit{Diff}_{\!\mathit{id}}^{#1}(#2)}

\newcommand{\restrict}[2]{#1\,\rule[-3pt]{0.4pt}{8pt}\vphantom{#1}_{\, #2}}
\newcommand{\reg}{\omega}
\renewcommand{\mid}{\,|\,}
\newcommand{\sym}{\mathscr{L}_{\mathrm{sym}}}
\newcommand{\targ}{{\mathrm{targ}}}
\newcommand{\id}{\mathit{id}}

\theoremstyle{plain}
\newtheorem{definition}{Definition}
\newtheorem{lemma}{Lemma}
\newtheorem{theorem}{Theorem}
\newtheorem{proposition}{Proposition}
\newtheorem{corollary}{Corollary}

\theoremstyle{remark}
\newtheorem{remark}{Remark}
\newtheorem{example}{Example}

\title[Mechanistic Modeling of Longitudinal Shape Changes]{Mechanistic Modeling of Longitudinal Shape Changes: equations of motion and inverse problems}

\author{Dai-Ni~Hsieh$^\dagger$, Sylvain~Arguill\`ere$^\ddagger$, Nicolas~Charon$^\dagger$, and Laurent~Younes$^\dagger$}
\address{$^\dagger$Department of Applied Mathematics and Statistics, Johns Hopkins University}
\address{$^\ddagger$Laboratoire Paul Painlev\'e, Universit\'e de Lille}
\email{dnhsieh@jhu.edu}
\email{sylvain.arguillere@univ-lille.fr}
\email{charon@cis.jhu.edu}
\author{\vspace{-50pt}}
\email{laurent.younes@jhu.edu}
\thanks{LY was partially supported by NIH R01DC016784 and NIH U19AG033655; NC was partially supported by NSF 1912030 and NSF 1945224.}

\usepackage{amsopn}

\calclayout

\graphicspath{{figures/}}

\begin{document}

\maketitle

\begin{abstract}
This paper examines a longitudinal shape evolution model in which a 3D volume progresses through a family of elastic equilibria in response to the time-derivative of an internal force, or yank, with an additional regularization to ensure diffeomorphic transformations. We consider two different models of yank and address the long time existence and uniqueness of solutions for the equations of motion in both models. In addition, we derive sufficient conditions for the existence of an optimal yank that best describes the change from an observed initial volume to an observed volume at a later time. The main motivation for this work is the understanding of processes such as growth and atrophy in anatomical structures, where the yank could be roughly interpreted as a metabolic event triggering morphological changes. We provide preliminary results on simple examples to illustrate, under this model, the retrievability of some attributes of such events.
\end{abstract}

%

\section{Introduction}
\label{sec:introduction}
We analyze in this paper a shape evolution paradigm introduced in \cite{Hsieh2019} in which a volume progresses along a family of regularized elastic equilibria controlled by the gradient of a time-dependent potential, this gradient being interpreted as the time-derivative of an internal force that we will refer to as ``yank'', following, e.g., \cite{Linjeb180414}. A primary motivation of our work is the modeling of shape changes in anatomical structures, where the driving potential may be loosely interpreted as a result of metabolic events, for example, caused by a disease in the structure. Potential applications of this framework include biological growth models \cite{dicarlo2002growth,lubarda2002mechanics,amar2005growth,tallinen2016growth,gerig2006computational} or longitudinal studies in computational anatomy, and in particular, slow changes in the brain resulting from neuro-degenerative diseases \cite{braak1991neuropathological,braak1995staging,aylward2004onset,QIU2009S51,hua2011accurate,lindberg2012hippocampal,singh2013hierarchical,Durrleman2013,adaszewski2013early,younes2014regionally,tang2019regional,younes2019identifying}. Such processes of pathogenesis are not well understood today. Thus we introduce a general framework under which more advanced models can be developed. In our experiments, we make very simple assumptions on the initiation and propagation of the potential. We then illustrate the possibility of inferring the causes of the shape changes only from geometric observations. 

The relationship between shape and yank in our model can be represented as a control system in which the velocity field at a given time is obtained  as the solution of a linear equation that depends on both.
We will provide conditions ensuring that this control system has a unique solution over an arbitrary time interval before formulating and studying the inverse problem of estimating an optimal yank based only on observed initial and final shapes. We will consider two situations in this context.  In the first model,  we will assume that the yank is unspecified at all times. We will then estimate the yank so that it minimizes a cost accumulating over time, resulting in an optimal control problem. In the second one, the assumption will be that the potential specifying the yank is fully characterized by its initial value and follows the shape transformation through basic advection. In this latter case, we will attempt to solve the inverse problem of determining this initial value (specified by a few parameters) based on partial information on the deformation, namely the boundary of the transformed volume. 

The overall paradigm defining the dynamical system is the same as that described in \cite{Hsieh2019}, where we assume that, at time $t$, an infinitesimal force $\delta F(t)$ is applied to a volume $M(t)$ in a zero-stress state, resulting in a new equilibrium at time $t+\delta t$, denoted by $M(t+\delta t)$, where $\delta t$ is small, therefore assuming that times needed to reach new equilibria are negligible compared to the time frame within which the whole process is considered. (Such an assumption of {\em evolving reference configuration} is typical in morphoelastic growth models \cite{ward1997mathematical,dicarlo2002growth,goriely2017mathematics}.) The new configuration $M(t+\delta t)$ is obtained by displacing each point $x$ in $M(t)$ by a small vector $\delta x$, which is obtained by solving a linear equation $\mathcal L(t) \delta x = \delta F(t)$, where $\mathcal L$ typically depends on $M$. Dividing by $\delta t$, introducing the velocity $v = \delta x/\delta t$ and the yank $j = \delta F/\delta t$, we are led to consider shape evolution processes in which $M$ is advected by the vector field $v$ as the solution of  $\mathcal L(t) v = j$. The existence of solutions of such a process is stated in \cref{thm:ctrl_problem,thm:opt_problem} under some assumptions on the operator $\mathcal L$ (which are satisfied, in particular, by properly regularized elastic operators) and on the yank $j$. Existence of solutions to the inverse problem of estimating $j$ from the initial and final shapes are provided in the same theorems.

The paper is organized as follows. Notation and a general description of our framework are provided in \cref{sec:problems}. Our main theorems are stated in \cref{sec:theorems} and proved in \cref{sec:proofs}. \Cref{sec:elasticity} provides specific examples to which our theorems apply. \Cref{sec:experiments} presents experimental results. We conclude with a discussion in \cref{sec:conclusion} and provide implementation details in \cref{sec:implementation}.


\section{Formulation of problems}
\label{sec:problems}

\subsection{Notation}
\label{sec:notation}


For an integer $s\geq 0$, we let $C_0^s(\mathbb{R}^3, \mathbb{R}^3)$ denote the space of $s$-times continuously differentiable vector fields $v$ such that the $k$th derivative $D^k v$ tends to 0 at infinity for every $k \leq s$. The space $C_0^s(\mathbb{R}^3, \mathbb{R}^3)$ is a Banach space equipped with the norm $\|v\|_{s, \infty} = \sum_{k=0}^s \max_{x \,\in\, \mathbb{R}^3}|D^k v(x)|$, where $|\cdot|$ denotes the operator norm of a multilinear map on a product of finite-dimensional vector spaces equipped with the Euclidean norm. If $s = 0$, we will write the customary $\|v\|_\infty$ instead of $\|v\|_{0, \infty}$.

Let $\mathit{id}: \mathbb{R}^3 \rightarrow \mathbb{R}^3$ be the identity map, i.e., $\mathit{id}(x) = x$. We denote by $\Diffid{s}{\mathbb{R}^3}$ the set of $C^s$ diffeomorphisms on $\mR^3$ that tend to identity at infinity. Thus every element $\varphi \in \Diffid{s}{\mathbb{R}^3}$ can be written as $\varphi = \mathit{id} + v$, where $v \in C_0^s(\mathbb{R}^3, \mathbb{R}^3)$. The affine Banach space $\id + C_0^s(\mathbb{R}^3, \mathbb{R}^3)$ is equipped with the induced metric $d(\varphi, \psi) = \|\varphi - \psi\|_{s, \infty}$, which makes $\Diffid{s}{\mR^3} \subset \id + C_0^s(\mathbb{R}^3, \mathbb{R}^3)$ an open subset. 

We will denote by $\mathscr{L}(B, \widetilde B)$ the vector space of bounded linear operators from a Banach space $B$ to another Banach space $\widetilde B$. Weak convergence of sequences $(x_n)$ in $B$ will be denoted by  $x_n \rightharpoonup x$. Denoting the topological dual of $B$ by $B^*$, we will use the notation $(\mu \mid v)$ rather than $\mu(v)$ to denote the evaluation of $\mu \in B^*$ at $v \in B$. {
We say a linear operator $A \in \mathscr{L}(B, B^*)$ is symmetric if the corresponding bilinear form $(v, w) \mapsto (Av \mid w)$ is symmetric. The subspace of symmetric linear operators will be denoted by $\sym(B, B^*)$.}

For a generic function $f: [0, T] \times \mathbb{R}^3 \rightarrow \mathbb{R}^3$, we will use the notation $f(t): \mathbb{R}^3 \rightarrow \mathbb{R}^3$ defined by $f(t)(x) = f(t, x)$. We will use $C$ to denote a generic constant and $C_a$ to show a generic constant depending on $a$. The value of such constants may change from equation to equation.

Throughout this paper, $V$ is a separable Hilbert space of vector fields on $\mR^3$ continuously embedded in $C_0^2(\mathbb{R}^3, \mathbb{R}^3)$, which is denoted by $V \hookrightarrow C_0^2(\mathbb{R}^3, \mathbb{R}^3)$, with inner product $\scp{\cdot}{\cdot}_V$ and norm $\|\cdot\|_V$. Since $V \hookrightarrow C_0^2(\mathbb{R}^3, \mathbb{R}^3)$, there exists a constant $c_V$ such that $\|v\|_{2, \infty} \leq c_V \|v\|_V$. The duality map $L_V: V \to V^*$ is given by
\[
(L_V \hspace{1pt} v \mid w) = \scp{v}{w}_V
\]
and provides an isometry from $V$ onto $V^*$. We denote the inverse of $L_V$ by $K_V \in \mathscr{L}(V^*, V)$, which, because of the embedding assumption, is a kernel operator \cite{aronszajn1950theory}. Note that
\[
	\|v\|_V^2 = (L_V \hspace{1pt} v \mid v) = (K_V^{-1} \hspace{1pt} v \mid v) .
\]
As an example, the space $V$ can be the reproducing kernel Hilbert space (RKHS) associated with a Mat\'ern kernel of some order $r$, and some width $\sigma$, which, in three dimensions, implies that $V$ is a Sobolev space $H^{r+2}$. For the specific value {$r=3$}, which we will use in our experiments, the kernel operator (when applied to a vector measure $\mu\in V^*$) takes the form
\[
(K_V \hspace{1pt} \mu)(x) = \int_{\mR^3} \kappa(|x-y|/\sigma) \, d\mu(y)
\]
with $\kappa(t) = (1+t+2t^2/15 + t^3/15)e^{-t}$.

If $B$ is a Banach space and $p\geq 1$, $L^p([0,T], B)$ denotes the space of Bochner integrable functions $f:[0,T] \to B$ such that $\int_0^T \|f(t)\|_B^p\,dt < \infty$. Recall that a function $f:[0,T] \to B$ is Bochner integrable if: (i) it is the almost-everywhere limit of a sequence of measurable functions that take a finite number of values (also called simple functions) and (ii) satisfies $\int_0^T \|f(t)\|_B\,dt < \infty$.

\subsection{Control  systems and inverse problems}
We now describe the dynamics we consider in this paper, which gradually deform shapes through elastic equilibria. We assume a mapping $A: \Diffid{1}{\mathbb{R}^3} \rightarrow \sym(V, V^*)$ defined by $\varphi \mapsto A_\varphi$. Given a time-dependent mapping $j: [0, T] \rightarrow V^*$, we model the deformation trajectory of a compact subset $M_0 \subset \mathbb{R}^3$ as $t \mapsto \varphi_j(t, M_0)$, where $\varphi_j \in C([0,T], \Diffid{1}{\mR^3})$ is a solution to the system
\begin{equation}
	\label{eq:ctrl_problem_system}
	\left\{
		\begin{array}{l}
			\partial_t \hspace{1pt} \varphi(t, x) =  v(t, \varphi(t, x)), \ \varphi(0, x) = x ,
			\\[5pt]
			\displaystyle
			v(t) = \underset{v' \,\in\, V}{\arg\min} \ \frac{\reg}{2} \, \|v'\|_V^2 + \frac{1}{2} \, (A_{\varphi(t)} \hspace{1pt} v' \mid v') - (j(t) \mid v')
		\end{array}
	\right.
\end{equation}
and $\omega > 0$ is a fixed regularization parameter. The first equation in this system will be seen as an ordinary differential equation in $\Diffid{1}{\mR^3}$. We can interpret the squared norm $\frac{\reg}{2} \, \|v'\|_V^2$ as a regularization term that is introduced to ensure that $v(t)$ and (as we will see in our results) $\varphi(t)- \id$ are both in $C^2_0(\mR^2, \mR^2)$. This term will also ensure that $\varphi(t)$ is a diffeomorphism at all times (similar regularizations were used  in works such as \cite{beg2005computing,younes2010shapes,younes2018hybrid}). The operator $A_{\varphi(t)}$, as we shall detail later, may be for instance an elastic operator in which case the second term $\frac{1}{2} \, (A_{\varphi(t)} \hspace{1pt} v' \mid v')$ represents the linear elastic energy associated to the deformation while $j(t)$ represents a yank inducing the motion of the material. In this context, the second equation in system~\cref{eq:ctrl_problem_system} essentially states that the deformation vector field at each time is governed by an infinitesimal version of the principle of virtual work \cite[Theorem 1.6, Chapter 5]{marsden1994mathematical} with regularization. As a result, the shape $\varphi_j(t, M_0)$ is deformed from a stress-free state to an equilibrium at all time in this dynamical system, as described earlier in the introduction. We postpone specific examples of elastic operators and yank until \cref{sec:elasticity}, after presenting sufficient conditions ensuring existence of solutions of our inverse problems in \cref{sec:theorems}, where we treat $A_\varphi$ and $j$ as general operators.

We let $\mathscr M$ denote a class of compact subsets of $\mR^3$ that represents our ``shape space'' and assume that it is stable by the action of diffeomorphisms, i.e., $\varphi(\mathscr{M}) \subset \mathscr{M}$ for all $\varphi \in \Diffid{1}{\mathbb{R}^3}$. A specific description of $\mathscr{M}$ is problem dependent (see \cref{rem:varifold}). Given two elements  $M_0, M_\targ \in \mathscr M$, providing the observed initial shape and final shape, or target, we aim to find $j$ within a given class such that the deformed $M_0$ in response to $j$ at time $T$, i.e., $\varphi_j(T, M_0)$, is close to $M_\targ$ in some sense. Closeness will be measured according to a discrepancy function $\rho: \mathscr M \times \mathscr M \to [0, +\infty)$ that compares compact sets (see examples in  \cref{rem:varifold}). We will focus on the following two frameworks regarding the time-dependent yank $j$: 
\begin{enumerate}[label=(\arabic*), leftmargin=0.75cm]
\item Free yank model. In system~\cref{eq:ctrl_problem_system}, one can interpret  $j$ as a control that drives the evolution of the state $\varphi$ through the vector field $v_\varphi$. Let $\mathcal{X}_{V^*\!, \, T}^p = L^p([0, T], V^*)$. We will consider the optimal control problem
\begin{equation}
    \label{eq:free.yank}
	\min_{j \,\in\, \mathcal{X}_{V^*\!, \, T}^2} \, \int_0^T (j(t) \mid v(t)) \, dt + \rho(\varphi(T, M_0), M_\targ)
\end{equation}
subject to system~\cref{eq:ctrl_problem_system}. 
We will give sufficient conditions guaranteeing the existence of solutions of this problem in \cref{thm:ctrl_problem}.

\item Parametric yank model. The yank is modeled as a function of a transformation $\varphi$ and of a finite-dimensional parameter $\theta$ belonging to a compact set $\Theta\subset \mR^m$. In this case, the finite-dimensional optimization problem of interest is
\begin{equation}
    \label{eq:parametrized.yank}
	\min_{\vphantom{\rule{0pt}{6.5pt}} \theta \,\in\, \Theta} \  \rho(\varphi(T, M_0), M_\targ)
\end{equation}
subject to \cref{eq:ctrl_problem_system} with $j(t) = j({\varphi(t)}, \theta)$, namely,
\begin{equation}
	\label{eq:opt_problem_system}
	\left\{
		\begin{array}{ll}
			\partial_t \hspace{1pt} \varphi(t, x) =  v(t, \varphi(t, x)), \  \varphi(0, x) = x,
			\\[5pt]
			\displaystyle
			v(t) = \underset{v' \,\in\, V}{\arg\min} \  \frac{\reg}{2} \, \|v'\|_V^2 + \frac{1}{2} \, (A_{\varphi(t)} \hspace{1pt} v' \mid v') - 
			(j({\varphi(t)}, \theta) \mid v').
		\end{array}
	\right. 
\end{equation}
Examples of such yanks are provided in \cref{sec:elasticity}. We  give sufficient conditions for this optimization problem to have a solution in \cref{thm:opt_problem}.
\end{enumerate}


\section{Main results}
\label{sec:theorems}


Given a compact subset $\varOmega \subset \mathbb{R}^3$, we define the seminorm 
\[
	\|v\|_{s, \infty}^\varOmega
	=
	\sum_{k=1}^s 
	\max_{x \,\in\, \varOmega} \, |D^k v(x)|
\]
on $C^s(\mathbb{R}^3, \mathbb{R}^3)$.
We require a regularity assumption on the discrepancy function $\rho$ appearing in the objective functionals \cref{eq:free.yank,eq:parametrized.yank}.

\begin{definition}
\label{def:continuous.rho}
We say that a discrepancy function $\rho: \mathscr M \times \mathscr M \to [0, +\infty)$ is continuous on $\mathscr M$ with respect to $\|\cdot\|_{s, \infty}$ if for all compact sets $M, M' \in \mathscr M$ and all sequences $(\varphi_n)_{n = 1}^\infty \subset \Diffid{s}{\mR^3}$ such that $\|\varphi_n - \varphi\|_{s, \infty}^M \rightarrow 0$ for some $\varphi \in \Diffid{s}{\mR^3}$, one has
\[
\rho(\varphi_n(M), M') \rightarrow \rho(\varphi(M), M').
\]
\end{definition}

\begin{theorem}[Free yank model]
\label{thm:ctrl_problem}
Let $A: \Diffid{1}{\mathbb{R}^3} \rightarrow \sym(V, V^*)$ be a mapping defined by $\varphi \mapsto A_\varphi$. Assume that $(A_\varphi \hspace{1pt} v \mid v) \geq 0$ for all $\varphi \in \Diffid{1}{\mathbb{R}^3}$ and $v \in V$. Let the two compact sets $M_0, M_\targ \in \mathscr M$ be given. Then the following results hold.
\begin{enumerate}[label = (\roman*), ref = \ref{thm:ctrl_problem}(\roman*), leftmargin=0.75cm]
\item
\label{thm:ctrl_problem_ode}
Suppose that $\varphi \mapsto A_\varphi$ is locally Lipschitz (for the $\|\cdot\|_{1,\infty}$ distance). Then, given $j \in \mathcal{X}_{V^*\!, \, T}^1$, the system
\begin{equation}
	\tag{\ref{eq:ctrl_problem_system}}
	\left\{
		\begin{array}{l}
			\partial_t \hspace{1pt} \varphi(t, x)
			= 
			v(t, \varphi(t, x)) , \ 
			\varphi(0, x) = x ,
			\\[5pt]
			\displaystyle
			v(t)
			=
			\underset{v' \,\in\, V}{\arg\min} \ 
			\frac{\reg}{2} \, \|v'\|_V^2
			+
			\frac{1}{2} \, (A_{\varphi(t)} \hspace{1pt} v' \mid v')
			-
			(j(t) \mid v')
		\end{array}
	\right. .
\end{equation}
has a unique solution $\varphi \in C([0, T], \Diffid{2}{\mathbb{R}^3})$.


\vspace{3pt}

\item
\label{thm:ctrl_problem_min}

Suppose that, for each $\gamma>0$, the mapping $\varphi \mapsto A_\varphi$ is Lipschitz with respect to the seminorm $\|\cdot\|_{1, \infty}^{M_0}$ on
\[
	\mathfrak{S}_\gamma
	=
	\{
		\varphi \in \Diffid{1}{\mathbb{R}^3} : 
		\|\varphi - id\|_{1, \infty} \leq \gamma
		\mbox{ and }
		\|\varphi^{-1} - id\|_{1, \infty} \leq \gamma
	\}.
\]
In addition, assume that the discrepancy function $\rho$ is continuous on $\mathscr M$ with respect to $\|\cdot\|_{1,\infty}$. Then there exists a minimizer of the optimal control problem
\[
	\min_{j \,\in\, \mathcal{X}_{V^*\!, \, T}^2} \, \int_0^T (j(t) \mid v(t)) \, dt + \rho(\varphi(T, M_0), M_\targ)
\]
where $v$ and $\varphi$ satisfy \cref{eq:ctrl_problem_system}.

\end{enumerate}
\end{theorem}

\medskip

Before stating our next theorem, we remind the reader that a collection of functions is said to be equi-Lipschitz if they are all Lipschitz and there exists a common Lipschitz constant that applies to all functions in the collection.

\medskip

\begin{theorem}[Parametric yank model]
\label{thm:opt_problem}
Let $A: \Diffid{1}{\mathbb{R}^3} \rightarrow \sym(V, V^*)$ be a mapping defined by $\varphi \mapsto A_\varphi$. Assume that $(A_\varphi \hspace{1pt} v \mid v) \geq 0$ for all $\varphi \in \Diffid{1}{\mathbb{R}^3}$ and $v \in V$. Moreover, let $\Theta \subset \mathbb{R}^m$ be a compact set and let $j: \Diffid{1}{\mathbb{R}^3} \times \Theta \rightarrow V^*$. Finally, let two compact sets $M_0, M_\targ \in \mathscr M$ be given. Then the following results hold.
\begin{enumerate}[label = (\roman*), ref = \ref{thm:opt_problem}(\roman*), leftmargin=0.75cm]
\item
\label{thm:opt_problem_ode}
Suppose that $\varphi \mapsto A_\varphi$ is locally Lipschitz and that $\varphi \mapsto j(\varphi, \theta)$ is locally Lipschitz (in both cases for $\|\cdot\|_{1,\infty}$ distance) and bounded in norm. Given $\theta \in \Theta$, the system
\begin{equation}
	\tag{\ref{eq:opt_problem_system}}
	\left\{
		\begin{array}{ll}
			\partial_t \hspace{1pt} \varphi(t, x)
			= 
			v(t, \varphi(t, x)) , \ 
			\varphi(0, x) = x ,
			\\[5pt]
			\displaystyle
			v(t)
			=
			\underset{v' \,\in\, V}{\arg\min} \ 
			\frac{\reg}{2} \, \|v'\|_V^2
			+
			\frac{1}{2} \, (A_{\varphi(t)} \hspace{1pt} v' \mid v')
			-
			(j({\varphi(t)}, \theta) \mid v')
		\end{array}
	\right. .
\end{equation}
has a unique solution $\varphi \in C([0, T], \Diffid{2}{\mathbb{R}^3})$.

\item
\label{thm:opt_problem_min}
Suppose that:
\begin{enumerate}[label=$\bullet$,leftmargin=0.25cm]
\item For each $\gamma > 0$, the mapping $\varphi \mapsto A_\varphi$ is Lipschitz and  the family of mappings $\{j(\cdot, \theta): \theta \in \Theta\}$ is equi-Lipschitz, both with respect to the seminorm $\|\cdot\|_{1, \infty}^{M_0}$, on the set
\[
	\mathfrak{S}_\gamma
	=
	\{
		\varphi \in \Diffid{1}{\mathbb{R}^3} : 
		\|\varphi - id\|_{1, \infty} \leq \gamma
		\mbox{ and }
		\|\varphi^{-1} - id\|_{1, \infty} \leq \gamma
	\}.
\]
\item For all $\varphi \in \Diffid{1}{\mathbb{R}^3}$,  $j(\varphi, \cdot)$ is continuous in the sense that
\[
	\theta_n \rightarrow \theta
	\ \ \mbox{ implies } \ \ 
	j(\varphi, \theta_n) \rightharpoonup j(\varphi, \theta).
\]
\item There exists a constant $J_\Theta$ such that 
\[
	\|j(\varphi, \theta)\|_{V^*} \leq J_\Theta
	\ \ \mbox{ for all }
	\varphi \in \Diffid{1}{\mathbb{R}^3} \mbox{ and } \theta \in \Theta .
\]
\item The discrepancy function $\rho$ is continuous on $\mathscr M$  with respect to $\|\cdot\|_{1,\infty}$. 
\end{enumerate}
\medskip
Then there exists a minimizer for the finite-dimensional optimization problem
\[
	\min_{\vphantom{\rule{0pt}{6.5pt}} \theta \,\in\, \Theta} \ \rho(\varphi(T, M_0), M_\targ)
\]
where $\varphi$ satisfies \cref{eq:opt_problem_system}.
\vspace{3pt}


\end{enumerate}
\end{theorem}
We will prove these two theorems in \cref{sec:proofs}.
Both control systems can be considered as ordinary differential equations with values in the affine Banach space $\mathit{id} + C_0^1(\mR^d, \mR^d)$. Most of the effort in proving the long-time existence and uniqueness of solutions, part (i) of the theorems, will rely on controlling the vector field $v(t)$ by the input, $j$ in the first case and $\theta$ in the second case. A key step for this is provided by \hyperref[lemma:inequalities_3]{Lemma}~\ref{lemma:inequalities_3}, which exploits the regularization term $\frac{\omega}{2} \, \|v'\|_V^2$. After this key step, we can carry out a proof using the Banach fixed point theorem and the Picard iteration in the space $C([0, T], \Diffid{1}{\mathbb{R}^3})$. Due to the regularity of the vector field, \cref{lemma:bounds} will show that the obtained unique solution is actually in $C([0, T], \Diffid{2}{\mathbb{R}^3})$. For part (ii) of the theorems, we will use the direct method of calculus of variations to prove the existence of minimizers.
Denoting the objective function by $f$, the two key steps of the direct method are proving that a minimizing sequence is bounded and that $f$ is weakly sequentially lower semicontinuous. The choice of the objective function affects the first step (see \cref{remark:cost_j_v}). To accomplish the second step, we need to show $f(j) \leq \liminf_{n \rightarrow \infty} f(j_n)$ for any sequence $j_n \rightharpoonup j$. We prove this by going through an intermediate step:
\[
	j_n \rightharpoonup j \ \Rightarrow \ \|\varphi_{j_n}\!(t) - \varphi_j(t)\|_{1, \infty}^{M_0} \rightarrow 0 \ \Rightarrow \ f(j) \leq \liminf_{n \rightarrow \infty} f(j_n) ,
\]
where $\varphi_{j_n}$ and $\varphi_j$ are solutions given $j_n$ and $j$ respectively. The seminorm $\|\cdot\|_{1, \infty}^{M_0}$ is adopted to resolve a difficulty  in this intermediate step: although $j_n \rightharpoonup j$ implies $\varphi_{j_n}\!(t, x) \rightarrow \varphi_j(t, x)$ pointwise, we do note have the uniform convergence. Fortunately, the convergence in the seminorm given by the Arzel{\`a}--Ascoli theorem suffices.

\vspace{5pt}
\begin{remark}
We stated our theorems in dimension three because it corresponds to most interesting situations in practice, but our proofs apply without change to any dimension (and we are providing some experimental illustrations in dimension two).
\end{remark}
\begin{remark}
\label{remark:cost_j_v}
The choice we made for the control cost $(j \mid v)$ in \cref{thm:ctrl_problem} is one among a large spectrum of costs for which the conclusions of the theorem are valid. We took this specific example for simplicity and also because it provided the best results in our experiments among some other options we tried. Other possible examples could be $\|j\|^2_{V^*}$, or $\|j\|^2_{L^2}$, for which our proofs can easily be modified (actually, simplified), with details being left to the reader.
\end{remark}
\begin{remark}
\label{rem:varifold}
In the experiments presented in this paper, we will use discrepancy functions based on the varifold pseudo-metrics introduced in \cite{Charon2013}
between certain surfaces associated with the two volumes (e.g., their boundaries). In this case, $\mathscr M$ is the set of all compact subsets $M \subset \mathbb{R}^3$ whose boundary $\partial M$ is a rectifiable surface (we refer to \cite{Simon1983} for the precise definition and properties of rectifiable sets). Then, given $M$ and $M'$ in $\mathscr M$, for $S$ and $S'$ two rectifiable surfaces extracted from $M$ and $M'$  (such as for instance the boundaries of the volumes or some corresponding internal layers) the discrepancy function takes the following form:
\[
\rho(M,M') = \nu(S,S) - 2\nu(S,S') + \nu(S',S')
\]
with
\[
\nu(S,S') = \int_S\int_{S'} \chi\left(\frac{|x-x'|}{\tau}\right) (n(x)^\top n'(x'))^2 \,d\sigma(x)\, d\sigma'(x')
\]
where $\sigma$ and $\sigma'$ are volume measures on $S$ and $S'$, $n$ and $n'$ are unit normal vector fields of $S$ and $S'$ and $\chi$ is some radial kernel function, scaled by a scalar $\tau>0$, which in our experiments is taken to be the Cauchy kernel
\[
\chi(t) = (1 + t^2)^{-2}.
\]
It can be then shown, cf., \cite[Proposition~6]{Charlier2017}, that such discrepancy functions are continuous on $\mathscr M$ with respect to $\|\cdot\|_{1,\infty}$, in the sense of  \cref{def:continuous.rho}. 

As a side note, one could alternatively select $\rho$ as the volume of the symmetric difference between the two sets, i.e., $\rho(M, M')=\text{vol}(M \bigtriangleup M')$ which is continuous on compact sets with respect to $\|\cdot\|_{0,\infty}$ and thus also with respect to $\|\cdot\|_{1,\infty}$, thereby satisfying the assumption of the above theorems. In this case, the shape space $\mathscr{M}$ is composed of all compact subsets of $\mathbb{R}^3$. 
\end{remark}

\begin{remark}
A model based on principles that are similar to ours has been introduced and studied in \cite{bressan2018model} in order to model tissue growth. In \cite{bressan2018model}, the second equation in system \cref{eq:ctrl_problem_system} is replaced by the minimization of an isotropic linear elastic energy under the constraint that $\mathrm{div}\, v = g(u)$, where $g$ is a fixed function and $u$ is the concentration of morphogen, a growth-induced chemical produced by cells. The concentration of morphogen is controlled by the density of cells via a linear elliptic equation, and the density of cells is advected under the motion.
While we consider here more general elasticity models, inducing additional complications in the analysis because these elastic properties must also be advected along the flow, the main difference between the two models results from the regularization $\frac{\omega}{2} \, \|v'\|_V^2$ added in \cref{eq:ctrl_problem_system}. It is this term that guarantees the regularity of the vector field $v$ (see \cref{lemma:minimizer}) and allows us to prove long-time existence of solutions, while \cite{bressan2018model} only obtains local existence results. Such a property is necessary to formulate optimal control problems, which form the second focus of our paper, and could not be addressed in \cite{bressan2018model}.  
\end{remark}


\section{Examples of elastic operators and yank}
\label{sec:elasticity}

In this section, we provide examples of elastic operators and yank that satisfy the conditions in \cref{thm:ctrl_problem,thm:opt_problem}. Denote the space of symmetric bilinear forms on the space of 3-by-3 symmetric matrices by $\Sigma_2(\mathrm{Sym}_3(\mathbb{R}), \mathrm{Sym}_3(\mathbb{R}))$. Given $\varphi \in \Diffid{1}{\mathbb{R}^3}$, an elastic operator $A_\varphi \in \sym(V, V^*)$ takes the following form
\begin{equation}
	\label{eq:elastic_operator}
	(A_\varphi \hspace{1pt} u \mid v)
	=
	\int_{\varphi(M_0)} E_\varphi(\varepsilon_u, \varepsilon_v) \, dx
	=
	\int_{\varphi(M_0)} (E_\varphi(x))(\varepsilon_u(x), \varepsilon_v(x)) \, dx ,
\end{equation}
where $E_\varphi: \varphi(M_0) \rightarrow \Sigma_2(\mathrm{Sym}_3(\mathbb{R}), \mathrm{Sym}_3(\mathbb{R}))$ is a stiffness tensor after the shape is deformed by $\varphi$, and $\varepsilon_u$ and $\varepsilon_v$ are linear strain tensors defined by
\[
	\varepsilon_u = \frac{1}{2} \left( Du + Du^\top \right)
	\ \ \mbox{ and } \ \ 
	\varepsilon_v = \frac{1}{2} \left( Dv + Dv^\top \right) .
\]
We recall the basic assumption in our model that the deformed configuration becomes a new reference for the next infinitesimal shape changes (with material properties advected by the flow). This is reflected in equation \cref{eq:elastic_operator}.

A simple example, assuming that the elastic property of an isotropic elastic material is unaffected by deformation, or persistent, is provided by $E_\varphi(\varepsilon_u, \varepsilon_v) = \lambda \, \mathrm{tr}(\varepsilon_u) \, \mathrm{tr}(\varepsilon_v) + 2\mu \, \mathrm{tr}(\varepsilon_u^\top \varepsilon_v)$, where $\lambda$ and $\mu$ are the Lam\'e parameters. More generally, the following proposition proved in \cref{sec:proofs} provides a sufficient condition on the mapping $\varphi \mapsto E_\varphi$ ensuring that the corresponding $A_\varphi$ satisfies the conditions of \cref{thm:ctrl_problem,thm:opt_problem}.


\begin{proposition}
\label{prop:elastic_operator}
Suppose that $E_\varphi(x)$ is positive definite for all $\varphi \in \Diffid{1}{\mathbb{R}^3}$ and $x \in \varphi(M_0)$. Moreover, suppose that, for each $\gamma > 0$, there exists $\alpha_\gamma > 0$ such that
\begin{equation}
\label{eq:prop_elastic_operator}
    	\int_{M_0} \, |E_\varphi \circ \varphi - E_\psi \circ \psi| \, dx
	\leq
	\alpha_\gamma \, \|\varphi - \psi\|_{1, \infty}^{M_0}
	\ \ \mbox{ for all }
	\varphi, \psi \in \mathfrak{S}_\gamma ,
\end{equation}
where
\[
	\mathfrak{S}_\gamma
	=
	\{
		\varphi \in \Diffid{1}{\mathbb{R}^3} : 
		\|\varphi - id\|_{1, \infty} \leq \gamma
		\mbox{ and }
		\|\varphi^{-1} - id\|_{1, \infty} \leq \gamma .
	\}
\]
Then, for $A_\varphi$ defined as in \cref{eq:elastic_operator}, the mapping $\varphi \mapsto A_\varphi$ satisfies the conditions of \cref{thm:ctrl_problem,thm:opt_problem}.
\end{proposition}


\begin{example}
\label{ex:elastic_operator_constant}
According to \cref{prop:elastic_operator}, the simplest example of the elastic operator is when the stiffness tensor $E_\varphi$ is constant and positive definite since the left-hand side of \cref{eq:prop_elastic_operator} is zero. Thus our example of persistent isotropic elastic material, i.e., $E_\varphi(\varepsilon_u,\varepsilon_v)=\lambda \, \mathrm{tr}(\varepsilon_u) \, \mathrm{tr}(\varepsilon_v) + 2\mu \, \mathrm{tr}(\varepsilon_u^\top \varepsilon_v)$, is a valid choice. More generally, suppose that $\varLambda: M_0 \rightarrow \Sigma_2(\mathrm{Sym}_3(\mathbb{R}), \mathrm{Sym}_3(\mathbb{R}))$ and that $\varLambda(x)$ is positive definite for all $x \in M_0$, then $E_\varphi \colonequals \varLambda \circ \varphi^{-1}$ also satisfies the conditions in \cref{prop:elastic_operator}. Note that this form of $E_\varphi$ preserves the elastic properties of the material from $x$ to $\varphi(x)$.
\end{example}


\begin{example}
Even more generally, let $F_\varphi: \varphi(M_0) \rightarrow GL(3, \mathbb{R})$ be a deformation-dependent frame field, where $GL(3, \mathbb{R})$ denotes the general linear group. We consider stiffness tensors of the form $\widetilde E_\varphi(\varepsilon_u, \varepsilon_v) \colonequals (\varLambda \circ \varphi^{-1})(F_\varphi^\top \varepsilon_u F_\varphi, F_\varphi^\top \varepsilon_v F_\varphi)$, where $\varLambda$ is the same as in \cref{ex:elastic_operator_constant}.

The following proposition, whose proof is elementary and left to the reader, provides sufficient conditions on $F$ ensuring that $\widetilde E$ satisfies the conditions in \cref{prop:elastic_operator}. Note that in this case, the elastic properties at $\varphi(x)$ are modified from the ones at $x$ through a change of the frame coordinates $F_\varphi(\varphi(x))$.
\begin{proposition}
\label{prop:elasticity_frame_change}
Suppose that $E$ satisfies the conditions in \cref{prop:elastic_operator}. Let $F_\varphi: \mathbb{R}^3 \rightarrow GL(3, \mathbb{R})$ be essentially bounded for each $\varphi \in \Diffid{1}{\mathbb{R}^3}$. If there exists $\beta_\gamma > 0$ such that
\[
	\|F_\varphi \circ \varphi - F_\psi \circ \psi\|_\infty
	\leq
	\beta_\gamma \, \|\varphi - \psi\|_{1, \infty}^{M_0}
	\ \ \mbox{ for all }
	\varphi, \psi \in \mathfrak{S}_\gamma ,
\]
then $\widetilde E$ defined by $\widetilde E_\varphi(\varepsilon_u, \varepsilon_v) = E_\varphi(F_\varphi^\top \varepsilon_u F_\varphi, F_\varphi^\top \varepsilon_v F_\varphi)$ also satisfies the conditions in \cref{prop:elastic_operator}. 
\end{proposition}

Let $w_1, w_2$, and $w_3$ be linearly independent vector fields on $M_0$. Examples of frame fields $F_\varphi$ that satisfy the previous assumptions include
\[
	F_\varphi
	=
	\left[
		\begin{array}{ccc}
			\displaystyle \frac{D\varphi \, w_1}{|D\varphi \, w_1|}, & 
			\displaystyle \frac{D\varphi \, w_2}{|D\varphi \, w_2|}, & 
			\displaystyle \frac{D\varphi \, w_3}{|D\varphi \, w_3|}
		\end{array}
	\right] \circ \varphi^{-1}
\]
and
\begin{align}
	\label{eq:frame_field}
	F_\varphi
	=
	\left[
		\begin{array}{ccc}
			\displaystyle \frac{D\varphi \, w_1}{|D\varphi \, w_1|}, &
			\displaystyle \frac{( D\varphi \, w_1 \times D\varphi \, w_2) \times D\varphi \, w_1}
			                   {|( D\varphi \, w_1 \times D\varphi \, w_2) \times D\varphi \, w_1|}, &
			\displaystyle \frac{D\varphi \, w_3}{|D\varphi \, w_3|}
		\end{array}
	\right] \circ \varphi^{-1} .
\end{align}
Note that the first two vectors of the latter $F_\varphi$ are orthonormal for all deformation $\varphi$.
\end{example}

\begin{example}
\label{ex:layers}
An elastic operator inspired by the laminar organization of the cerebral cortex using the frame field \cref{eq:frame_field} was introduced in \cite{Hsieh2019}; we describe it here for completeness. Suppose that a compact subset $M_0 \subset \mathbb{R}^3$ has two surfaces $\mathcal{M}_{\mathrm{bottom}}$ and $\mathcal{M}_{\mathrm{top}}$ as bottom and top layers. Moreover, suppose that we are given a diffeomorphism $\varPhi: [0, 1] \times \mathcal{M_{\mathrm{bottom}}} \rightarrow M_0$ such that $\varPhi(0, \mathcal{M_{\mathrm{bottom}}}) = \mathcal{M_{\mathrm{bottom}}}$ and $\varPhi(1, \mathcal{M_{\mathrm{bottom}}}) = \mathcal{M_{\mathrm{top}}}$. Note that $\varPhi(\nu, \mathcal{M_{\mathrm{bottom}}}) \equalscolon \mathcal{M}_\nu$ is a surface for each $\nu \in [0, 1]$. We refer to $\varPhi$ as a layered structure of $M_0$. Such a structure induces a transversal vector field $S \colonequals \partial_\nu \varPhi$ (\cref{fig:layered_structure}). Let $T_1$ and $T_2$ be linearly independent vector fields on $M_0$ such that $\restrict{T_1}{\mathcal{M}_\nu}$ and $\restrict{T_2}{\mathcal{M}_\nu}$ are tangent to $\mathcal{M}_\nu$. Then, $T_1$, $T_2$, and $S$ are linearly independent vector fields on $M_0$. If we let $w_1 = T_1$, $w_2 = T_2$, and $w_3 = S$ in \cref{eq:frame_field} and define
\begin{align}
	\label{eq:layered_stiffness}
	\begin{split}
	\bar \varLambda(\varepsilon, \varepsilon)
	&=
	\lambda_{\mathrm{tan}}(\varepsilon_{11} + \varepsilon_{22})^2
	+
	\mu_{\mathrm{tan}}(\varepsilon_{11}^2 + \varepsilon_{22}^2 + 2 \varepsilon_{12}^2)
	\\
	&\hspace{13pt}
	\phantom{}
	+
	\mu_{\mathrm{tsv}} \, \varepsilon_{33}^2
	+
	\mu_{\mathrm{ang}}(2 \varepsilon_{13}^2 + 2 \varepsilon_{23}^2) ,
	\end{split}
\end{align}
where $\varepsilon_{ij}$ denotes the $ij$th element of $\varepsilon \in \mathrm{Sym}_3(\mathbb{R})$ and $\lambda_{\mathrm{tan}}$, $\mu_{\mathrm{tan}}$, $\mu_{\mathrm{tsv}}$, and $\mu_\mathrm{ang}$ are constants, then the corresponding elastic operator
\begin{align}
	\label{eq:layered_elastic_operator}
	(A_\varphi \hspace{1pt} u \mid v)
	=
	\int_{\varphi(M_0)} \bar \varLambda(F_\varphi^\top \varepsilon_u F_\varphi, F_\varphi^\top \varepsilon_v F_\varphi) \, dx
\end{align}
is well-defined \cite{Hsieh2019} and a valid choice by \cref{ex:elastic_operator_constant,prop:elasticity_frame_change}. Note that the layered structure on a deformed shape $\varphi(M_0)$ becomes $(\nu, x) \mapsto \varphi \circ \varPhi(\nu, \varphi^{-1}(x))$. The elastic material corresponding to this elastic operator has the property that it is isotropic along the directions tangent to the layers. \Cref{fig:layered_tangential_deformation,fig:layered_transversal_deformation} illustrate deformations $\varphi_j(T, M_0)$ according to system \cref{eq:ctrl_problem_system} under different elastic parameters $\mu_{\mathrm{tan}}$ and $\mu_{\mathrm{tsv}}$ when we apply the same yank $j$ to a layered shape $M_0$ (\cref{fig:layered_shape}).

\begin{figure}[hbt]
	\centering
	\includegraphics[trim = 30 50 30 20, clip, width = 0.3\textwidth]{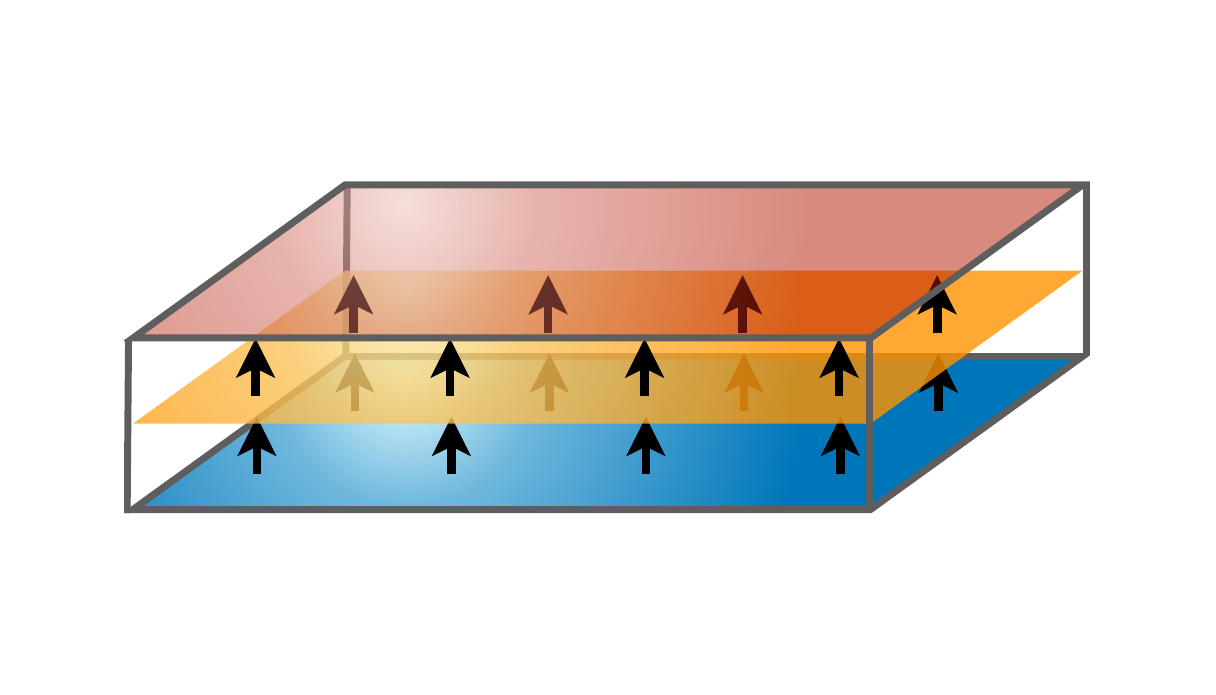}
	\hfill
	\includegraphics[trim = 30 50 30 20, clip, width = 0.3\textwidth]{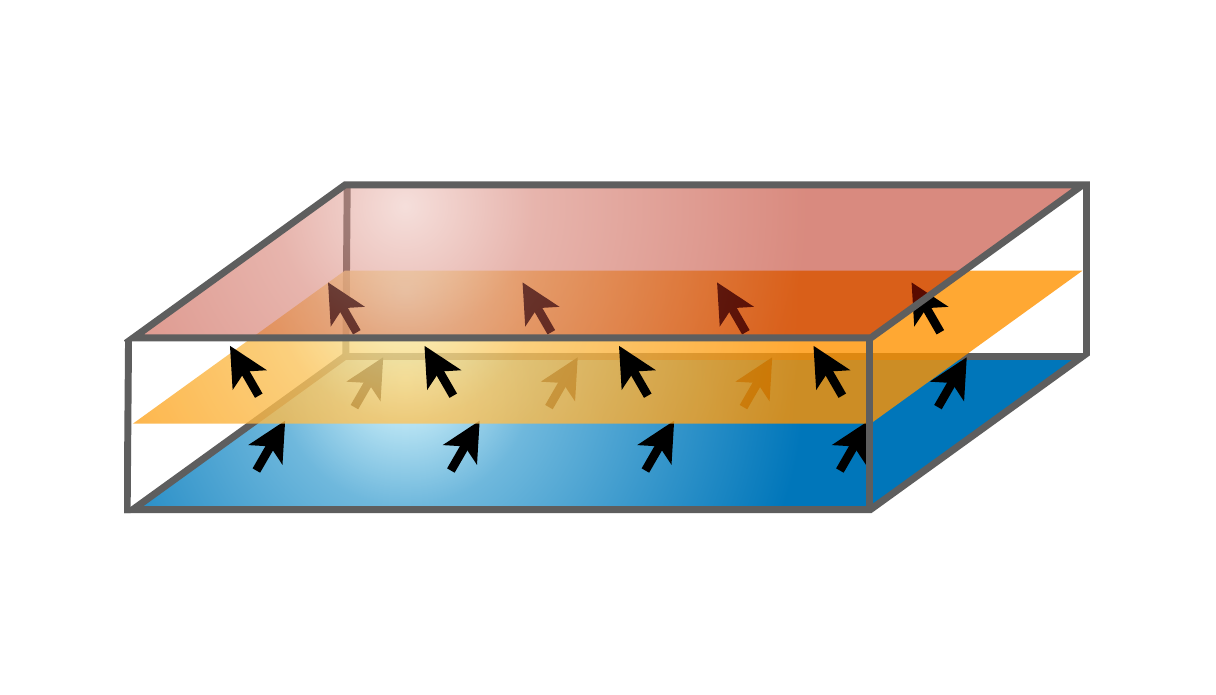}
	\hfill
	\includegraphics[trim = 30 50 30 20, clip, width = 0.3\textwidth]{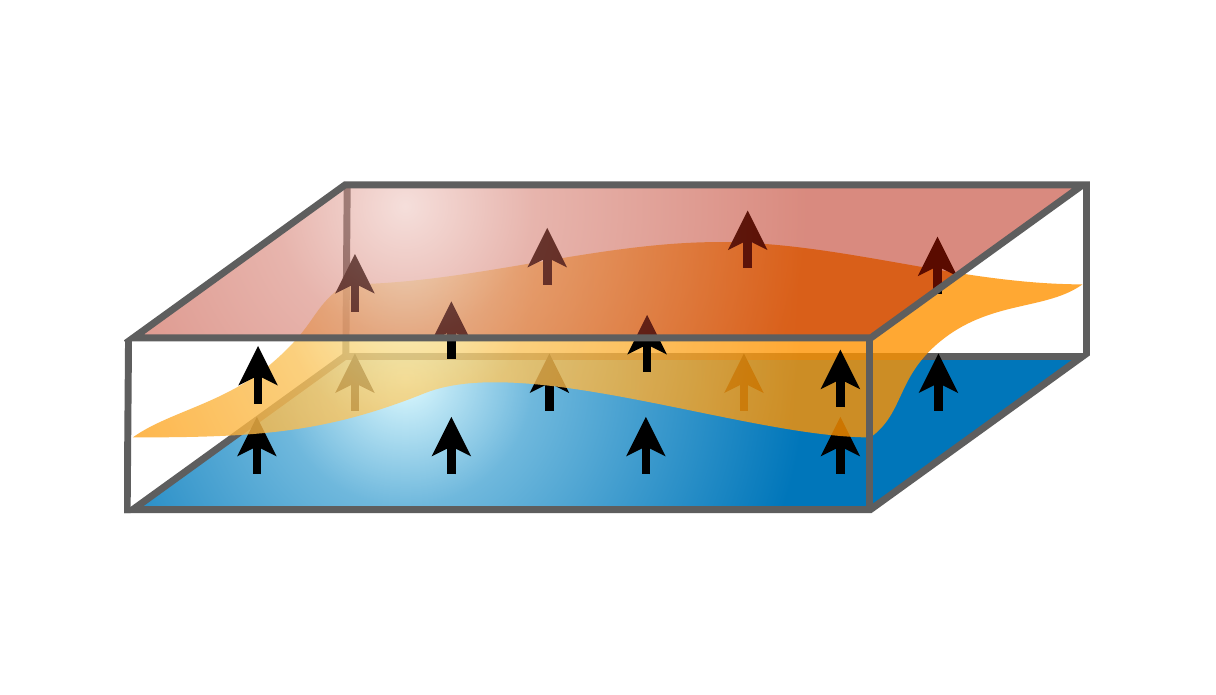}
	\caption{Different layered structures of the same rectangular region. Shown in the figures are top layer, one middle layer, bottom layer, and the transversal vector field.}
	\label{fig:layered_structure}
\end{figure}

\begin{figure}[hbt]
	\centering
	\begin{subfigure}[t]{0.48\textwidth}
		\centering
		\includegraphics[trim = 70 100 50 90, clip, width = 0.95\textwidth]{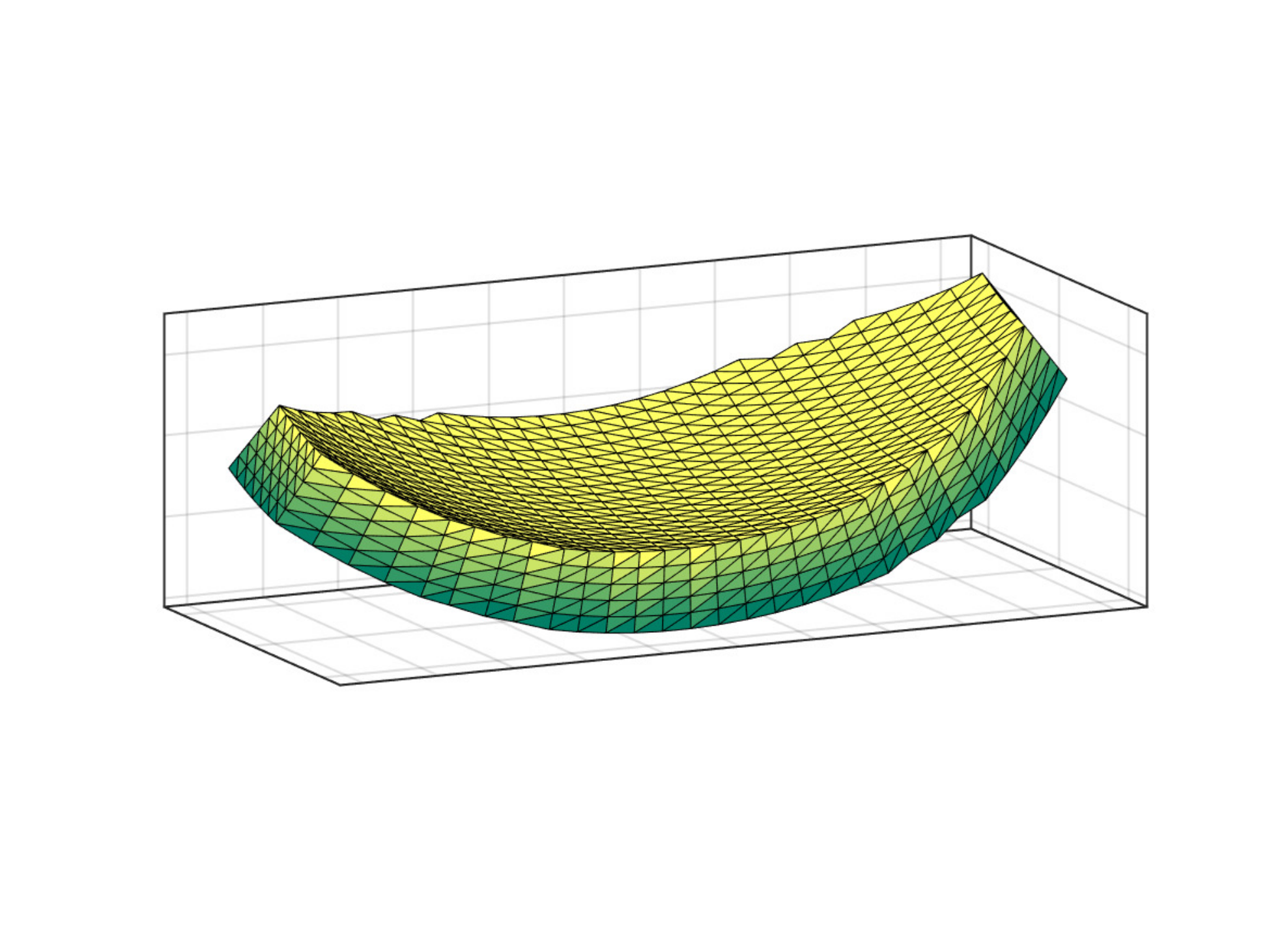}
		\caption{A simulated layered shape.}
		\label{fig:layered_shape}
	\end{subfigure}
	\\[20pt]
	\begin{subfigure}[t]{0.48\textwidth}
		\centering
		\includegraphics[trim = 70 100 50 90, clip, width = 0.95\textwidth]{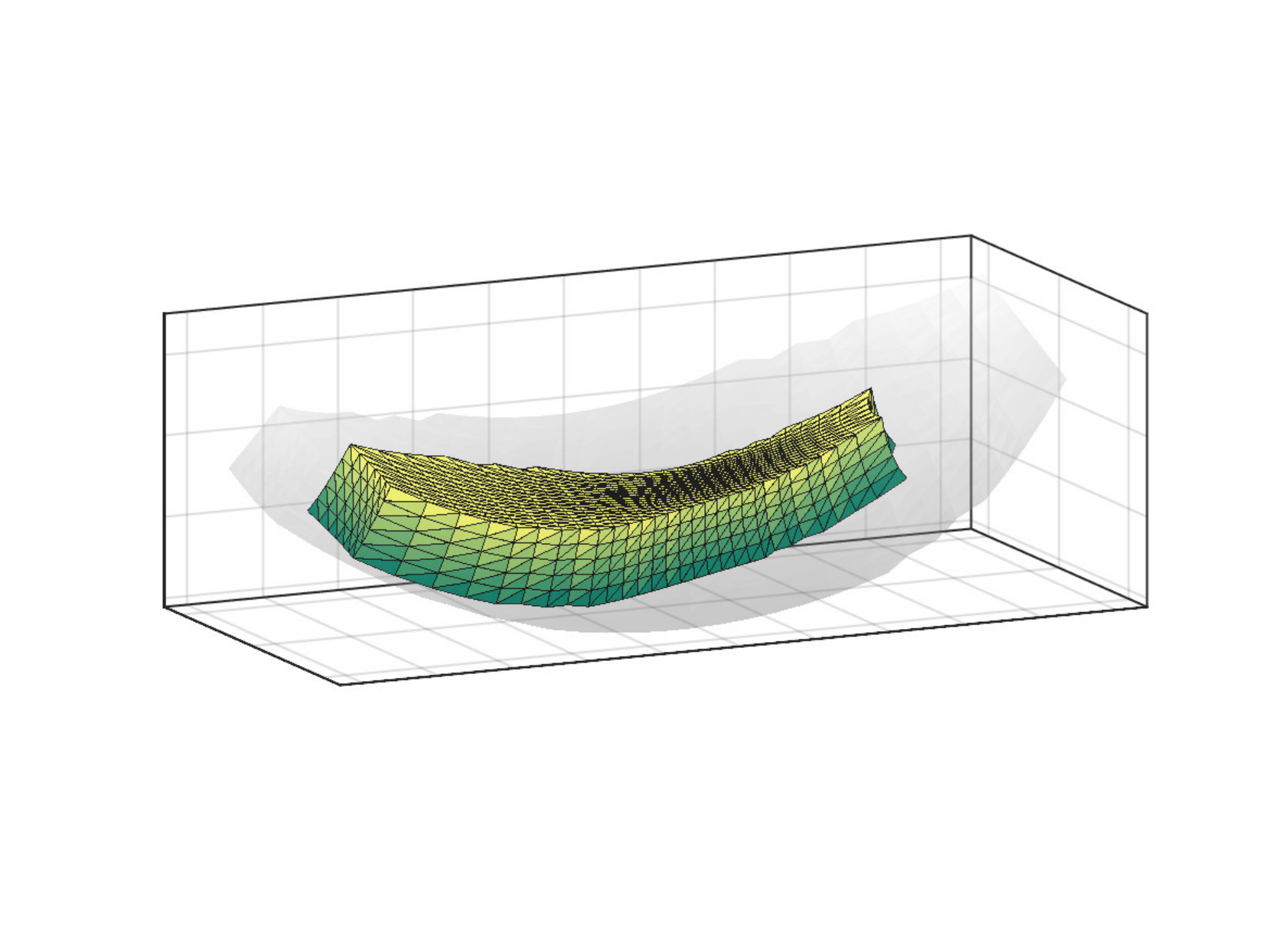}
		\caption{Tangential deformation.}
		\label{fig:layered_tangential_deformation}
	\end{subfigure}
	\begin{subfigure}[t]{0.48\textwidth}
		\centering
		\includegraphics[trim = 70 100 50 90, clip, width = 0.95\textwidth]{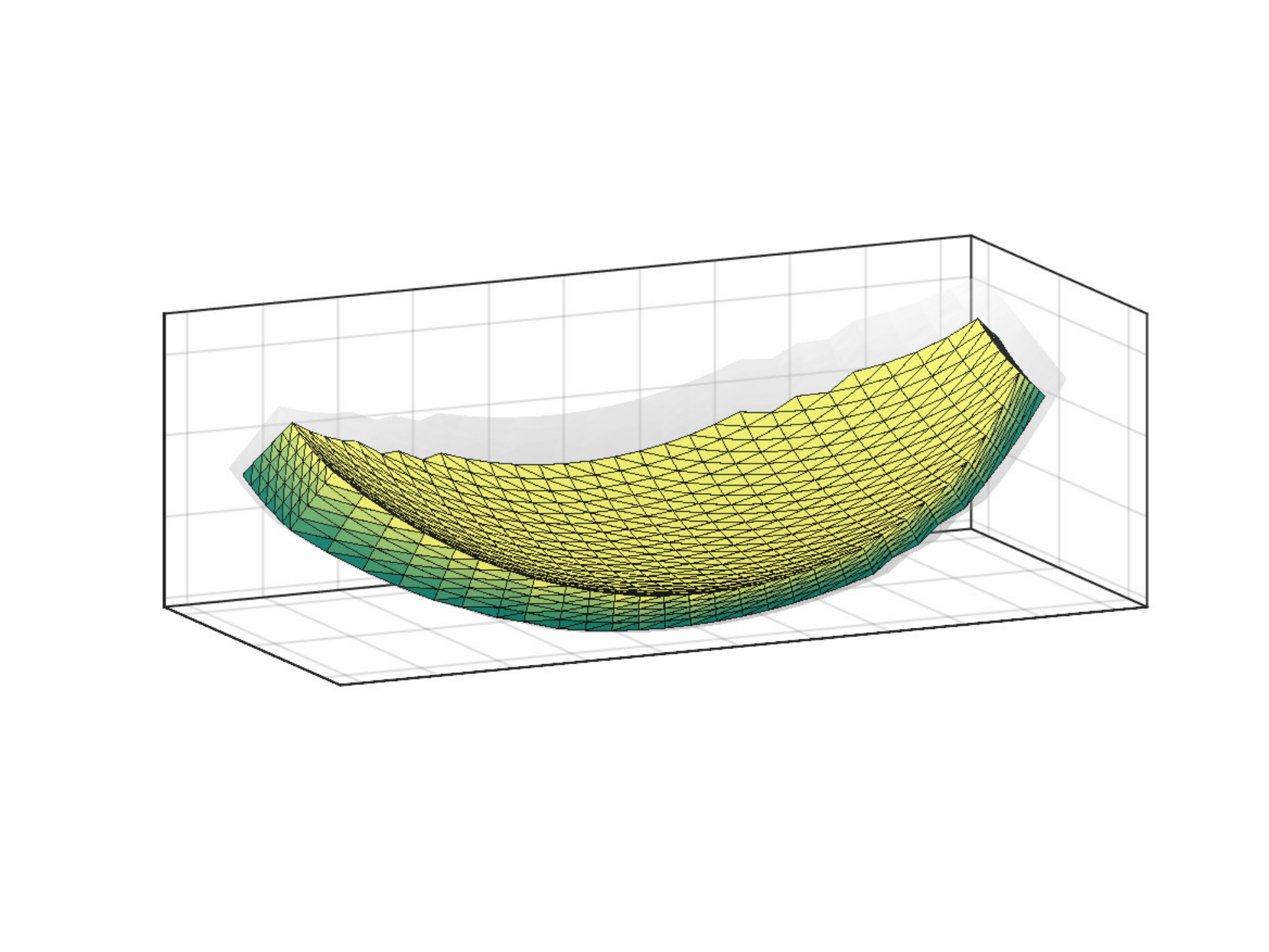}
		\caption{Transversal deformation.}
		\label{fig:layered_transversal_deformation}
	\end{subfigure}
	\caption{Responses to the {same} yank under different layered elastic parameters. In (b), $\mu_{\mathrm{tan}} = 0.02 \, \mu_{\mathrm{tsv}}$. In (c),  $\mu_{\mathrm{tsv}} = 0.02 \, \mu_{\mathrm{tan}}$.}
\end{figure}
\end{example}

\begin{example}
\label{ex:force}
Now we provide an example of yank which has a density as the gradient of a transported potential. Let $\Theta \subset \mathbb{R}^m$ be a compact set and let $g: \Theta \rightarrow L^\infty(\mathbb{R}^3, \mathbb{R})$ defined by $\theta \mapsto g_\theta$. We interpret $g_\theta$ as a parametrized potential. We assume that there exists $G_\Theta > 0$ such that 
$\|g_\theta\|_\infty \leq G_\Theta$ for all $\theta\in \Theta$
and $g_{\theta_n}\!(x) \rightarrow g_\theta(x)$ for all $x \in \mathbb{R}^3$ when $\theta_n \rightarrow \theta$. For technical reasons, let $\varOmega$ be a fixed bounded subset of $\mathbb{R}^3$ and let $\chi: \mathbb{R}^3 \rightarrow [0, 1]$ be a $C^\infty$ cutoff function of compact support such that $\restrict{\chi}{\varOmega} \equiv 1$. Under this setting, the yank $j(\varphi, \theta)$ defined by
\[
	(j(\varphi, \theta) \mid v) = -\int_{\varphi(M_0)} \chi \ g_\theta \circ \varphi^{-1} \, \mathrm{div}(v) \, dx
\]
satisfies the conditions in \hyperref[thm:opt_problem_min]{Theorem}~\ref{thm:opt_problem_min}. Note that if $\varphi(M_0) \subset \varOmega$ and $g_\theta$ is differentiable with support in the interior of $M_0$, then $(j(\varphi, \theta) \mid v) = \int_{\varphi(M_0)} \nabla(g_\theta \circ \varphi^{-1})^\top v \ dx$. In this case, it follows that $j(\varphi, \theta) = \nabla(g_\theta \circ \varphi^{-1}) \, \mathbbm{1}_{\varphi(M_0)} \, dx$, which motivates the above formulation.

We check that the conditions on $j$ in \hyperref[thm:opt_problem_min]{Theorem}~\ref{thm:opt_problem_min} are satisfied. Since
\[
	|(j(\varphi, \theta) \mid v)|
	\leq
	\|g_\theta\|_\infty \, \|v\|_{1, \infty} \, \|\chi\|_{L^1}
	\leq
	G_\Theta \, 
	c_V \|v\|_V \, \|\chi\|_{L^1} ,
\]
we see that $j(\varphi, \theta) \in V^*$ for all $\varphi \in \Diffid{1}{\mathbb{R}^3}$ and $\theta \in \Theta$, and $\|j(\varphi, \theta)\|_{V^*} \leq G_\Theta \, c_V \|\chi\|_{L^1} \equalscolon J_\Theta$.
For a convergent sequence $\theta_n \rightarrow \theta$ in $\Theta$, the assumption $g_{\theta_n}\!(x) \rightarrow g_{\theta}(x)$ for all $x \in \mathbb{R}^3$ and the dominated convergence theorem imply $(j(\varphi, \theta_n) \mid v) \rightarrow (j(\varphi, \theta) \mid v)$ for all $\varphi \in \Diffid{1}{\mathbb{R}^3}$ and $v \in V$. It remains to check that 
$\{j(\cdot, \theta): |\theta| \in \Theta\}$ 
is equi-Lipschitz with respect to the seminorm $\|\cdot\|_{1, \infty}^{M_0}$ on $\mathfrak{S}_\gamma$. Note that $A \mapsto \det A$ is a polynomial of degree $3$ in elements of $A \in \mathbb{R}^{3 \times 3}$. By the mean value theorem, there exists a constant $C > 0$ such that 
\begin{equation}
\label{eq:ex4.det}
|\det A - \det B| \leq C \, (|A| + |B|)^2 \, |A - B| 
\end{equation}
for all $A, B \in \mathbb{R}^{3 \times 3}$. It follows that, for all $\varphi, \psi\in \mathfrak S_\gamma$,
\begin{align*}
	&\hspace{15pt}
	|(j(\varphi, \theta) \mid v) - (j(\psi, \theta) \mid v)|
	\\
	&\leq
	\int_{M_0}
	\left|
		\vphantom{\sum}
		g_\theta \ (\chi \circ \varphi) \ (\mathrm{div}(v) \circ \varphi) \, |\det D\varphi|
		-
		g_\theta \ (\chi \circ \psi) \ (\mathrm{div}(v) \circ \psi) \, |\det D\psi|
	\right|
	dx
	\\
	&\leq
	\|g_\theta\|_\infty
	\left(
		\vphantom{\sum}
		\|\nabla \chi\|_\infty \, \|v\|_{1, \infty} \, \|\det D\varphi\|_\infty
		+
		 \|v\|_{2, \infty} \, \|\det D\varphi\|_\infty
	\right.
	\\
	&\hspace{45pt}
	\left.
		\vphantom{\sum}
		\phantom{}
		+
		 \|v\|_{1, \infty} \ C \, (\|D\varphi\|_\infty + \|D\psi\|_\infty)^2
	\right)
	\|\varphi - \psi\|_{1, \infty}^{M_0} \,
	\mathrm{vol}(M_0)
	\\
	&\leq
	G_\Theta \, C_\gamma \, \|v\|_V \,
	\|\varphi - \psi\|_{1, \infty}^{M_0} ,
\end{align*}
where we have made a change of variables to obtain the first inequality, split the integrand into several terms, then used \cref{eq:ex4.det} in the second inequality, and  the assumptions $\|g_\theta\|_\infty \leq G_\Theta$ and $\varphi, \psi\in \mathfrak S_\gamma$ in the last inequality. 

\end{example}


\section{Experiments}
\label{sec:experiments}

We performed experiments on simulated and real data. 
We used 2D simulated data to compare retrieved solutions with known solutions. In all experiments, we assume that shapes have a layered structure described in  \cref{ex:layers} and illustrated in \cref{fig:layered_structure}. The discrepancy function $\rho(\cdot, \cdot)$ is defined based on the varifold pseudo-metrics of \cite{Charon2013} (cf., also \cref{rem:varifold}), and is used to register certain layers of $M_0$ and $M_\targ$. In addition, to prevent applied forces to only induce rigid motions on the generated shapes, our simulations penalize the motion of the bottom layer. This is achieved by adding a penalty to the operator $A_\varphi$, replacing the second equation in \cref{eq:ctrl_problem_system} by
\begin{equation}
\label{eq:bottom.penalty}
	v
	=
	\underset{v' \,\in\, V}{\arg\min} \ 
	\frac{\reg}{2} \, \|v'\|_V^2
	+
	\frac{1}{2} \, (A_\varphi \, v' \mid v')
	-
	(j \mid v')
	+
	\frac{\beta}{2} \int_{\varphi(\mathcal{M}_{\mathrm{bottom}})} (v'^\top n)^2 \, d\sigma ,
\end{equation}
where $n$ is a unit normal vector field to $\varphi(\mathcal{M}_{\mathrm{bottom}})$. Note that we can define
\[
	(\widetilde A_\varphi \hspace{1pt} u \mid w)
	=
	(A_\varphi \hspace{1pt} u \mid w)
	+
	\beta \int_{\varphi(\mathcal{M}_{\mathrm{bottom}})} (u^\top n)(w^\top n) \, d\sigma
\]
and apply our results to $\widetilde A$. Indeed, the added term satisfies the assumption of \cref{thm:ctrl_problem,thm:opt_problem} (this will be justified in \cref{sec:proofs} at the end of the proof of \cref{prop:elastic_operator}).

All computations are implemented in CUDA and run on a computer equipped with GPU NVIDIA GeForce RTX 2080 Ti.

\subsection{2D simulations}
We take $V$ to be the RKHS associated to a Mat\'ern kernel of order 3 with width $\sigma = 0.2$ in our 2D simulations (see \cref{sec:notation}). For the varifold pseudo-metric, we use a Cauchy kernel with width 0.3 for the spatial kernel and a Binet kernel for the Grassmannian kernel (i.e., $\rho$ is as described in \cref{rem:varifold} with $\tau = 0.3$). We fix the end time $T = 1$.

\subsubsection{Free yank problem}
\Cref{fig:mixSin_template} shows a simulated layered shape with the layered structure $\varPhi: [0, 1] \times [0, 3] \rightarrow \mathbb{R}^2$ given by
\[
	\varPhi(\nu, x)
	=
	\frac{1}{20} \, \nu \left( 20 + \sin(6 x) + \frac{1}{2} \sin(10 x) + \sin(14 x) + \frac{3}{10} \sin(18 x) \right) .
\]
Denote the discretized triangles by $\{\mathcal{T}_k\}_{k = 1}^K$. We approximate $j \in V^*$ by a simple function $j = \sum_{k = 1}^K j_k \, \mathrm{area}(\mathcal{T}_k) \, \mathbbm{1}_{\mathcal{T}_k}$, where $\mathbbm{1}_{\mathcal{T}_k}$ is the indicator function on $\mathcal{T}_k$. For the purpose of illustration, we generated a deformed shape (\cref{fig:mixSin_target}) using a yank which is supported in three spatial regions, two on the top layer and one on the middle layer (\cref{fig:mixSin_yank}). The vectors $j_k$ are mapped on the vertices for visualization. 
Note that the support of the yank is simply advected by the deformation. We used the persistent isotropic elastic operator in this case, that is, $E_\varphi(\varepsilon_u, \varepsilon_v) = \lambda \, \mathrm{tr}(\varepsilon_u) \, \mathrm{tr}(\varepsilon_v) + 2\mu \, \mathrm{tr}(\varepsilon_u^\top \varepsilon_v)$, with $\lambda = 0$ and $\mu = 0.5$. Since we assume that the deformed shape is isotropic at all time, here the layered structure is actually irrelevant to the definition of the elastic operator. Using layers extracted from the deformed shape as targets, we then searched a minimizer of our free yank problem using limited-memory BFGS.

\begin{figure}
	\centering
	\begin{subfigure}{0.48\textwidth}
		\centering
		\includegraphics[width = 0.9\textwidth]{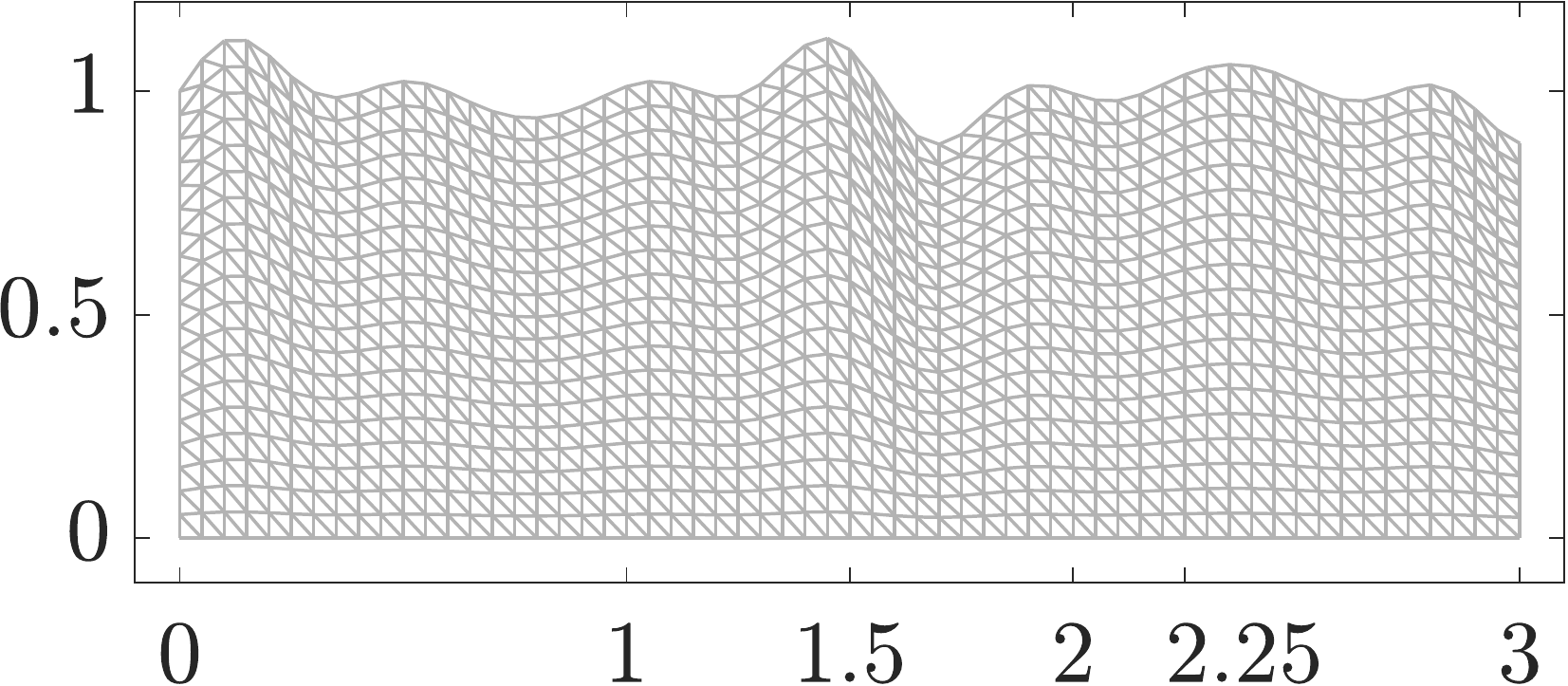}
		\caption{Template.}
		\label{fig:mixSin_template}
	\end{subfigure}
	\begin{subfigure}{0.48\textwidth}
		\centering
		\includegraphics[width = 0.9\textwidth]{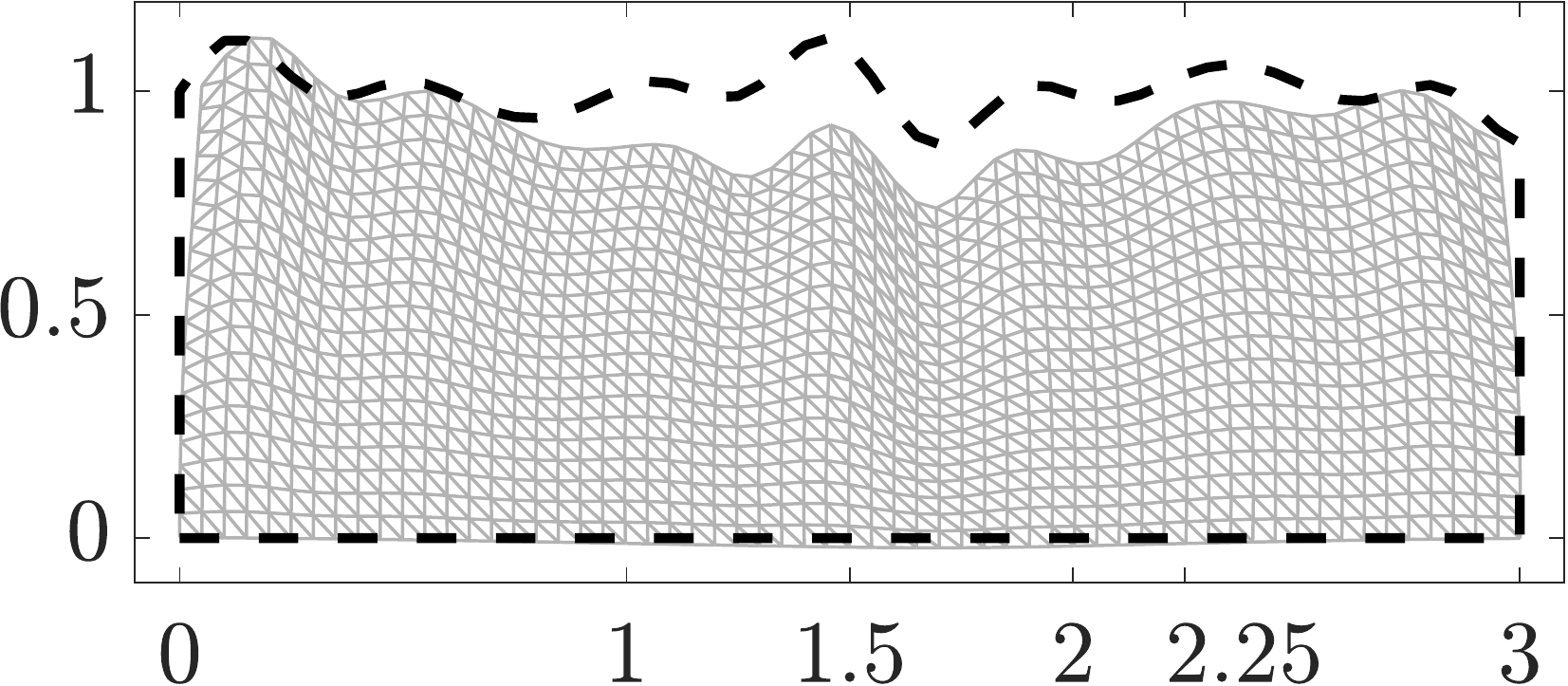}
		\caption{Deformed template.}
		\label{fig:mixSin_target}
	\end{subfigure}
	\\[15pt]
	\begin{subfigure}{\textwidth}
		\centering
		\includegraphics[width = 0.4\textwidth]{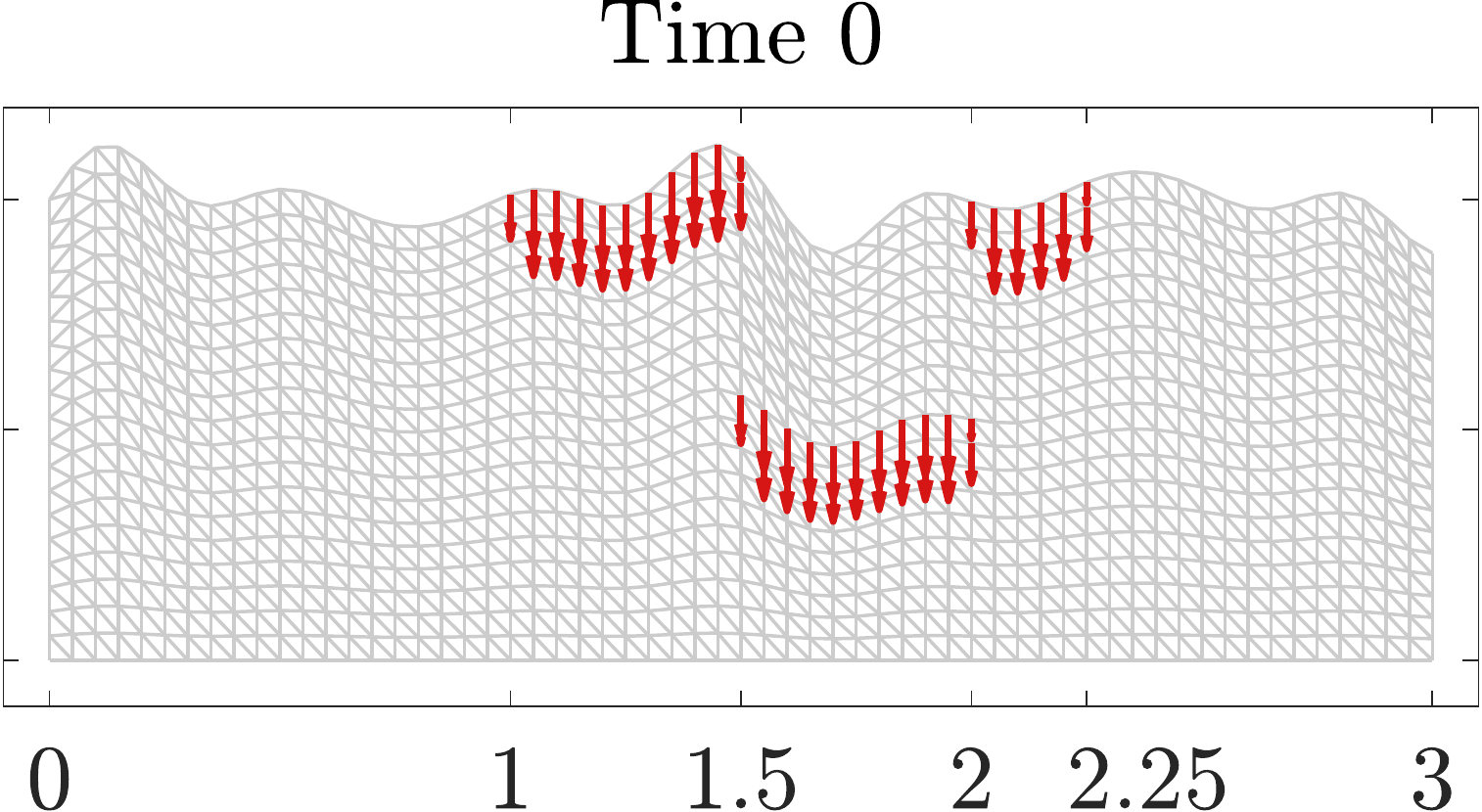}
		\hspace{10pt}
		\includegraphics[width = 0.4\textwidth]{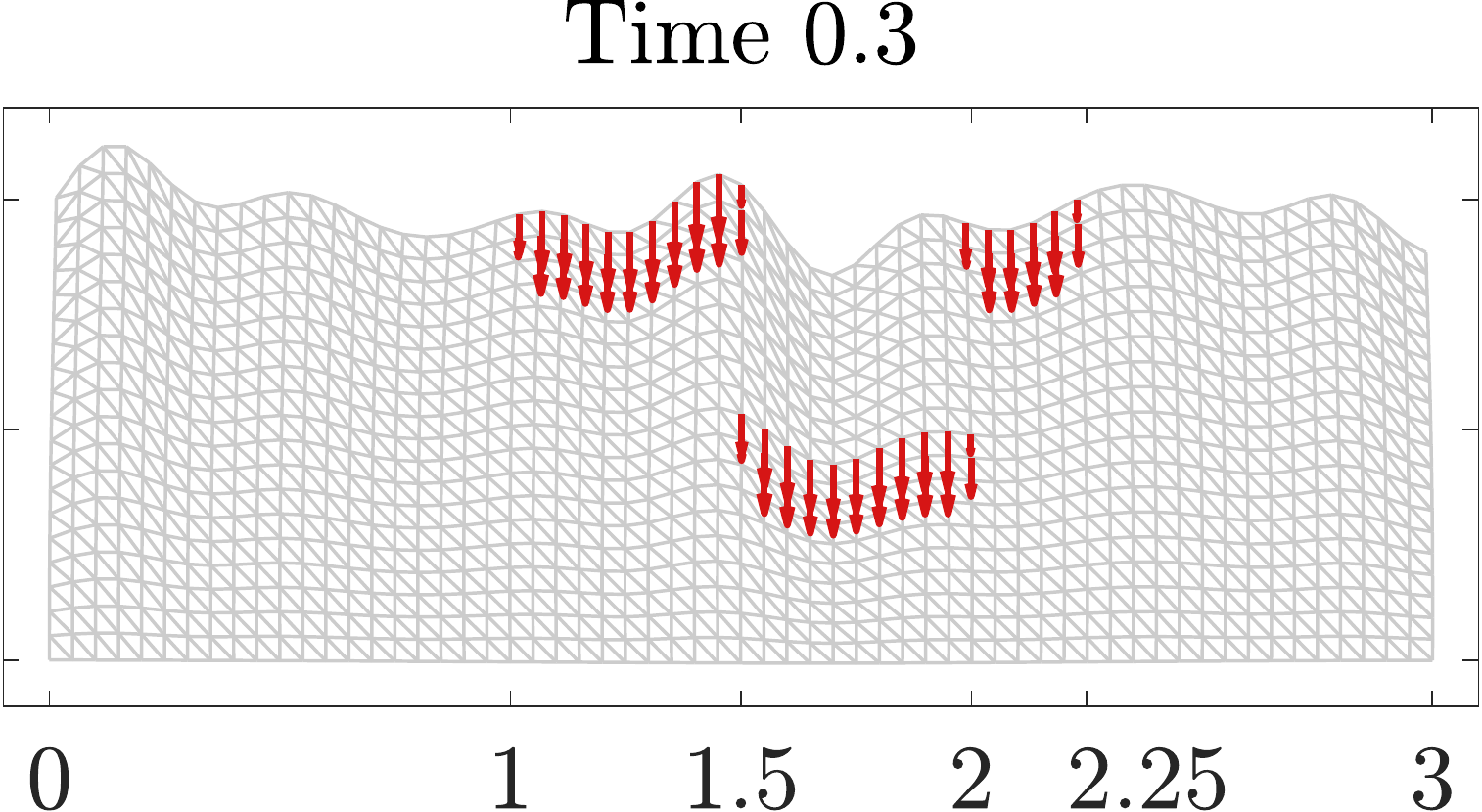}
		\\[10pt]
		\hspace{0pt}
		\includegraphics[width = 0.4\textwidth]{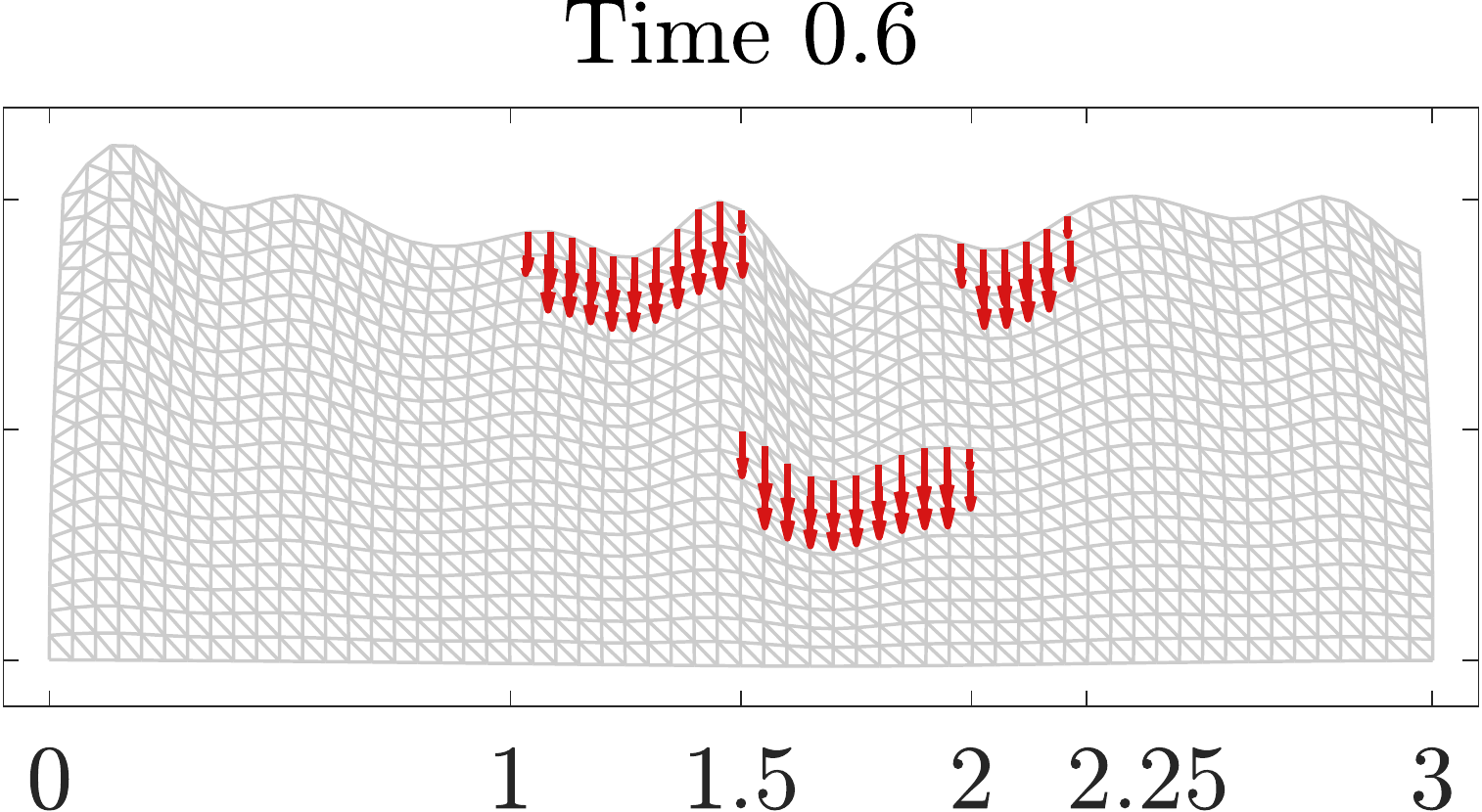}
		\hspace{10pt}
		\includegraphics[width = 0.4\textwidth]{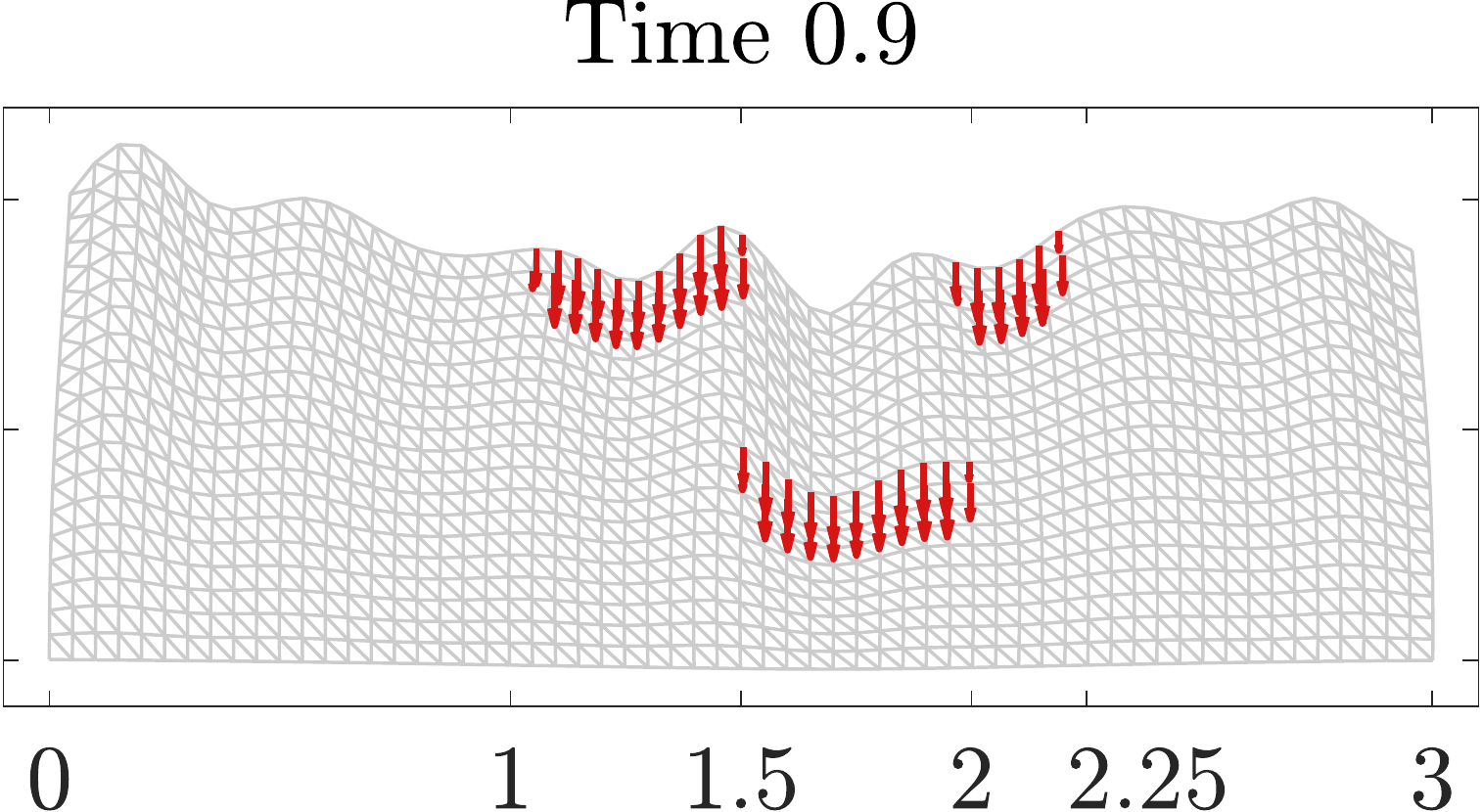}
		\vspace{5pt}
		\caption{Time-dependent yank used to generate the deformed template in (b) (vectors scaled by 20).}
		\label{fig:mixSin_yank} 
	\end{subfigure}
	\caption{Simulated data for the free yank problem.}
\end{figure}

We first consider registering top and bottom layers from ``template'' ($M_0$) to ``target'' ($M_\targ$), depicted in \cref{fig:mixSin_twoLayers_template,fig:mixSin_twoLayers_target}. Assuming the correct elastic model parameters $\lambda$ and $\mu$ are used, the registration and retrieved yank are shown in \cref{fig:mixSin_twoLayers_registration,fig:mixSin_twoLayers_yank}. We observe from \cref{fig:mixSin_twoLayers_yank} that large magnitude of the retrieved yank mainly occurs on top and bottom layers. Although the horizontal position of the true yank in the interval $[1, 2.25]$ is captured quite accurately, no yank is found in the middle layer due to the lack of information regarding the internal deformation in the discrepancy cost $\rho$.

\begin{figure}[hbt!]
	\centering
	\begin{subfigure}{0.48\textwidth}
		\centering
		\includegraphics[width = 0.9\textwidth]{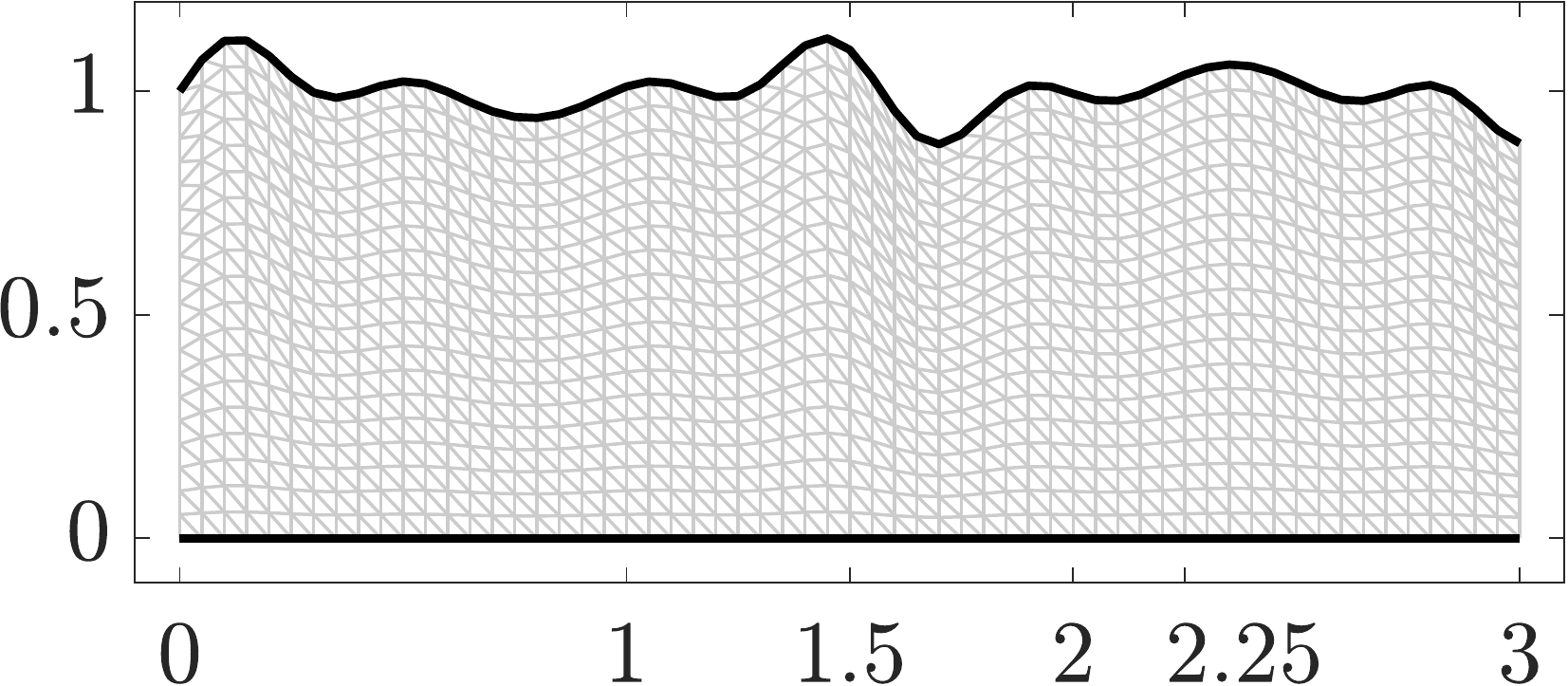}
		\caption{Template.}
		\label{fig:mixSin_twoLayers_template}
	\end{subfigure}
	\begin{subfigure}{0.48\textwidth}
		\centering
		\includegraphics[width = 0.9\textwidth]{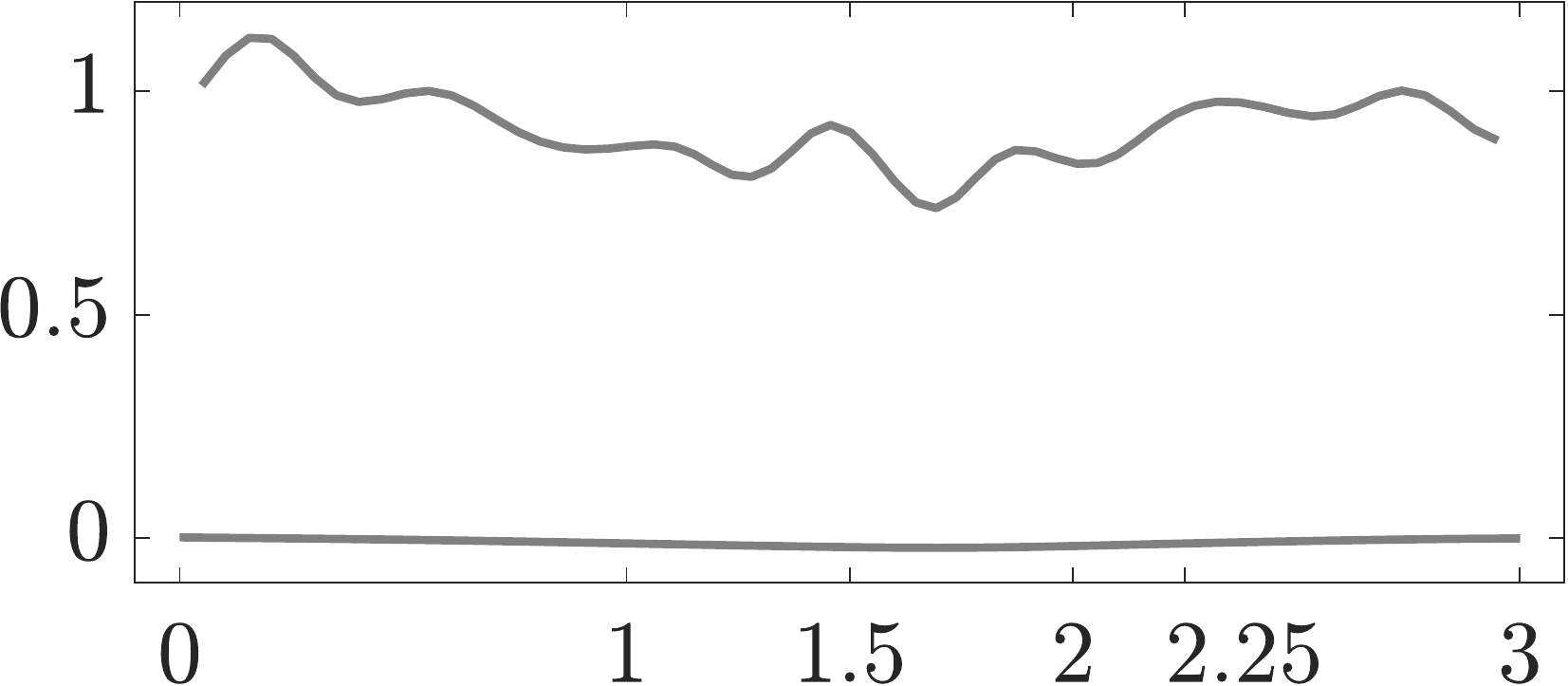}
		\caption{Simulated target.}
		\label{fig:mixSin_twoLayers_target}
	\end{subfigure}
	\\[10pt]
	\begin{subfigure}{\textwidth}
		\centering
		\includegraphics[width = 0.432\textwidth]{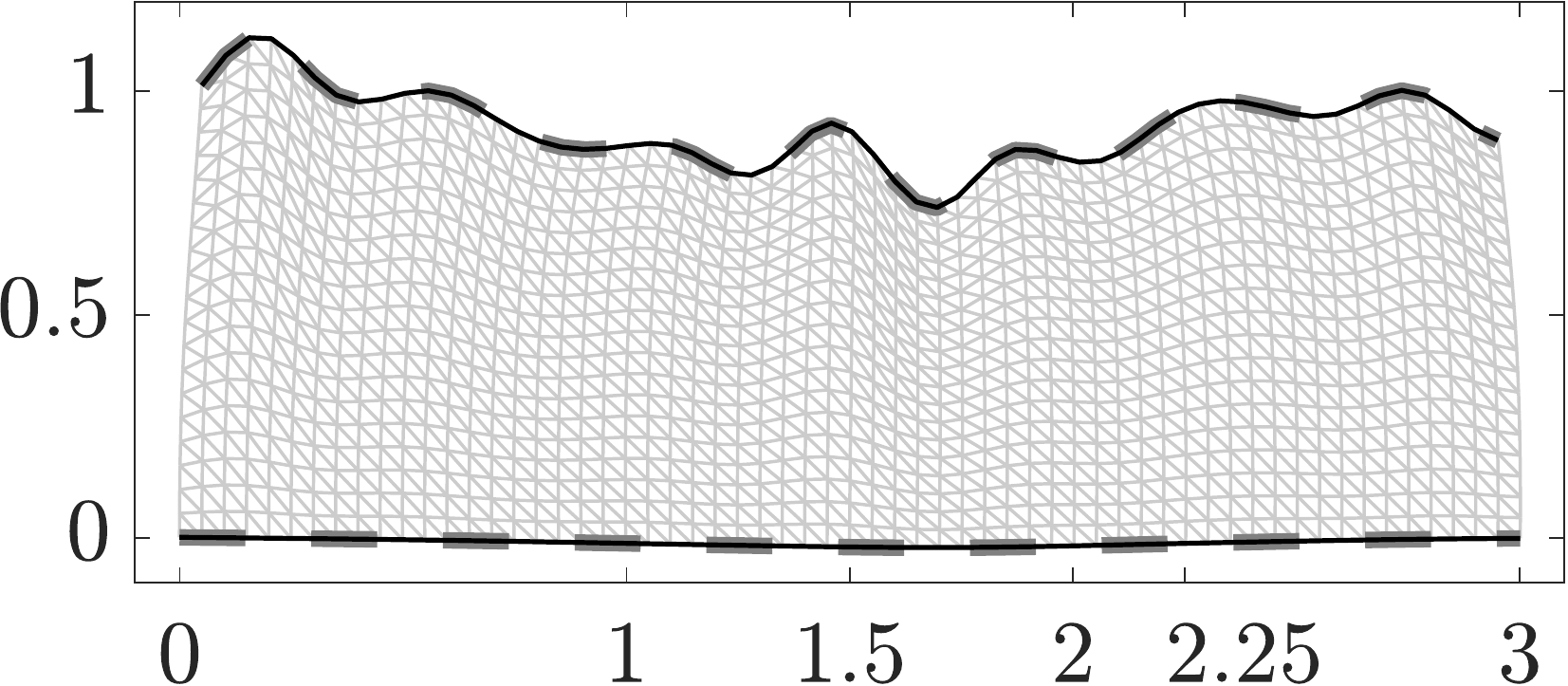}
		\caption{Registration result. The black lines indicate the deformed top and bottom layers of the template; the dashed gray lines indicate the ones of the target.}
		\label{fig:mixSin_twoLayers_registration}
	\end{subfigure}
	\\[15pt]
	\begin{subfigure}{\textwidth}
		\centering
		\includegraphics[width = 0.4\textwidth]{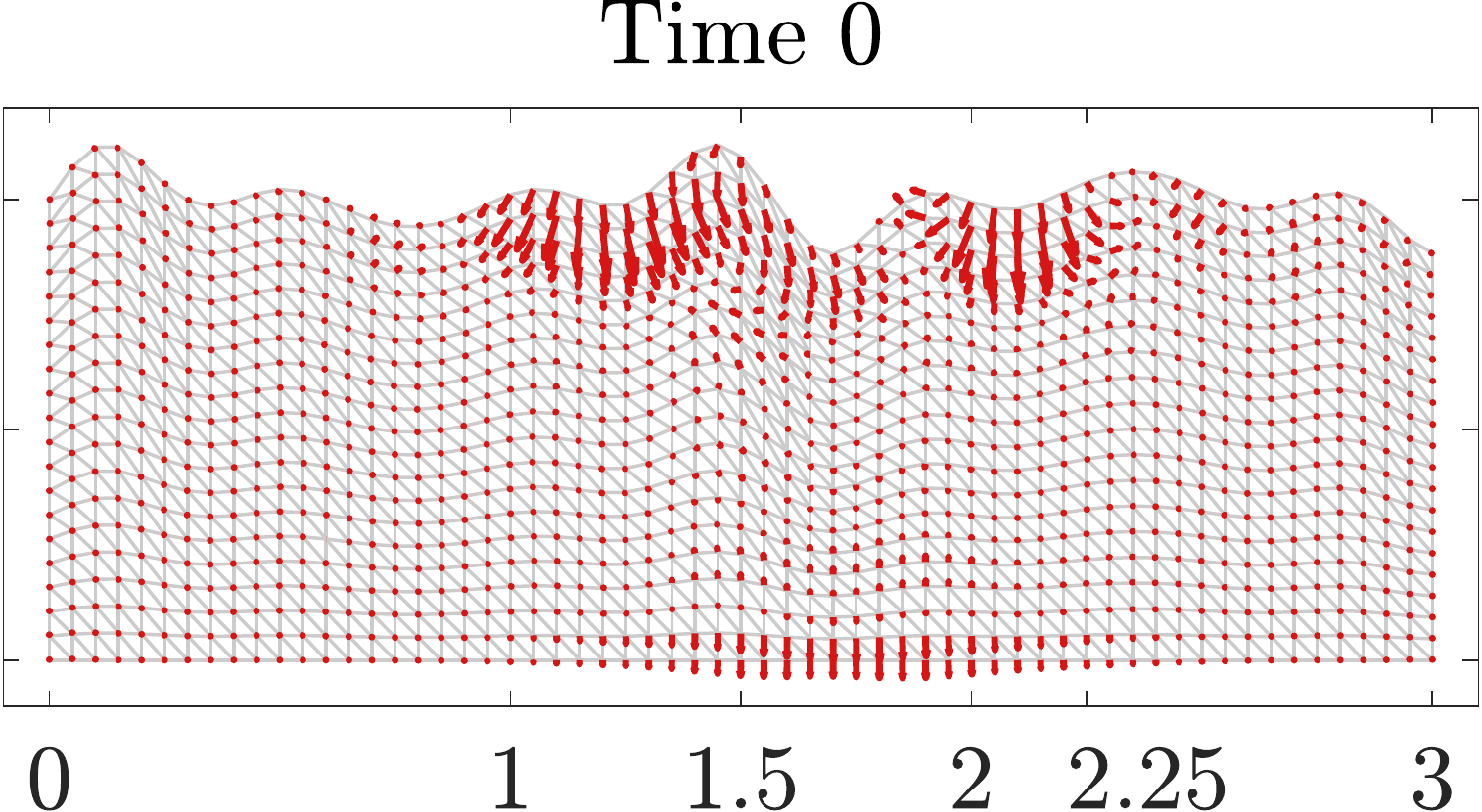}
		\hspace{10pt}
		\includegraphics[width = 0.4\textwidth]{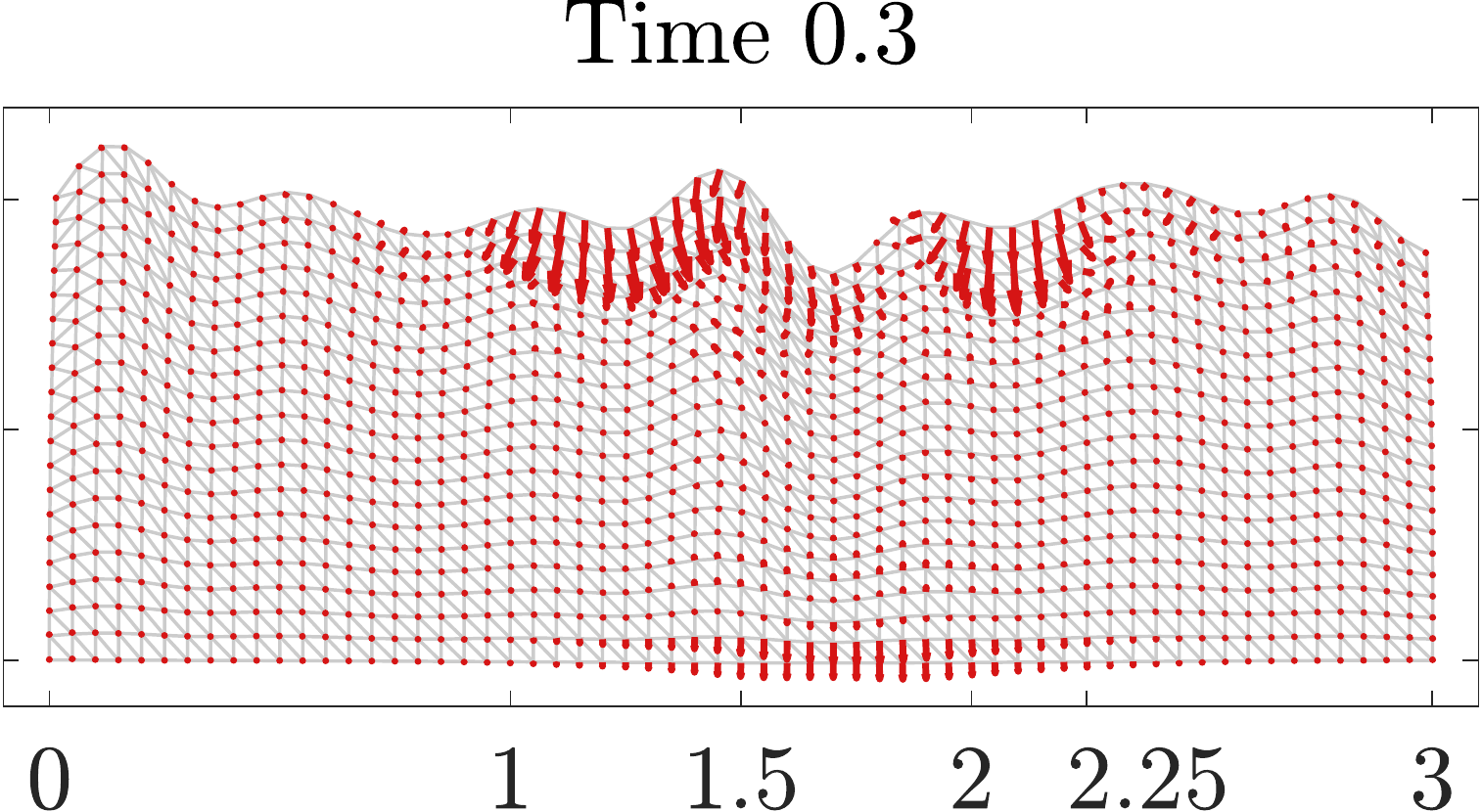}
		\\[10pt]
		\hspace{0pt}
		\includegraphics[width = 0.4\textwidth]{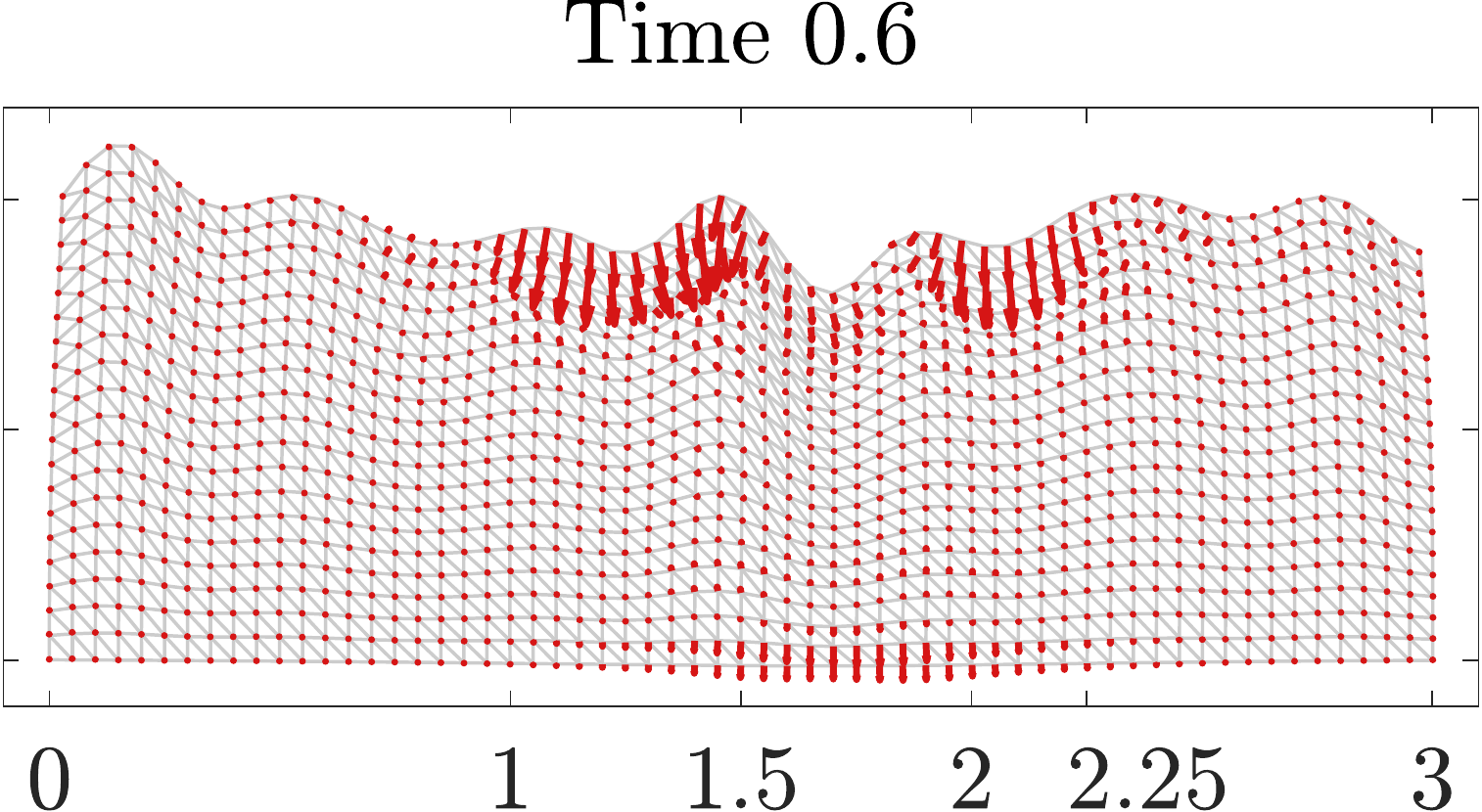}
		\hspace{10pt}
		\includegraphics[width = 0.4\textwidth]{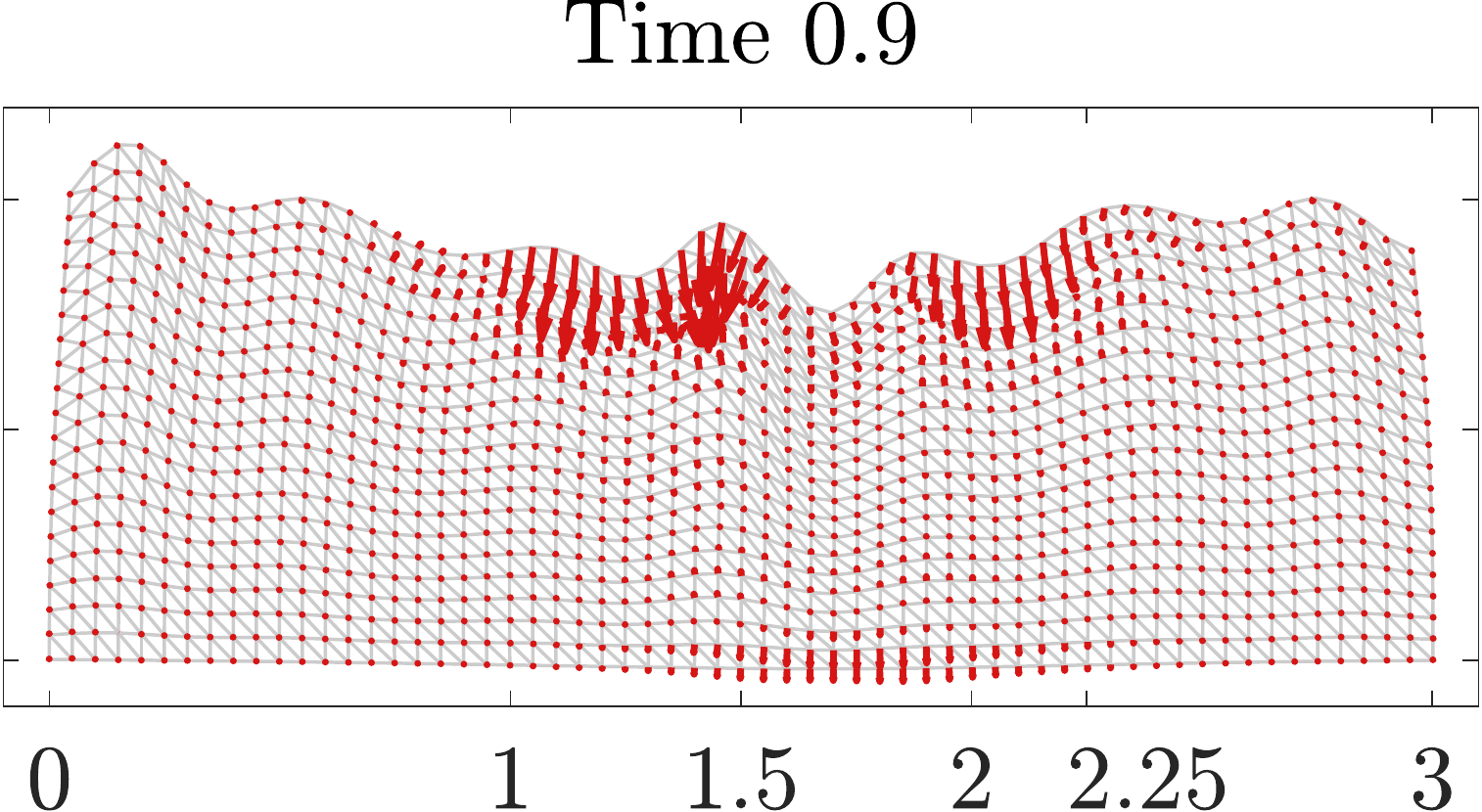}
		\vspace{5pt}
		\caption{Retrieved time-dependent yank (vectors scaled by 20).}
		\label{fig:mixSin_twoLayers_yank}
	\end{subfigure}
	\caption{Result of the free yank problem registering top and bottom layers.}
\end{figure}


In comparison, \cref{fig:mixSin_allLayers} shows the estimated registration and yank when the deformation of of all layers is observed (up to tangential motion along the layers) and  taken into account in the matching by adding discrepancy terms for each of these layers. We see that in this case the three spatial regions of support of the true yank can be located. However, observing the internal layer structure of the target is not typical in applications where usually only the external boundary of the considered volumes can be acquired or segmented.

If one does not want to assume that too much information, such as internal displacements, is available from observed data, it becomes necessary to impose more constraints on the yank itself, by assuming that prior information is known on its structure. This motivates our second model using a parametric yank. 

\begin{figure}[hbt!]
	\centering
	\begin{subfigure}{0.48\textwidth}
		\centering
		\includegraphics[width = 0.9\textwidth]{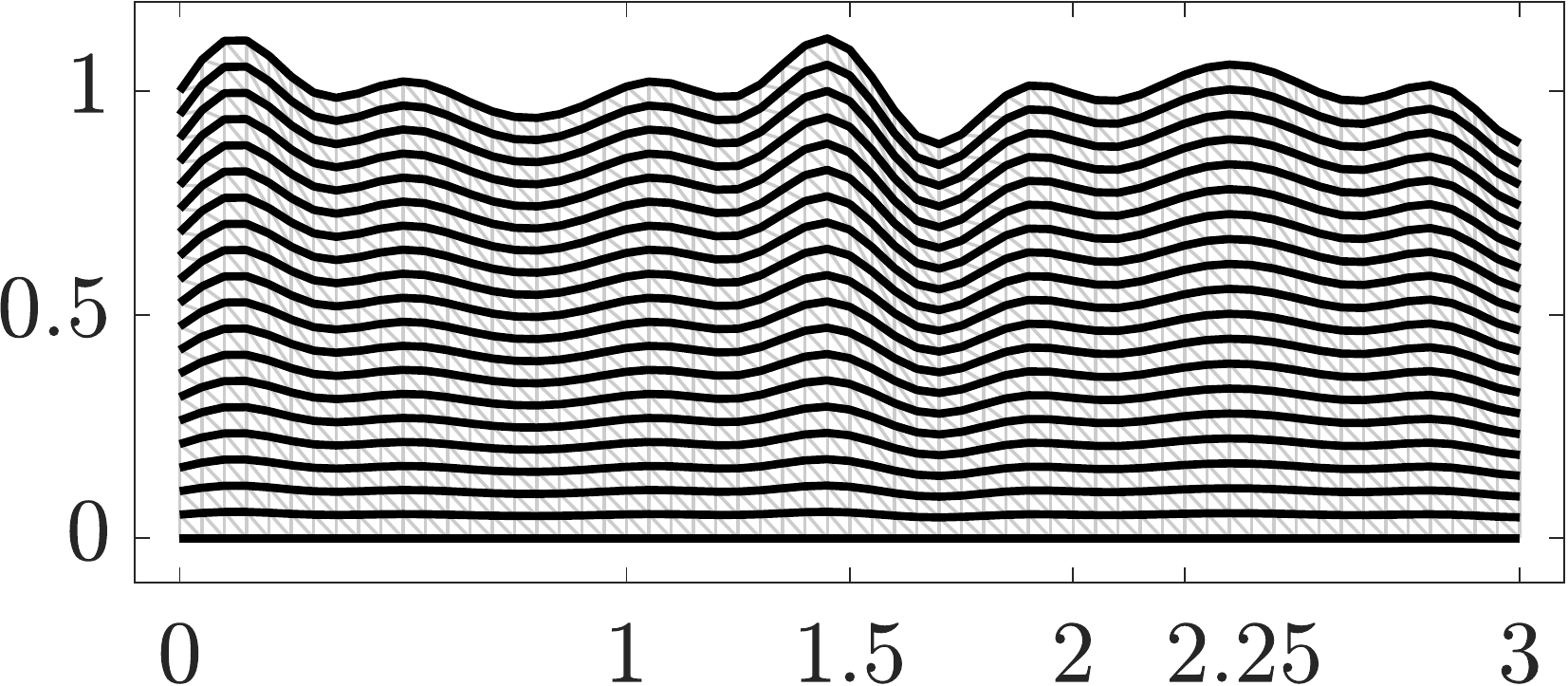}
		\caption{Template.}
	\end{subfigure}
	\begin{subfigure}{0.48\textwidth}
		\centering
		\includegraphics[width = 0.9\textwidth]{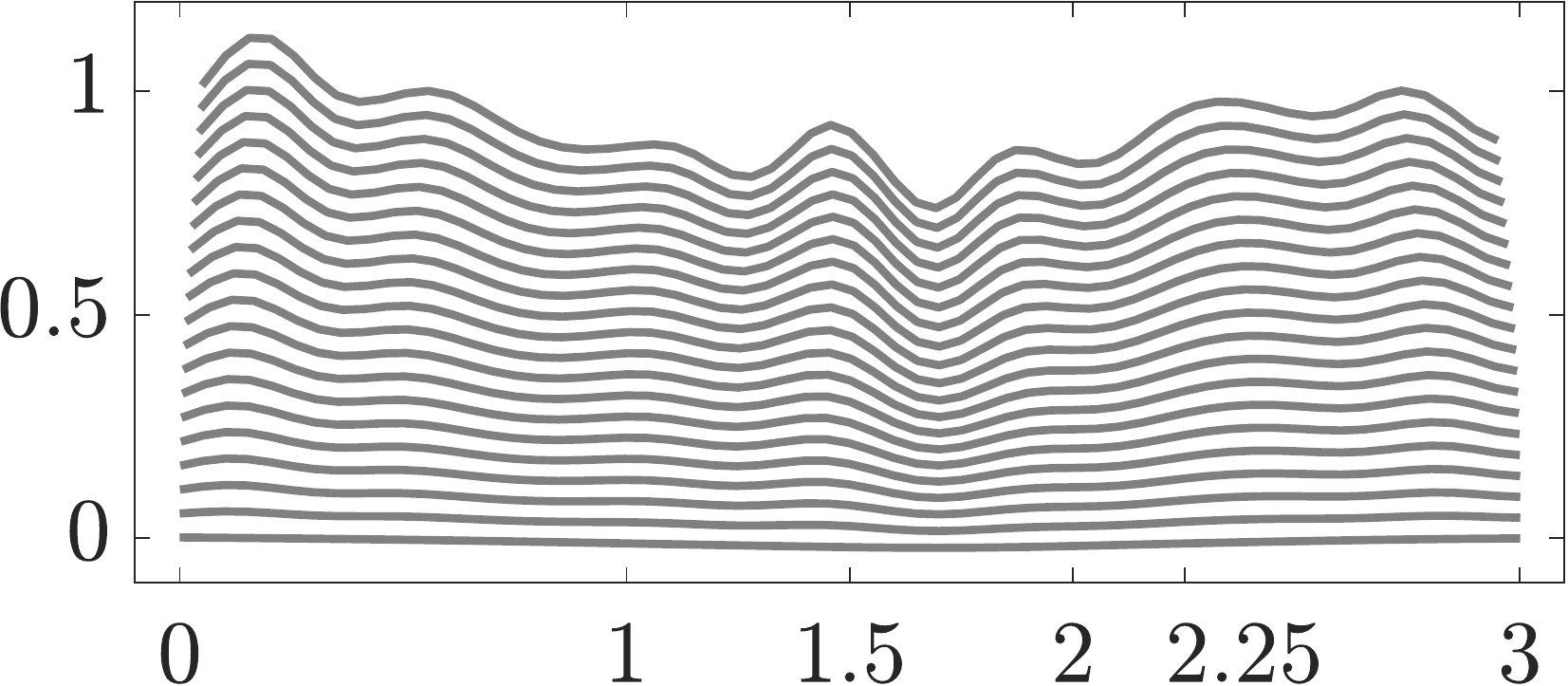}
		\caption{Simulated target.}
	\end{subfigure}
	\\[10pt]
	\begin{subfigure}{\textwidth}
		\centering
		\includegraphics[width = 0.432\textwidth]{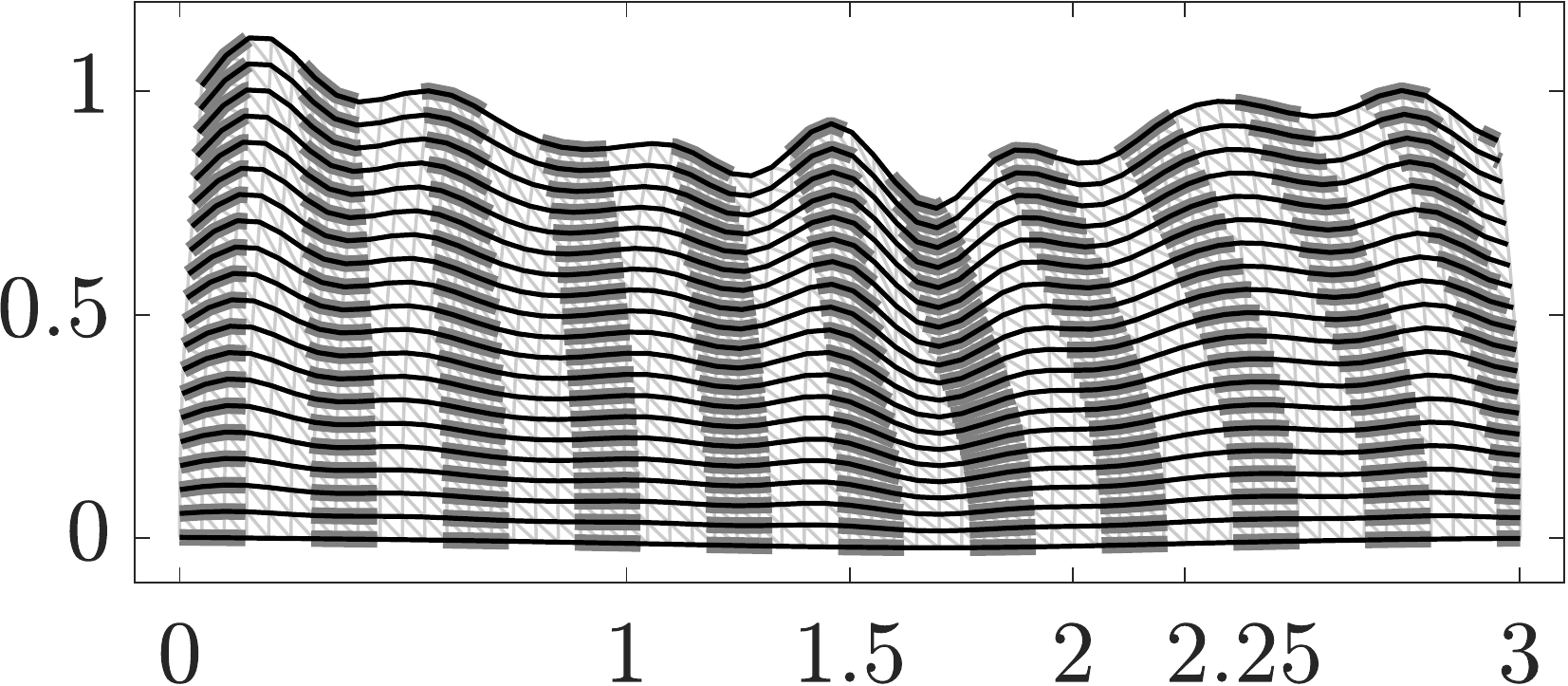}
		\caption{Registration result. The black lines indicate the deformed layers of the template; the dashed gray lines indicate the ones of the target.}
	\end{subfigure}
	\\[15pt]
	\begin{subfigure}{\textwidth}
		\centering
		\includegraphics[width = 0.4\textwidth]{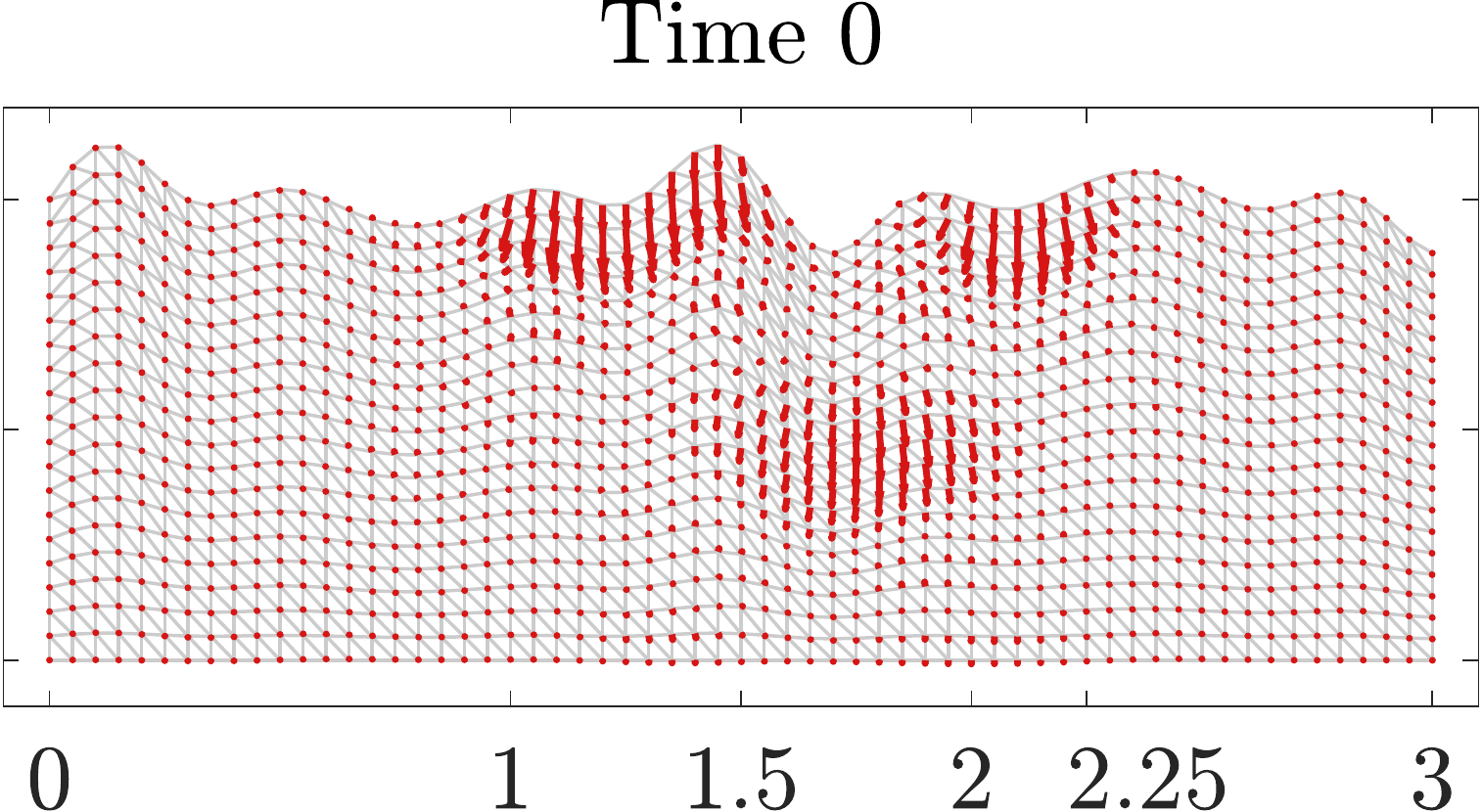}
		\hspace{10pt}
		\includegraphics[width = 0.4\textwidth]{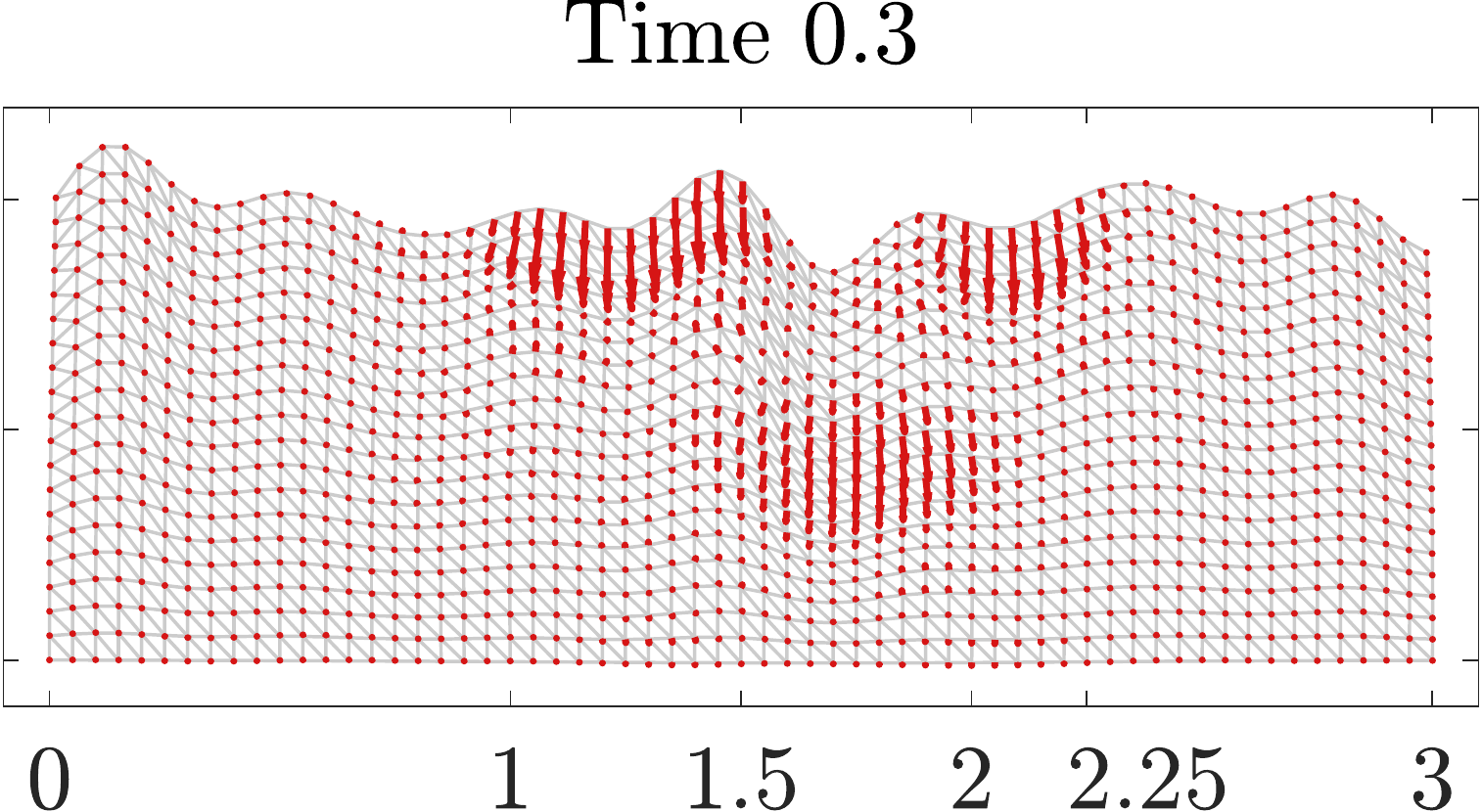}
		\\[10pt]
		\hspace{0pt}
		\includegraphics[width = 0.4\textwidth]{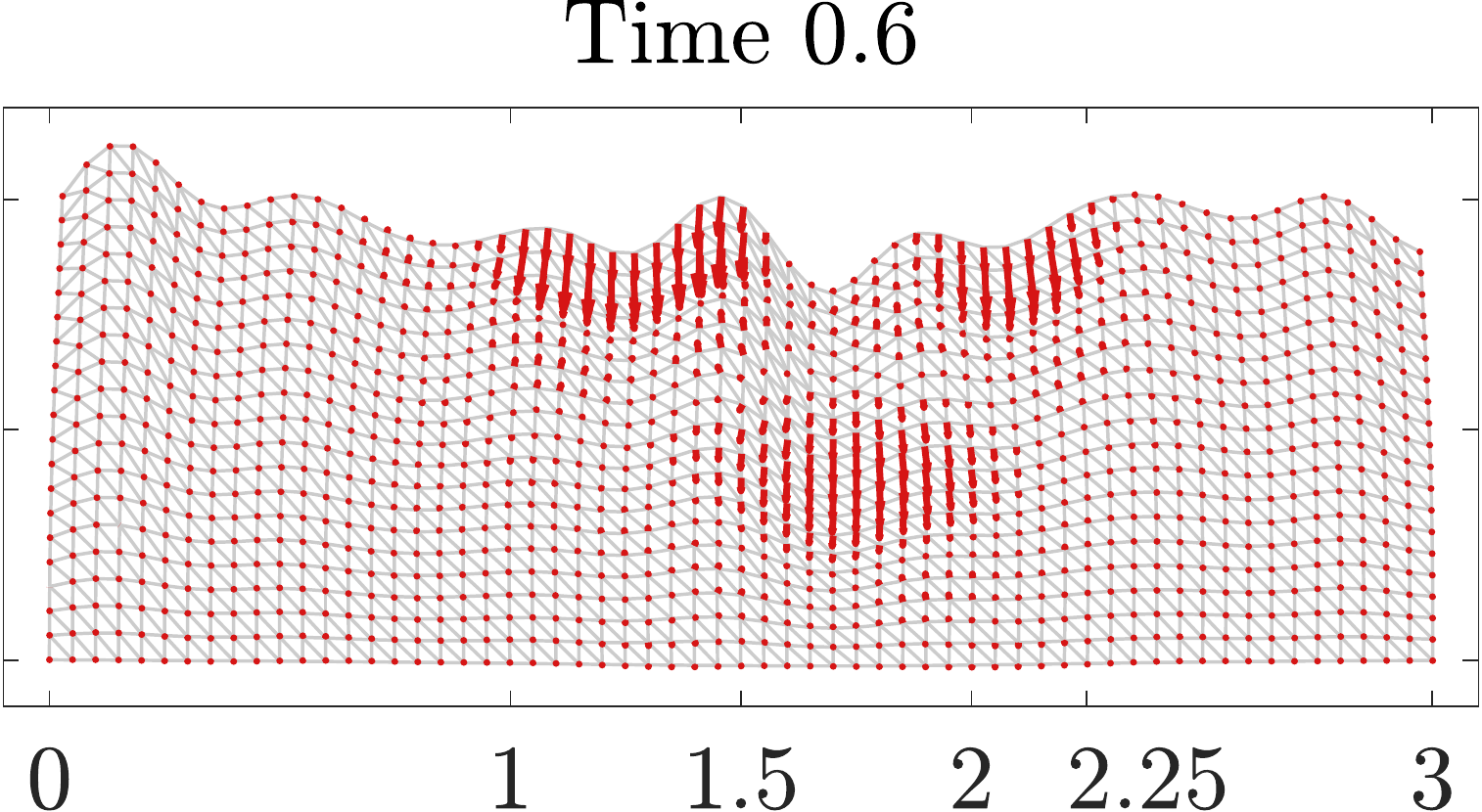}
		\hspace{10pt}
		\includegraphics[width = 0.4\textwidth]{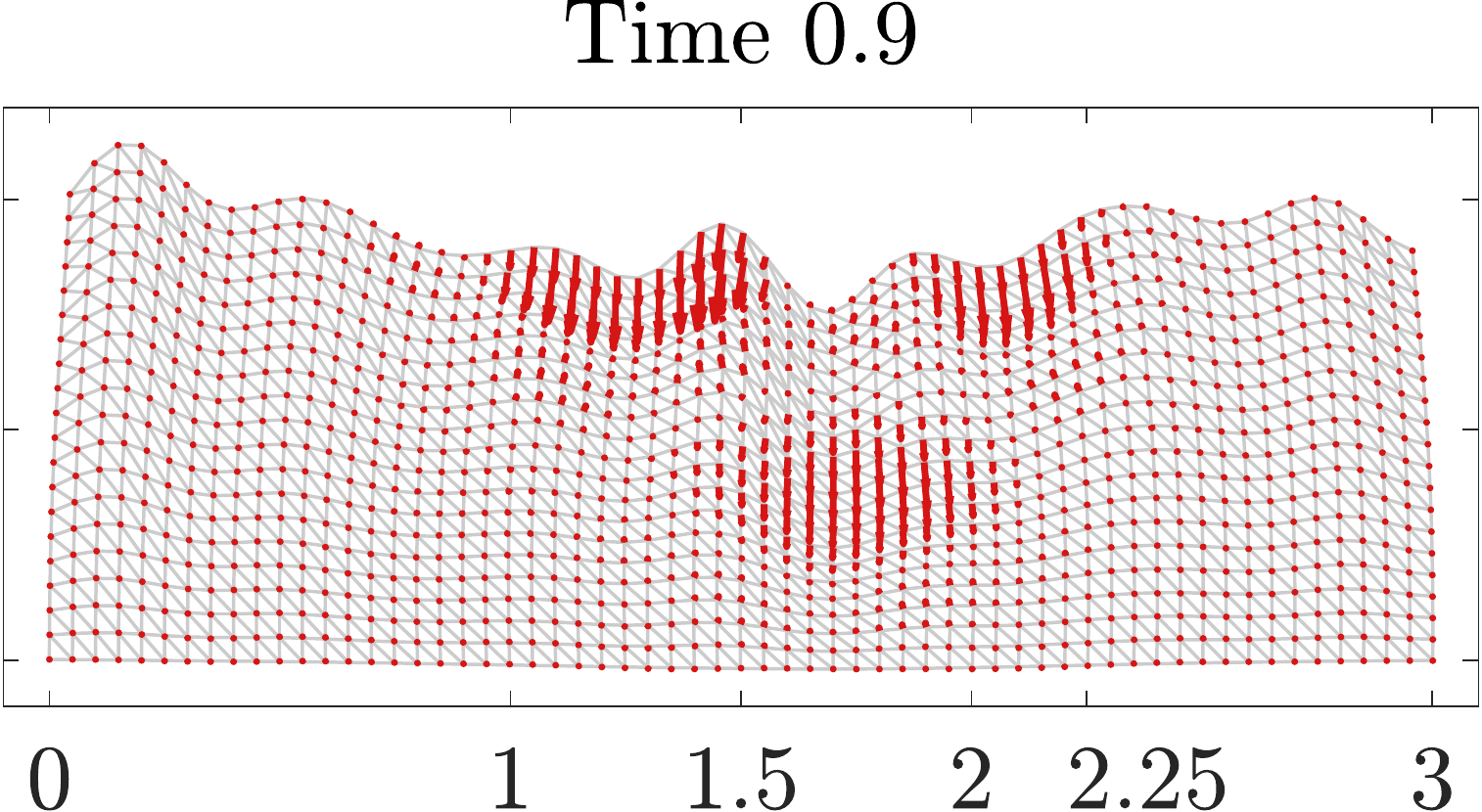}
		\vspace{5pt}
		\caption{Retrieved time-dependent yank (vectors scaled by 20).}
	\end{subfigure}
	\caption{Result of the free yank problem registering all layers.}
	\label{fig:mixSin_allLayers}
\end{figure}

\subsubsection{Parametric yank problem}

To mimic the laminar organization of cortical volumes \cite{damasio1995human,amunts2013bigbrain}, we simulated a layered shape  for this experiment. \Cref{fig:sin_template} shows our simulated shape whose middle layer is the graph of $x \mapsto 0.25 \, \cos(2.5 (x - 0.1)) + 0.6$. Other layers are generated through normal displacement starting from the middle layer with a step size 0.05. We use a parametric yank of the form of \cref{ex:force}, that is, 
\[
	(j(\varphi, \theta) \mid v) = -\int_{\varphi(M_0)} \chi \ g_\theta \circ \varphi^{-1} \, \mathrm{div}(v) \, dx .
\]
The potential $g_\theta$ we used is a $C^1$ compactly supported function
\begin{equation}
    \label{eq:g.theta}
	g_{(c, h)}(x; r)
	=
	\left\{ 
		\begin{array}{ll}
			\displaystyle h \left( \frac{|x - c|^2}{r^2} - 1 \right)^2 & \mbox{ if } |x - c| \leq r \\
			0 & \mbox{ otherwise }
		\end{array}
	\right. .
\end{equation}
Note that the parameter $\theta = (c, h)$ is composed of the center $c = (c_x, c_y) \in \mathbb{R}^2$ and the height $h \in \mathbb{R}$.
We assume that the radius $r$ is known. \Cref{fig:sin_potential} shows the potential with $c = (1.5, 0.5)$, $h = 2$, and $r = 0.25$. Given $\theta = (c, h)$, we then computed the solution $\varphi_\theta$ to the system \cref{eq:opt_problem_system} under the layered elastic operator (equation \cref{eq:layered_stiffness}) with $\lambda_{\mathrm{tan}} = 0$ and $\mu_{\mathrm{tan}} = \mu_{\mathrm{tsv}} = \mu_{\mathrm{ang}} = 1$. The deformation $\varphi_\theta(1, M_0)$ is shown in \cref{fig:sin_target}, and the yank $j(\varphi_\theta(t), \theta)$ is shown in \cref{fig:sin_yank}. The top and bottom layers of $\varphi_\theta(1, M_0)$ were extracted as the target for our finite-dimensional optimization problem. Using a BFGS optimization method with multiple starting points sampled by a Latin hypercube design, we can retrieve $(c_x, c_y, h) = (1.5, 0.5, 2)$ within an absolute error $10^{-4}$.
\begin{figure}[hbt!]
	\centering
	\begin{subfigure}[b]{0.48\textwidth}
		\centering
		\includegraphics[width = 0.93\textwidth]{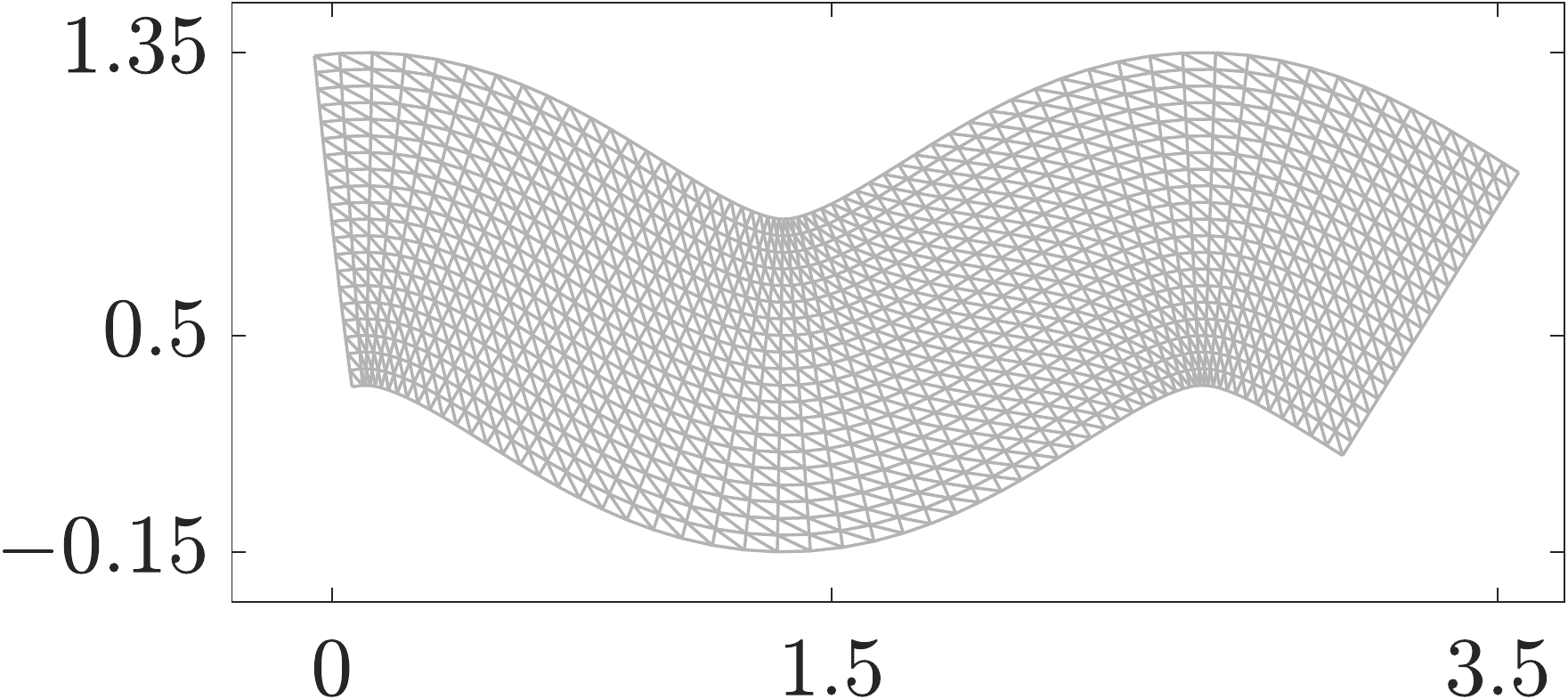}
		\caption{Template.}
		\label{fig:sin_template}
	\end{subfigure}
	\ 
	\begin{subfigure}[b]{0.48\textwidth}
		\centering
		\includegraphics[trim = 0 95 40 100, clip, width = \textwidth]{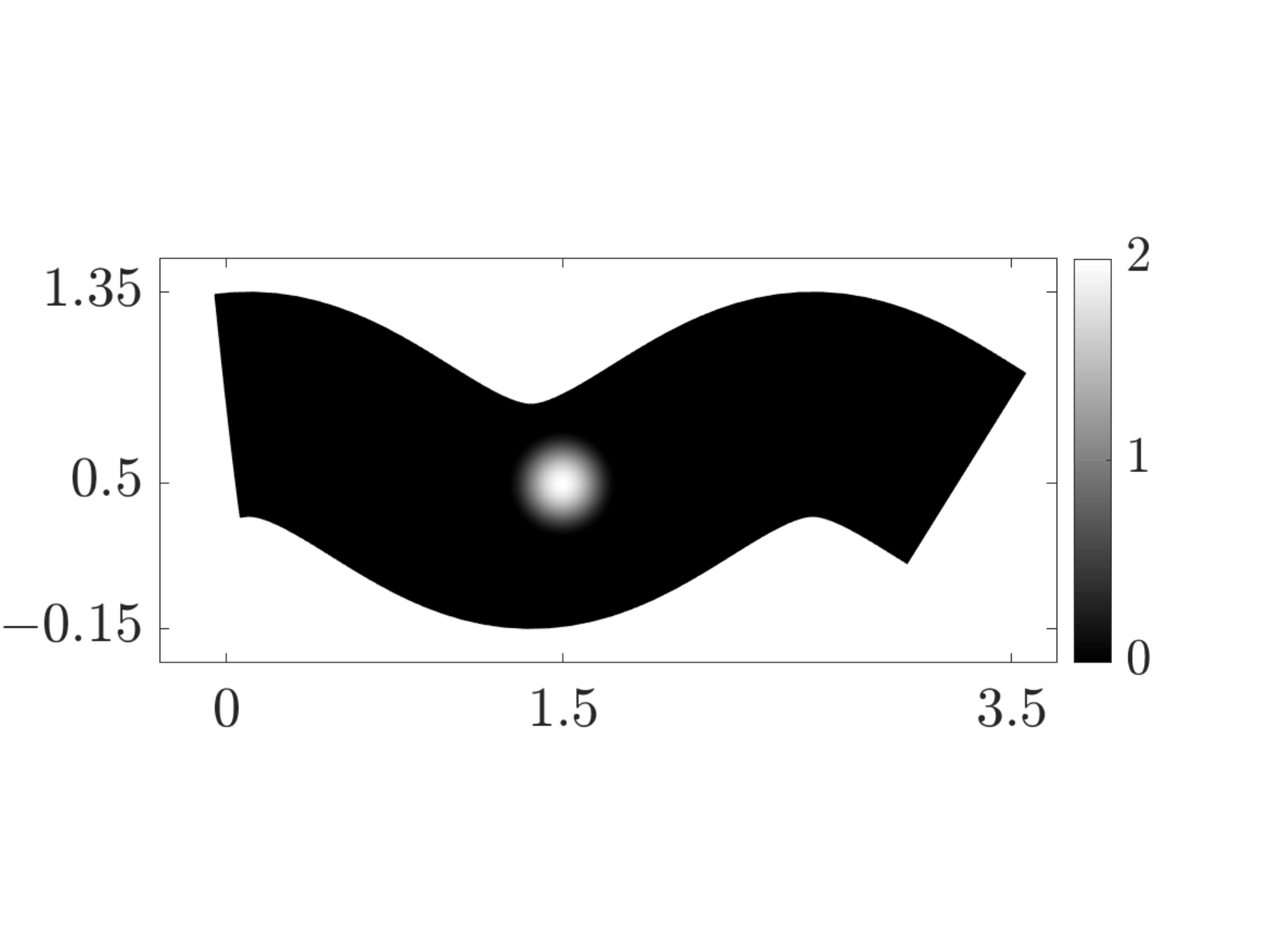}
		\caption{Potential at time zero.}
		\label{fig:sin_potential}
	\end{subfigure}
	\\[10pt]
	\begin{subfigure}[b]{0.48\textwidth}
		\centering
		\includegraphics[width = 0.93\textwidth]{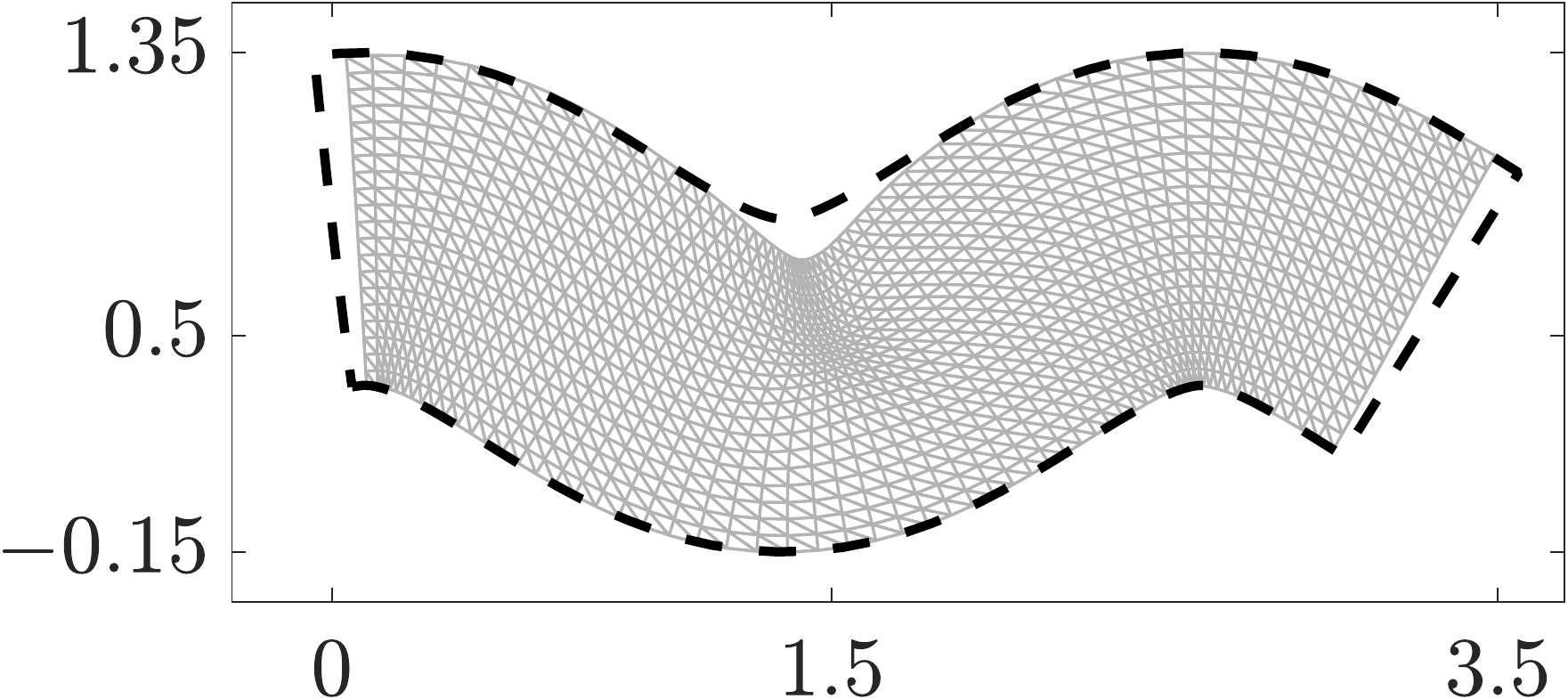}
		\caption{Deformed template.}
		\label{fig:sin_target}
	\end{subfigure}
	\\[15pt]
	\begin{subfigure}{\textwidth}
		\centering
		\includegraphics[width = 0.4\textwidth]{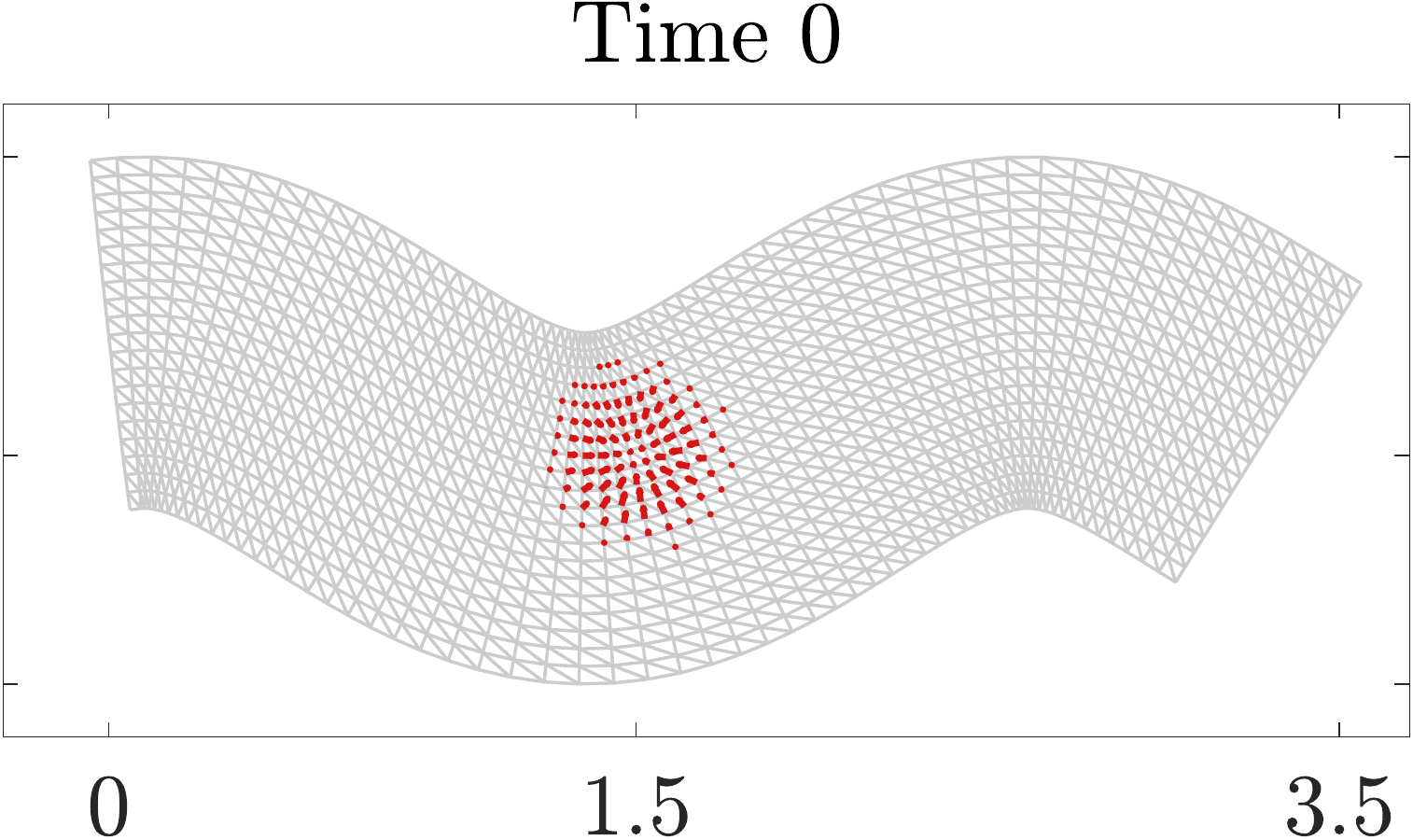}
		\hspace{10pt}
		\includegraphics[width = 0.4\textwidth]{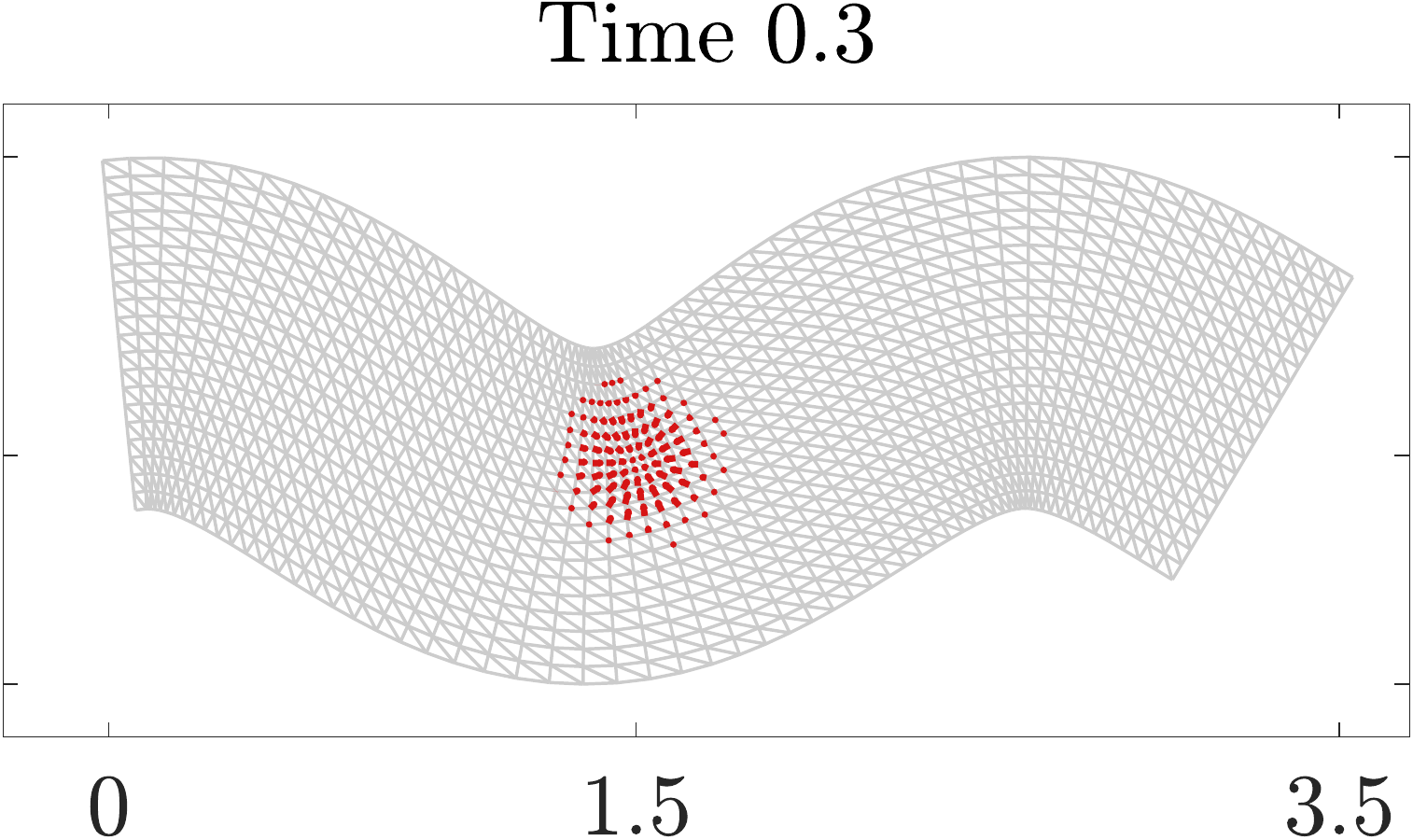}
		\\[10pt]
		\includegraphics[width = 0.4\textwidth]{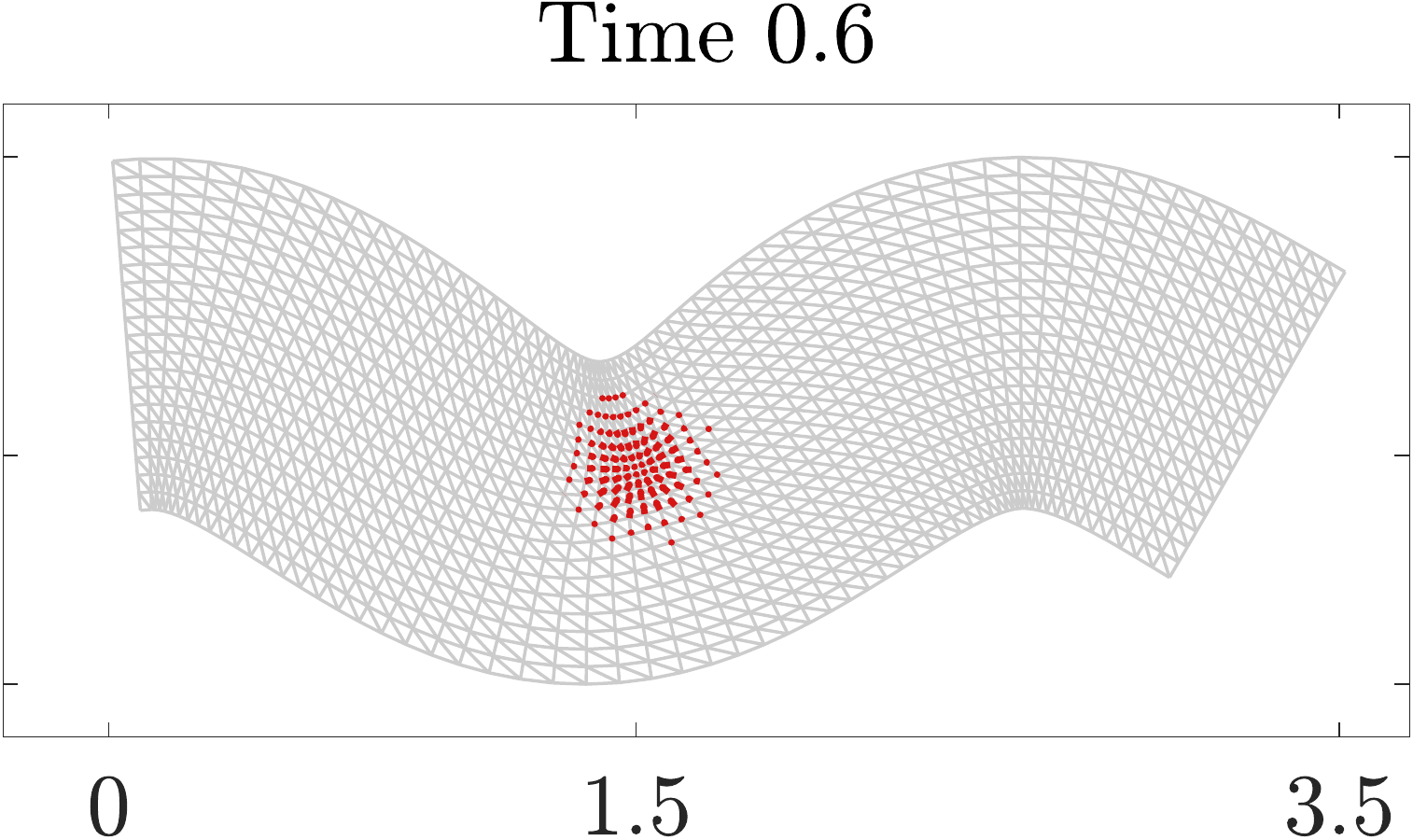}
		\hspace{10pt}
		\includegraphics[width = 0.4\textwidth]{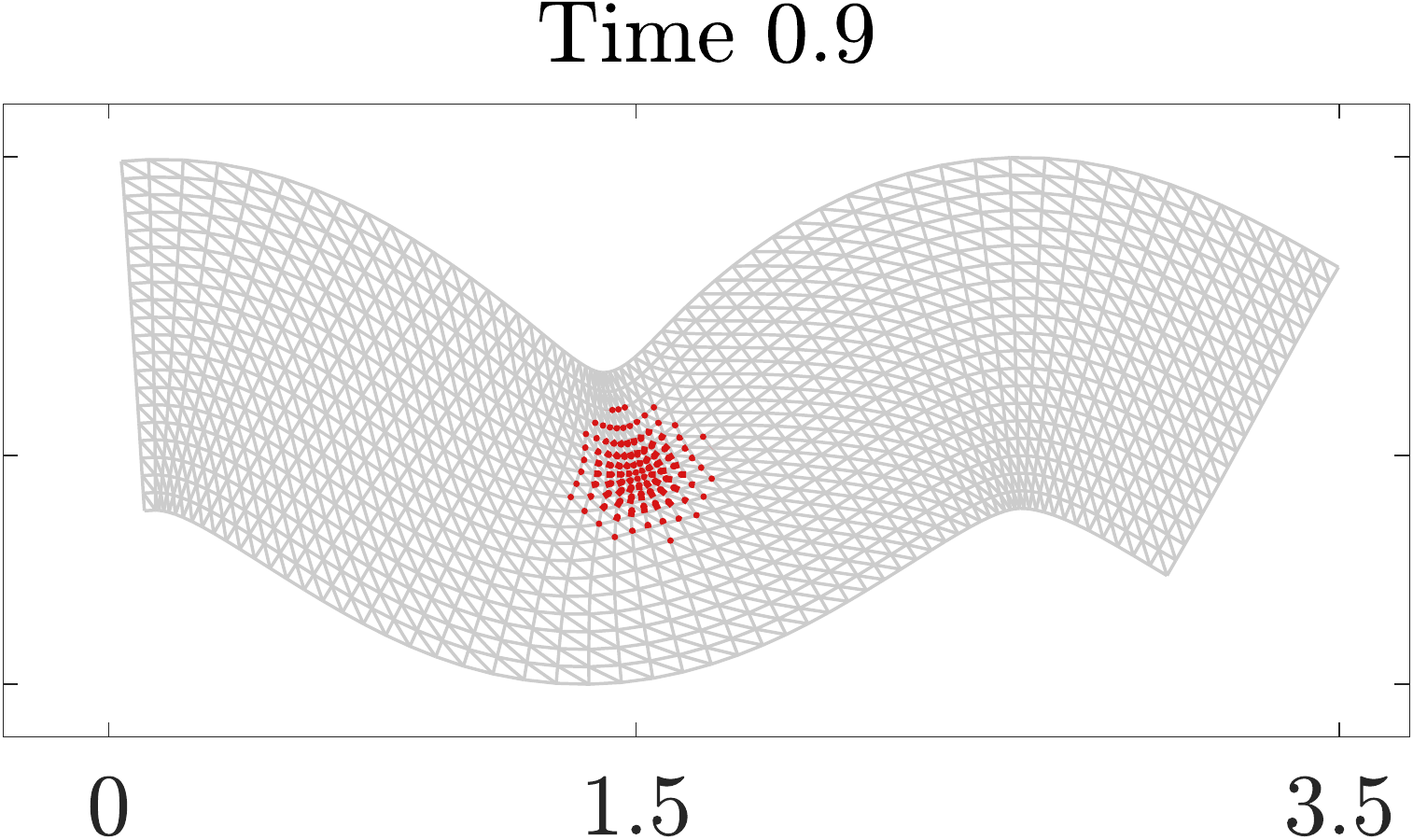}
		\vspace{5pt}
		\caption{Time-dependent yank induced by the transported potential.}
		\label{fig:sin_yank}
	\end{subfigure}
	\caption{Simulated data for the parametrized yank problem.}
\end{figure}

We now examine the robustness of our method  when $r$ or the elastic parameters that are used in the inverse problem differ from those used to generate the  target. In \cref{fig:sin_sensitivity_radius}, we plot the computed minimizer $\theta^* = (c_x^*, c_y^*, h^*)$ when we fix a different $r$. While the retrieved height $h^*$ is inversely proportional to the radius $r$ with a fitted relationship $h^* = \mathcal{O}(r^{-2.37})$, the retrieved center $(c_x^*, c_y^*)$ remains close to the true one $(1.5, 0.5)$. We remark that the relationship $h^* = \mathcal{O}(r^{-p})$ was also observed in other simulated shapes, but with a different $p > 0$. The retrieved center is also quite stable when we vary the elastic parameters as we can see from \cref{fig:sin_sensitivity_elastic}, except for very small $\mu_{\mathrm{tan}}$ or $\mu_{\mathrm{tsv}}$.

\begin{figure}[hbt!]
	\centering
		\includegraphics[height = 100pt]{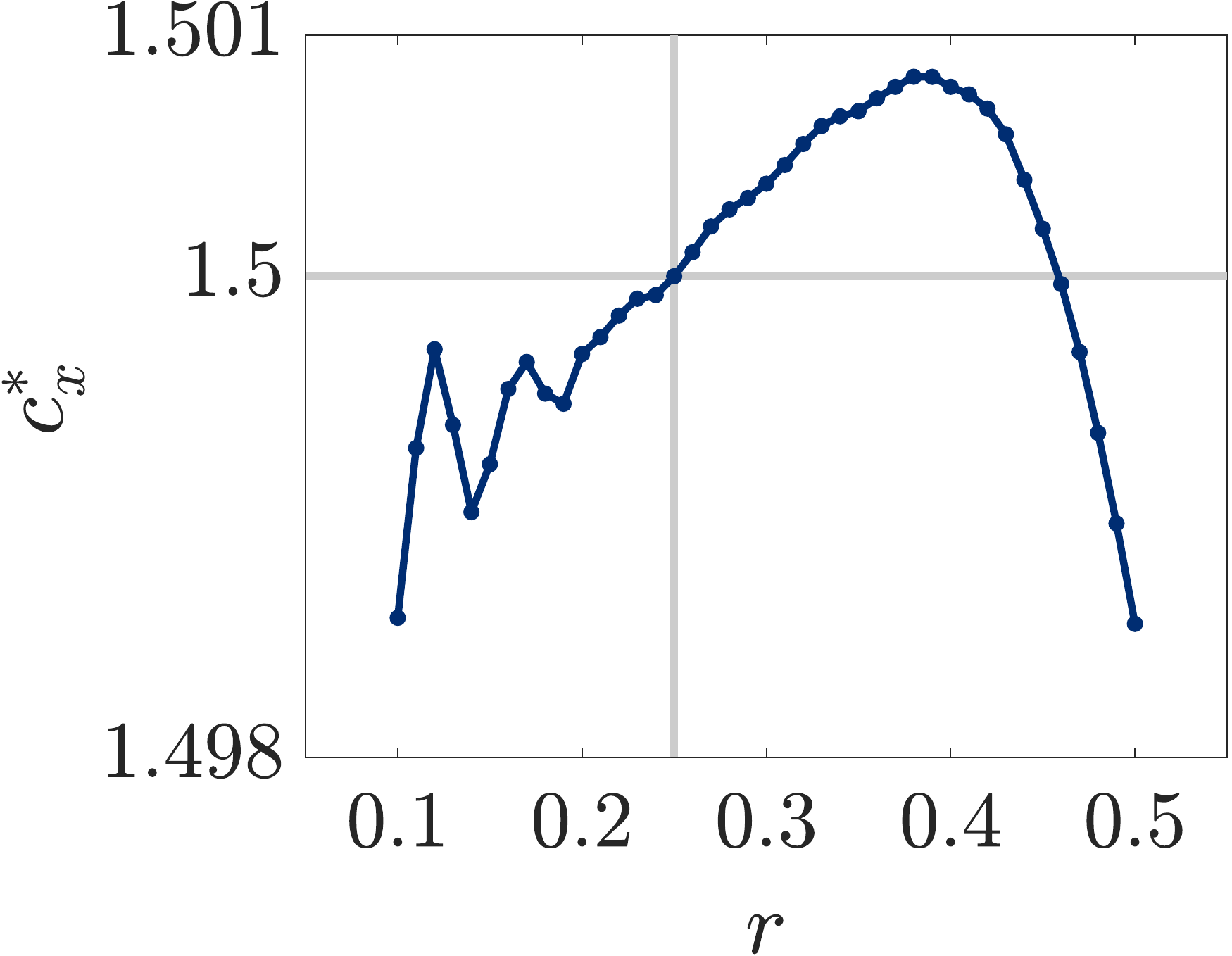}
		\ \ \ 
		\includegraphics[height = 100pt]{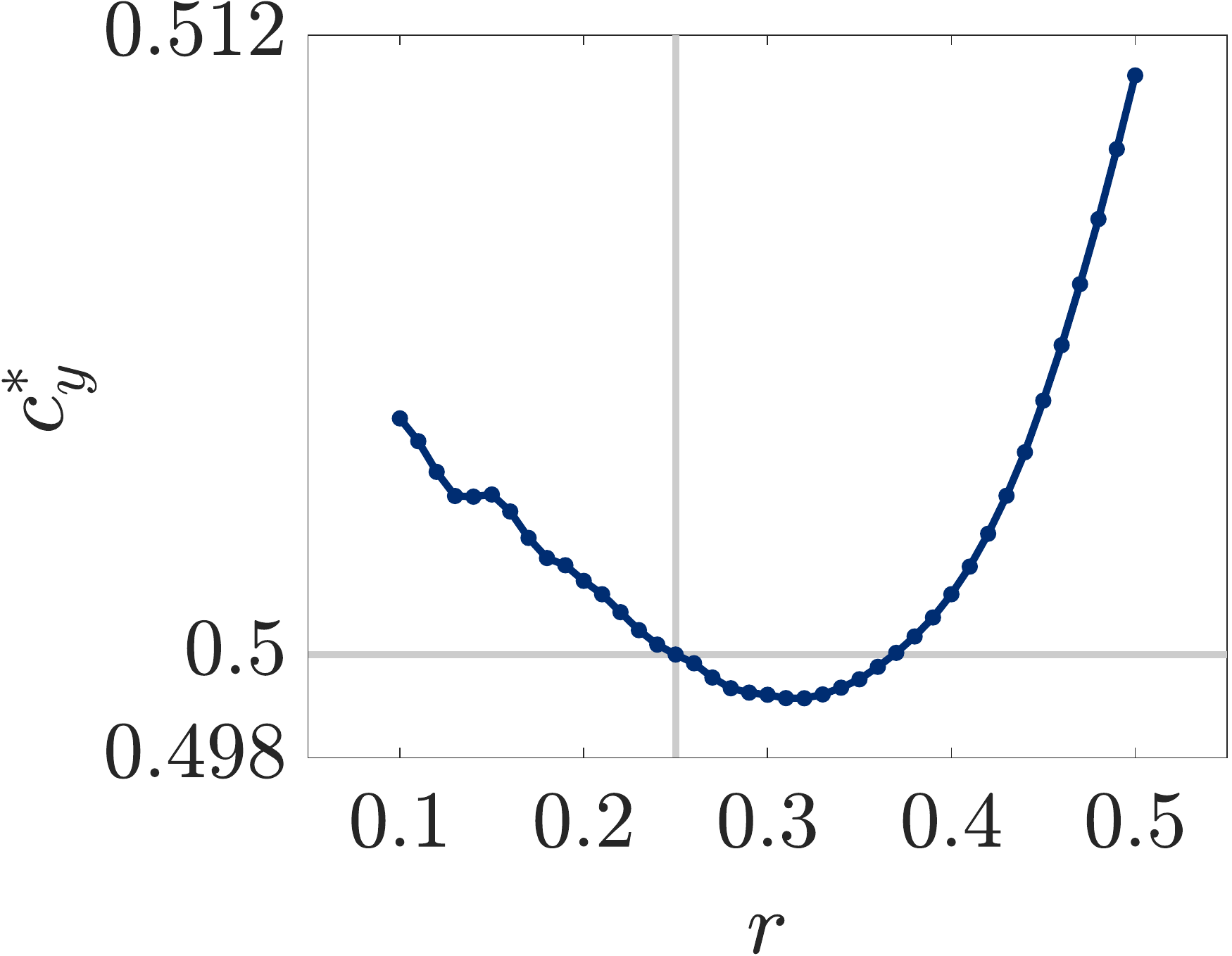}
		\ \ \ 
		\includegraphics[height =  97pt]{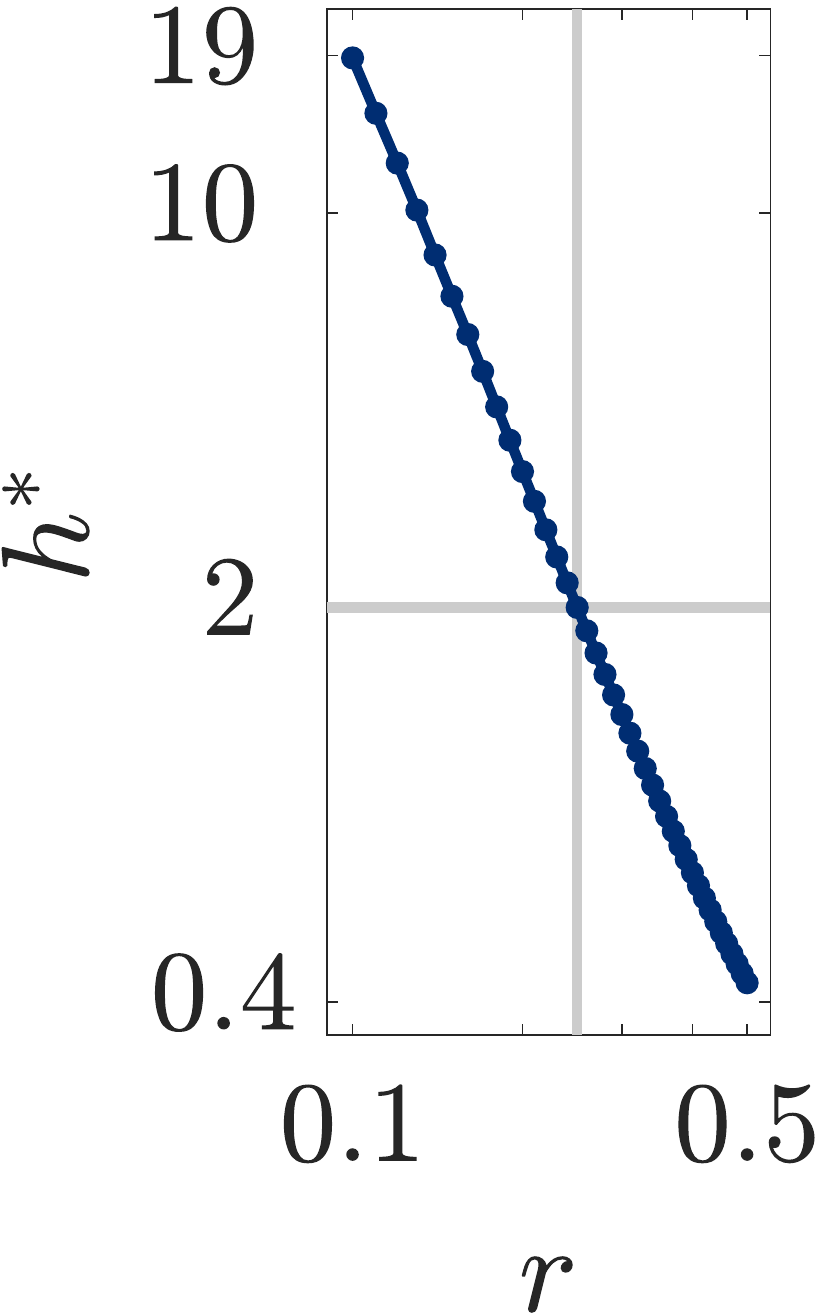}
	\caption{Sensitivity of minimizers with respect to the radius of potential. Gray lines indicate the true parameters. The slope of the log-log plot on the right is $-2.37$.}
	\label{fig:sin_sensitivity_radius}
\end{figure}

\begin{figure}[hbt!]
	\centering
	\begin{subfigure}[b]{0.325\textwidth}
		\centering
		\includegraphics[width = \textwidth]{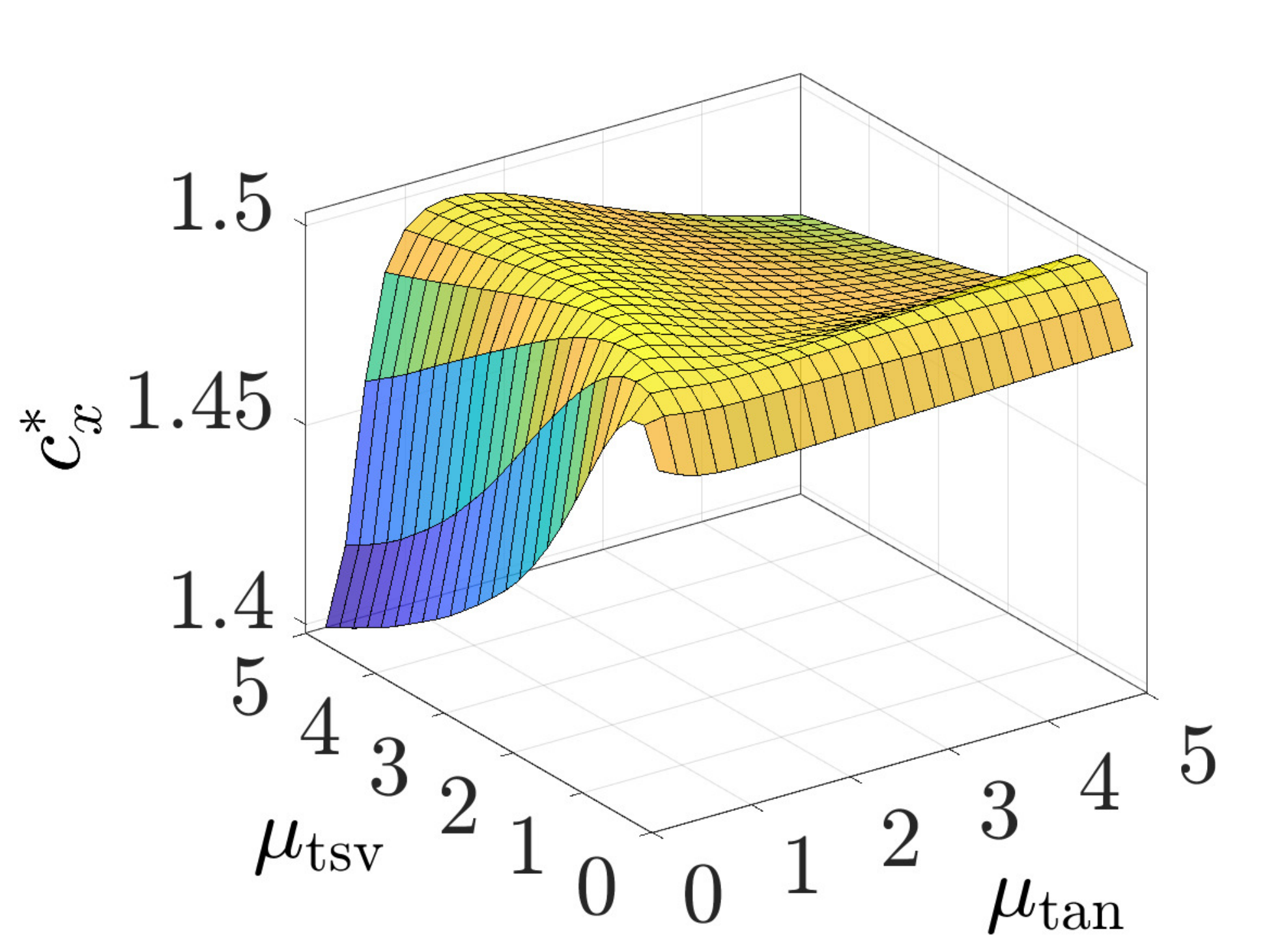}
	\end{subfigure}
	\begin{subfigure}[b]{0.325\textwidth}
		\centering
		\includegraphics[width = \textwidth]{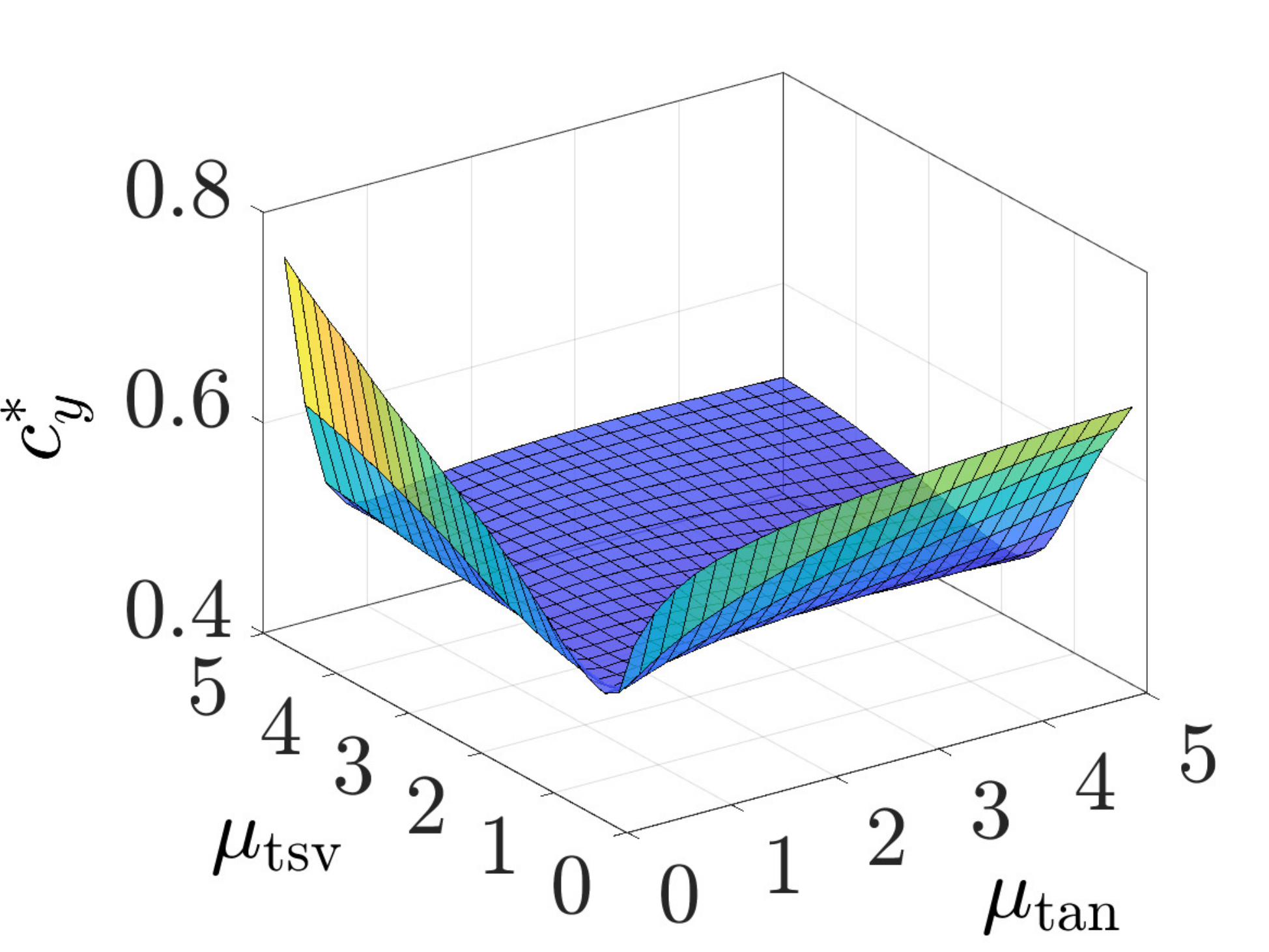}
	\end{subfigure}
	\begin{subfigure}[b]{0.325\textwidth}
		\centering
		\includegraphics[width = \textwidth]{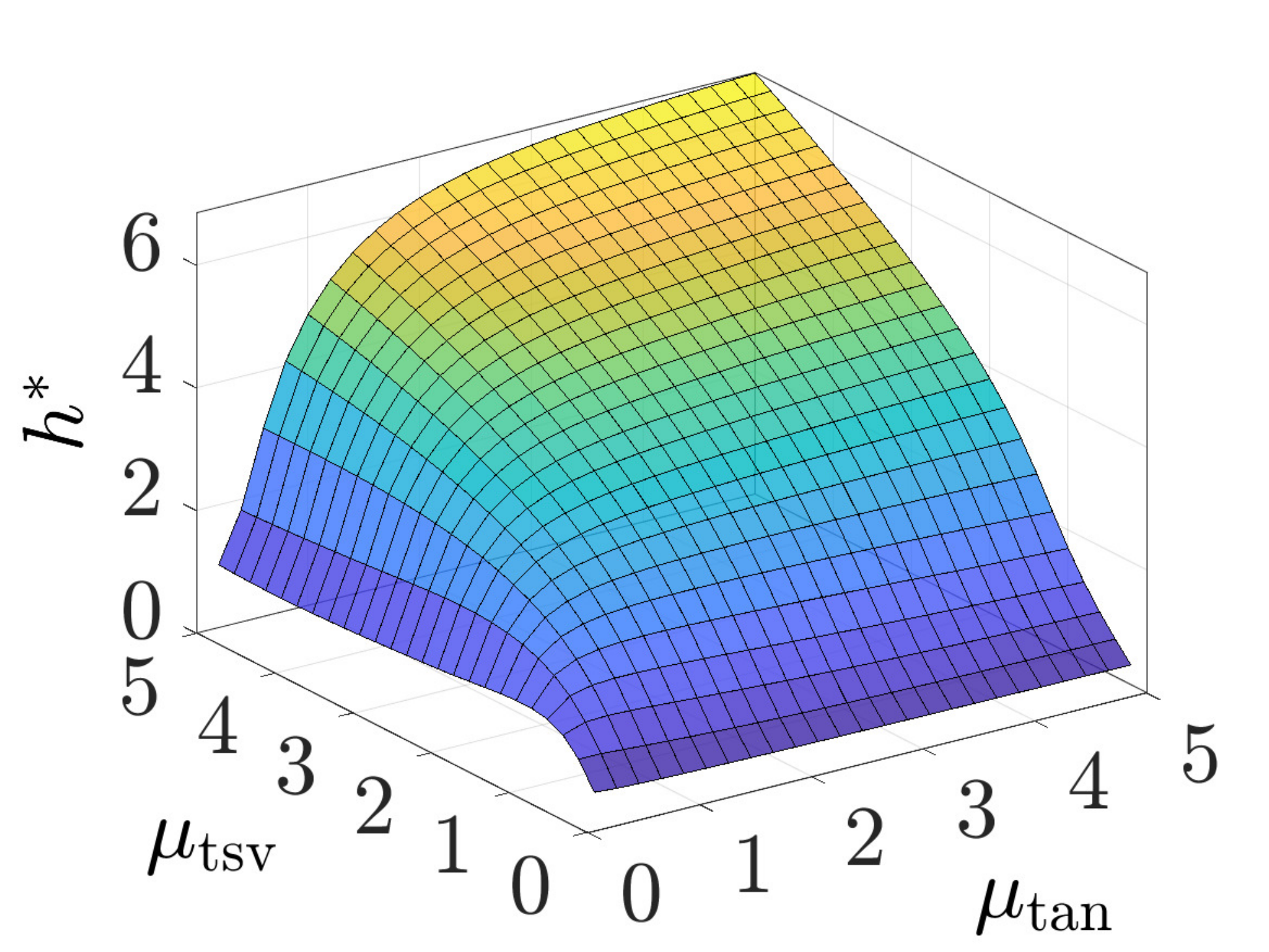}
	\end{subfigure}
	\caption{Sensitivity of minimizers with respect to elastic parameters.}
	\label{fig:sin_sensitivity_elastic}
\end{figure}

\subsection{3D real data}
We now propose an experiment using 3D data derived from the BIOCARD dataset \cite{miller2015amygdalar}, which is a longitudinal study of Alzheimer's disease. More precisely, the template and target shown in \cref{fig:real.data} were obtained by computing shape averages \cite{ma2010bayesian} of scans of the entorhinal cortex of subjects diagnosed with mild cognitive impairment in the cohort, using scans at the beginning of the study for the template, and after ten years of study for the target (the study is still ongoing with new scans being acquired). Participants enrolled in the BIOCARD cohort were all cognitively normal when the MRI scans were first acquired so that any observed atrophy in these brain volumes  among those who progress to cognitive impairment provides significant information.  

The layered structure on the source volume was inferred using the algorithm defined in \cite{ratnanather20183d,younes2019normal}, which uses a normal propagation scheme between the lower and upper surfaces delimiting the shapes.  
The initial potential function estimated in this experiment is a sum of two compactly-supported functions such as defined in equation \cref{eq:g.theta}. \Cref{fig:real.data.results} summarizes the result that were obtained, with the location of the estimated potential and the resulting deformation. Note that these results are only provided here as an illustration of the proposed method, and we do not attempt to provide any new explanation yet on the pathogenesis of the disease. We hope, however, that this method may lead to new developments in this context in future work.

\begin{figure}[hbt!]
	\centering
	\begin{subfigure}{0.48\textwidth}
		\centering
		\includegraphics[trim = 0 40 0 50, clip, width = \textwidth]{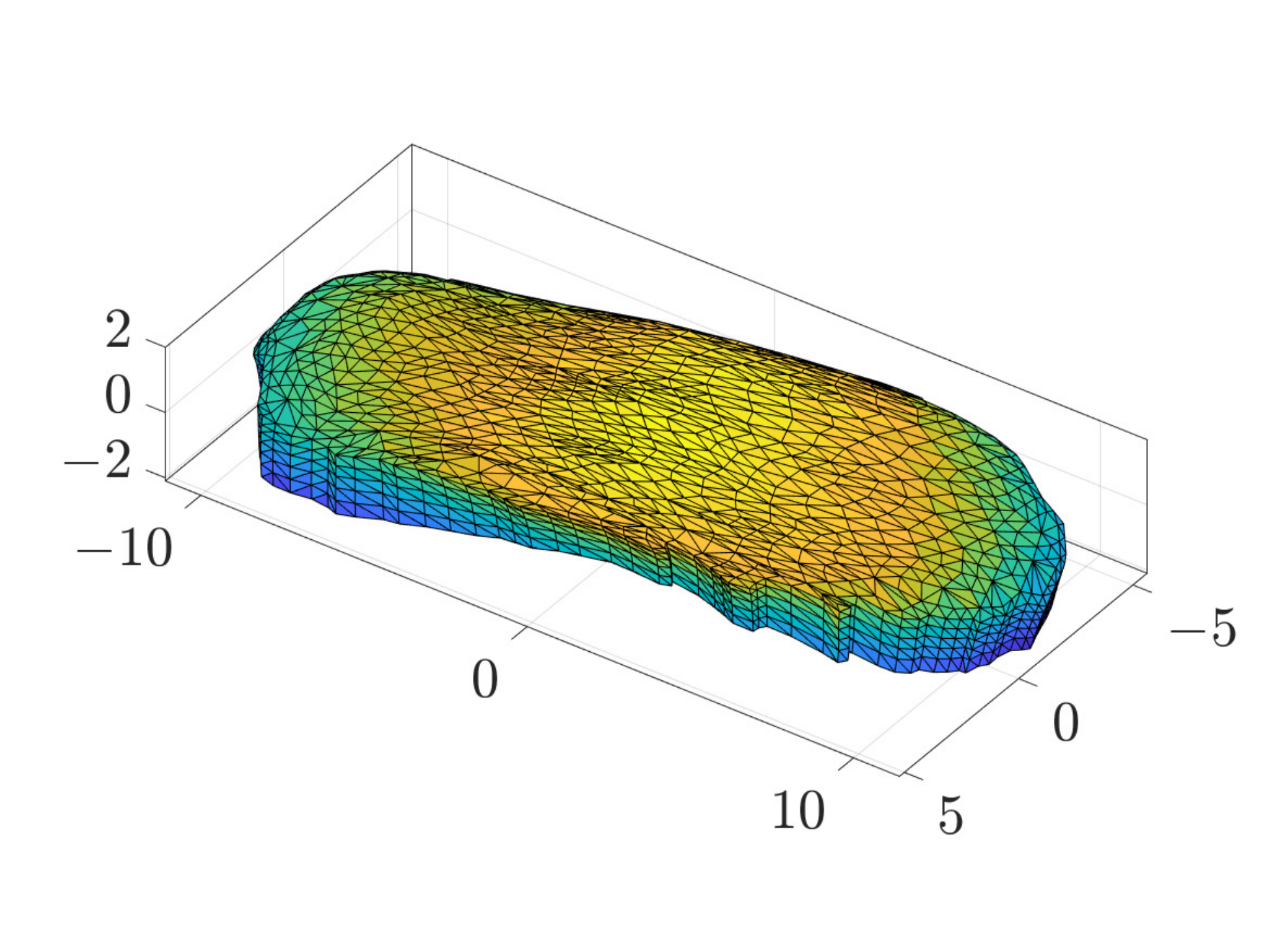}
		\caption{Template.}
	\end{subfigure}
	\begin{subfigure}{0.48\textwidth}
		\centering
		\includegraphics[trim = 0 40 0 50, clip, width = \textwidth]{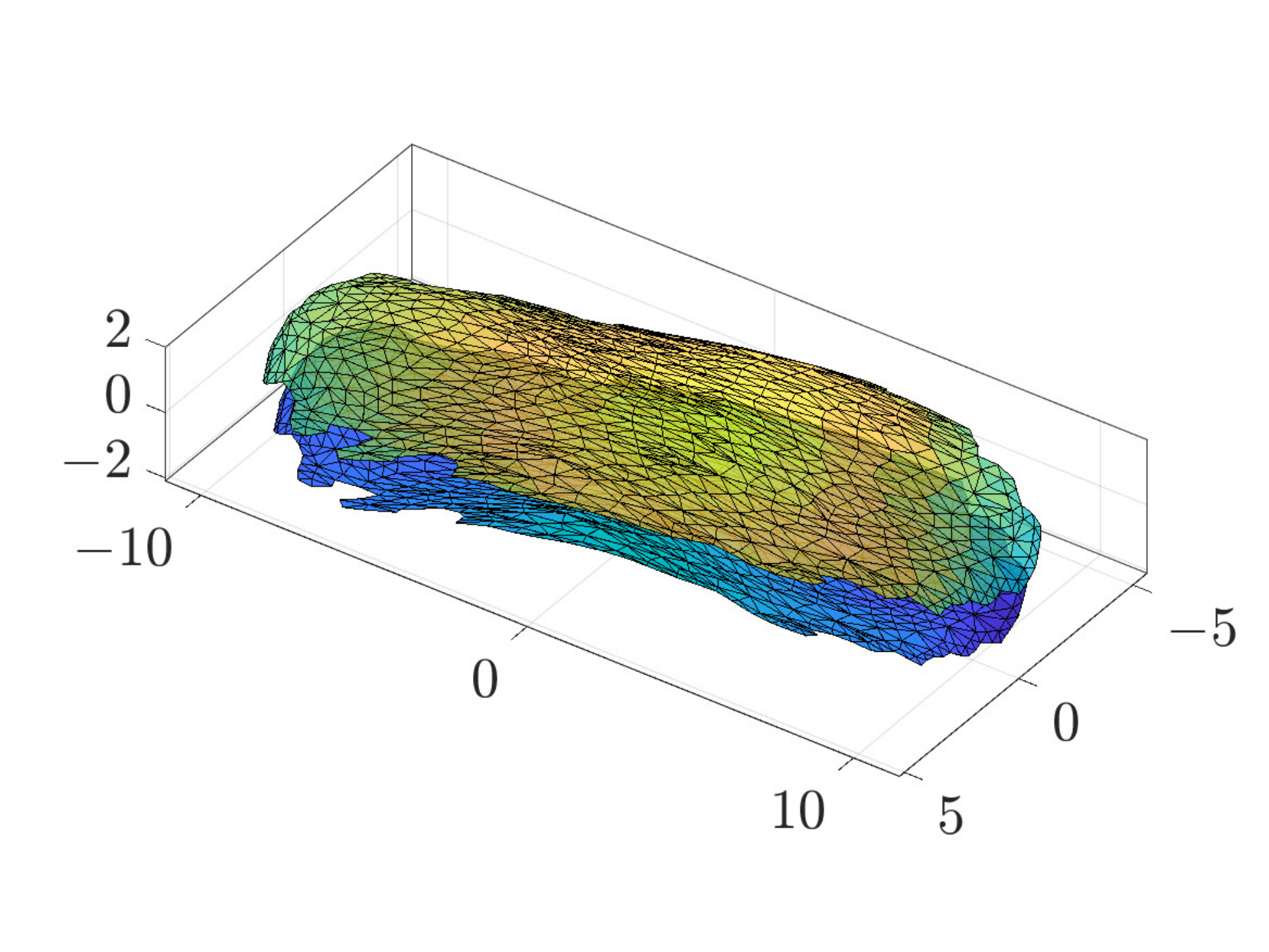}
		\caption{Target.}
	\end{subfigure}
	\\[15pt]
	\begin{subfigure}{\textwidth}
		\centering
		\includegraphics[trim = 0 20 0 10, clip, width = 0.48\textwidth]{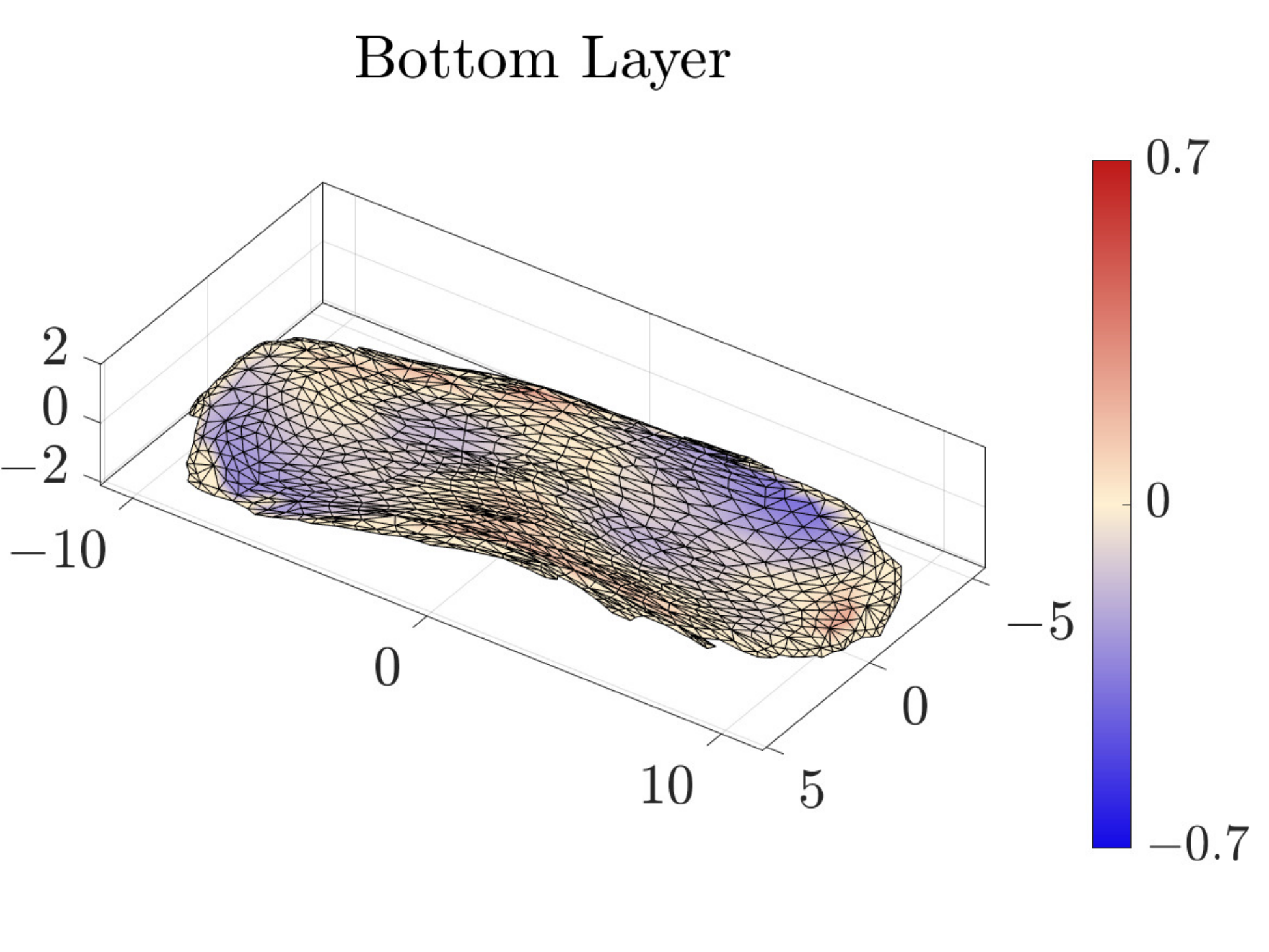}
		\includegraphics[trim = 0 20 0 10, clip, width = 0.48\textwidth]{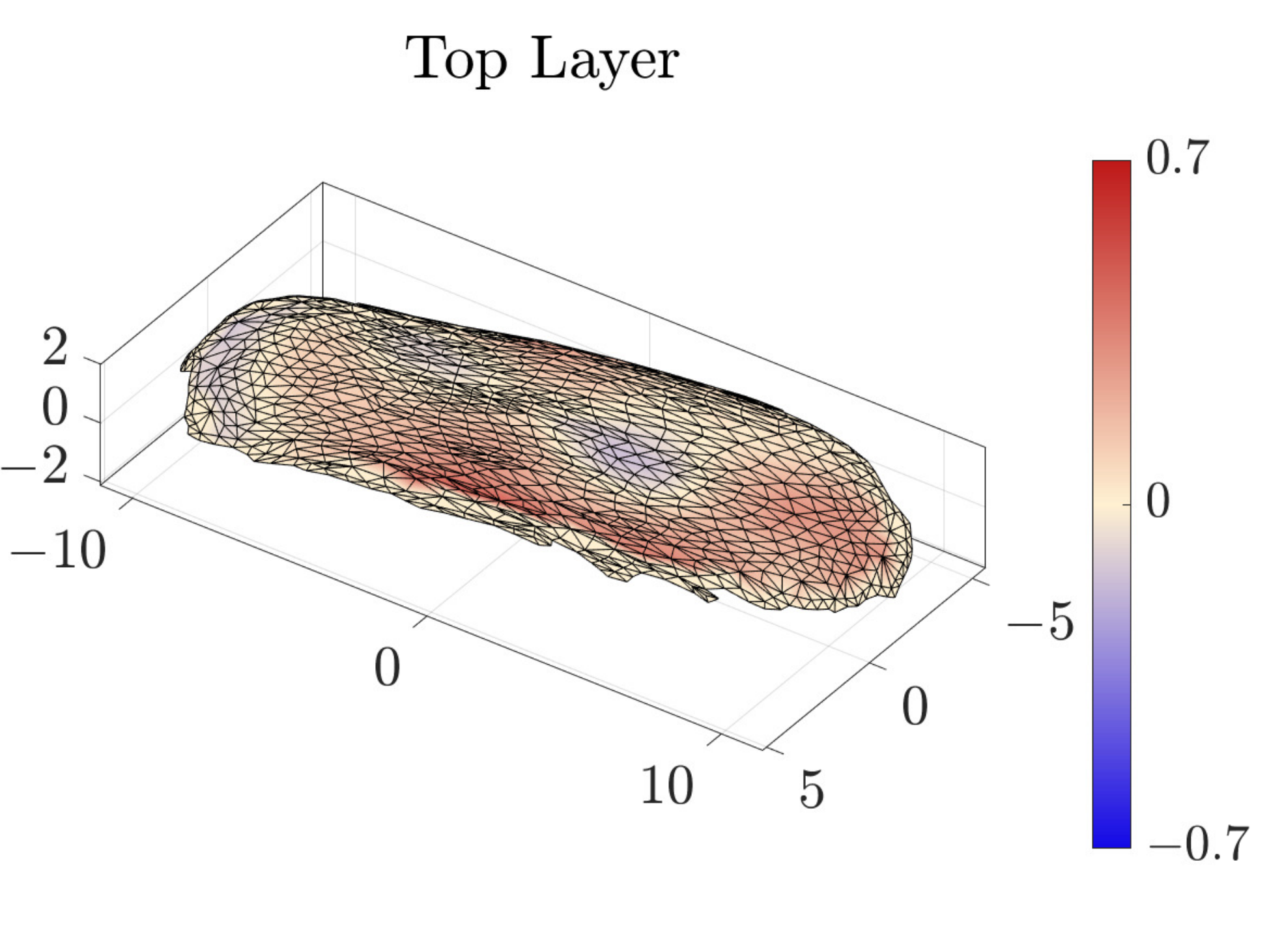}
		\caption{Comparison between the template and target. The figures show the bottom and top layer of the template, and the color represents the $z$-coordinate of the template minus the one of the target.}
	\end{subfigure}
	\caption{\label{fig:real.data} Entorhinal cortex volumes averaged over multiple patients from the BIOCARD dataset. }
\end{figure}
\begin{figure}[hbt!]
	\begin{subfigure}{\textwidth}
		\centering
		\includegraphics[trim = 20 60 0 60, clip, width = 0.48\textwidth]{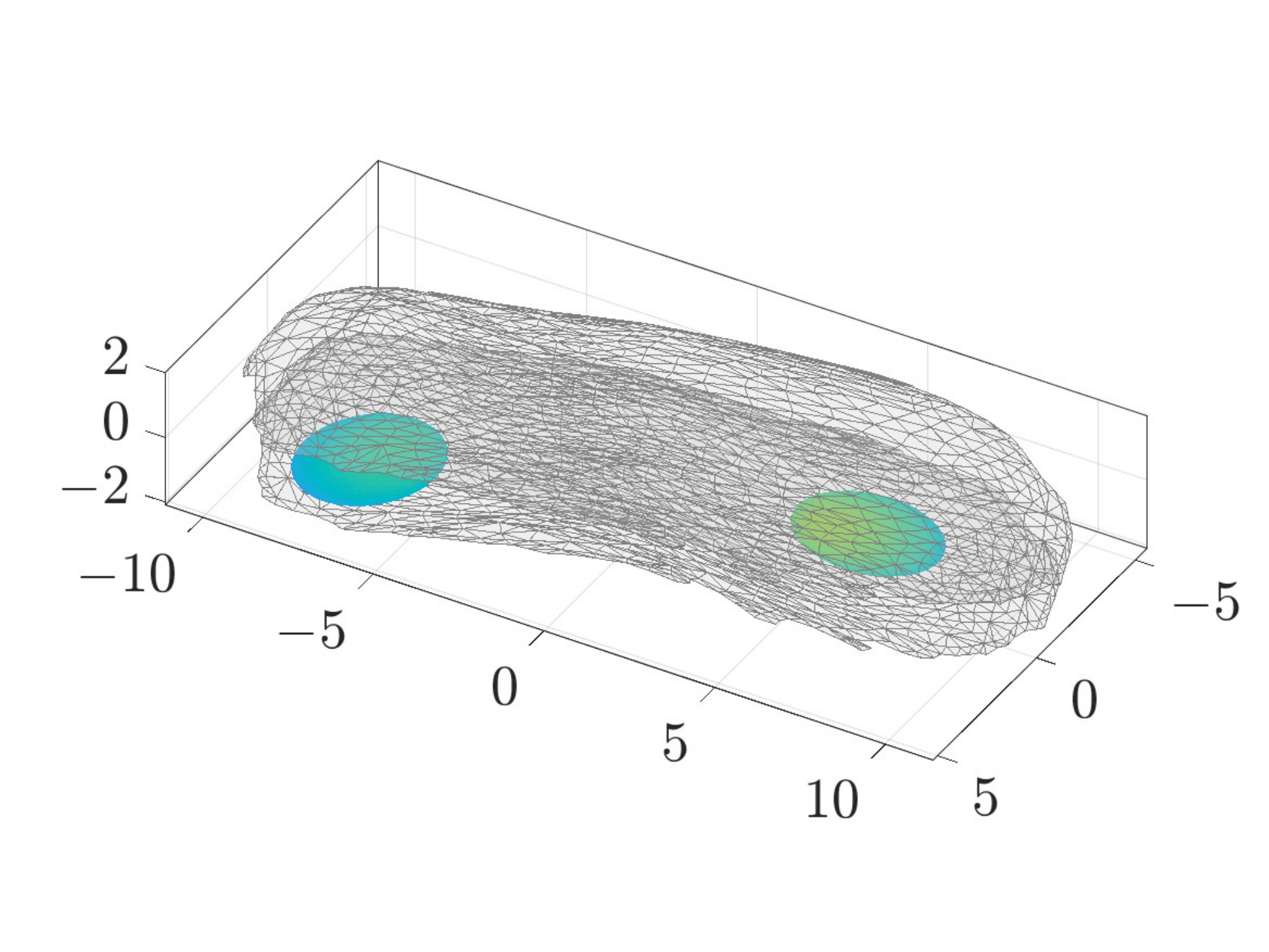}
		\includegraphics[trim = 20 60 0 60, clip, width = 0.48\textwidth]{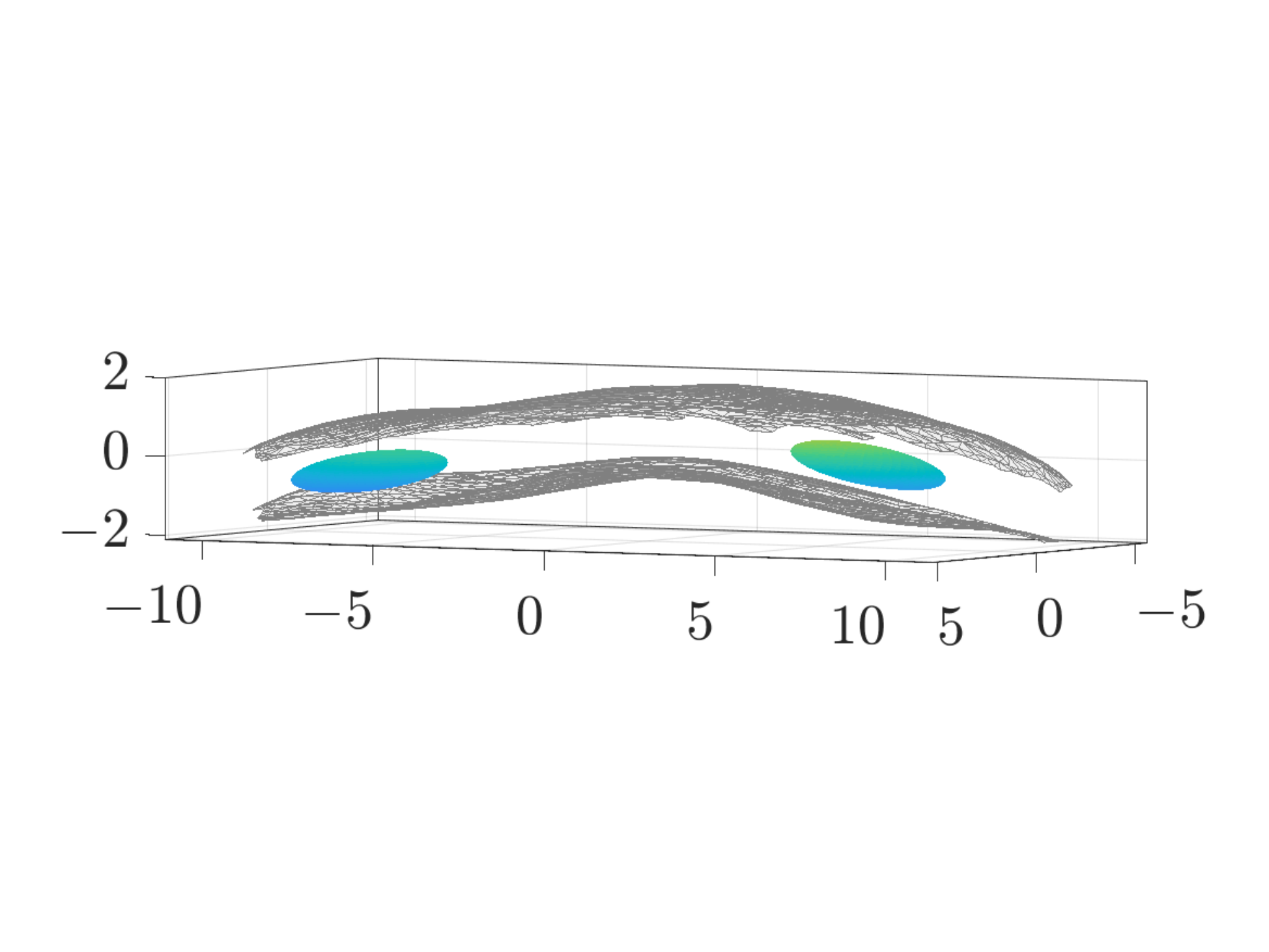}
		\caption{Support of potential.}
	\end{subfigure}
	\\[15pt]
	\begin{subfigure}{\textwidth}
		\centering
		\includegraphics[trim = 0 40 110 0, clip, width = 0.27\textwidth]{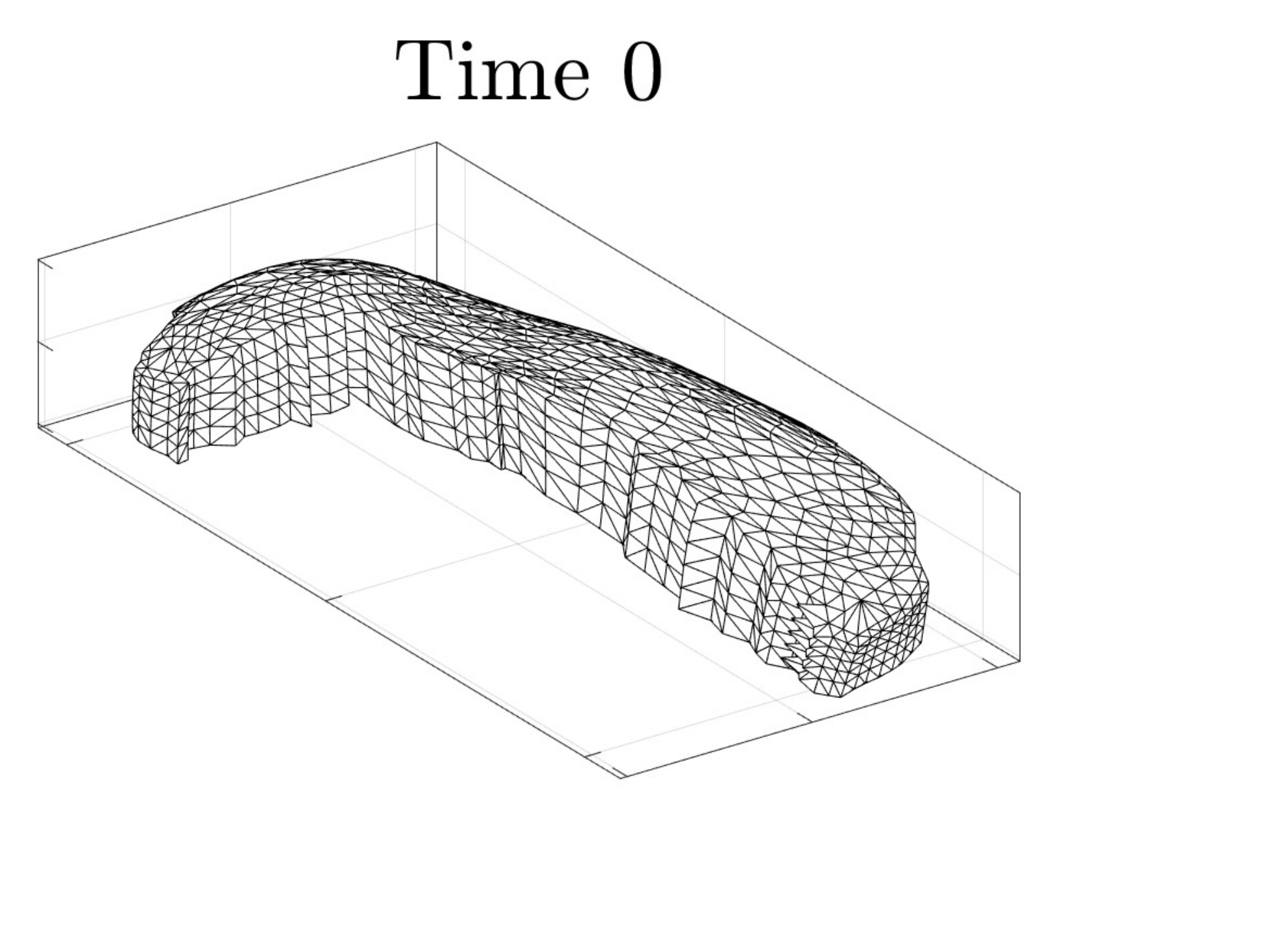}
		\includegraphics[trim = 0 40 110 0, clip, width = 0.27\textwidth]{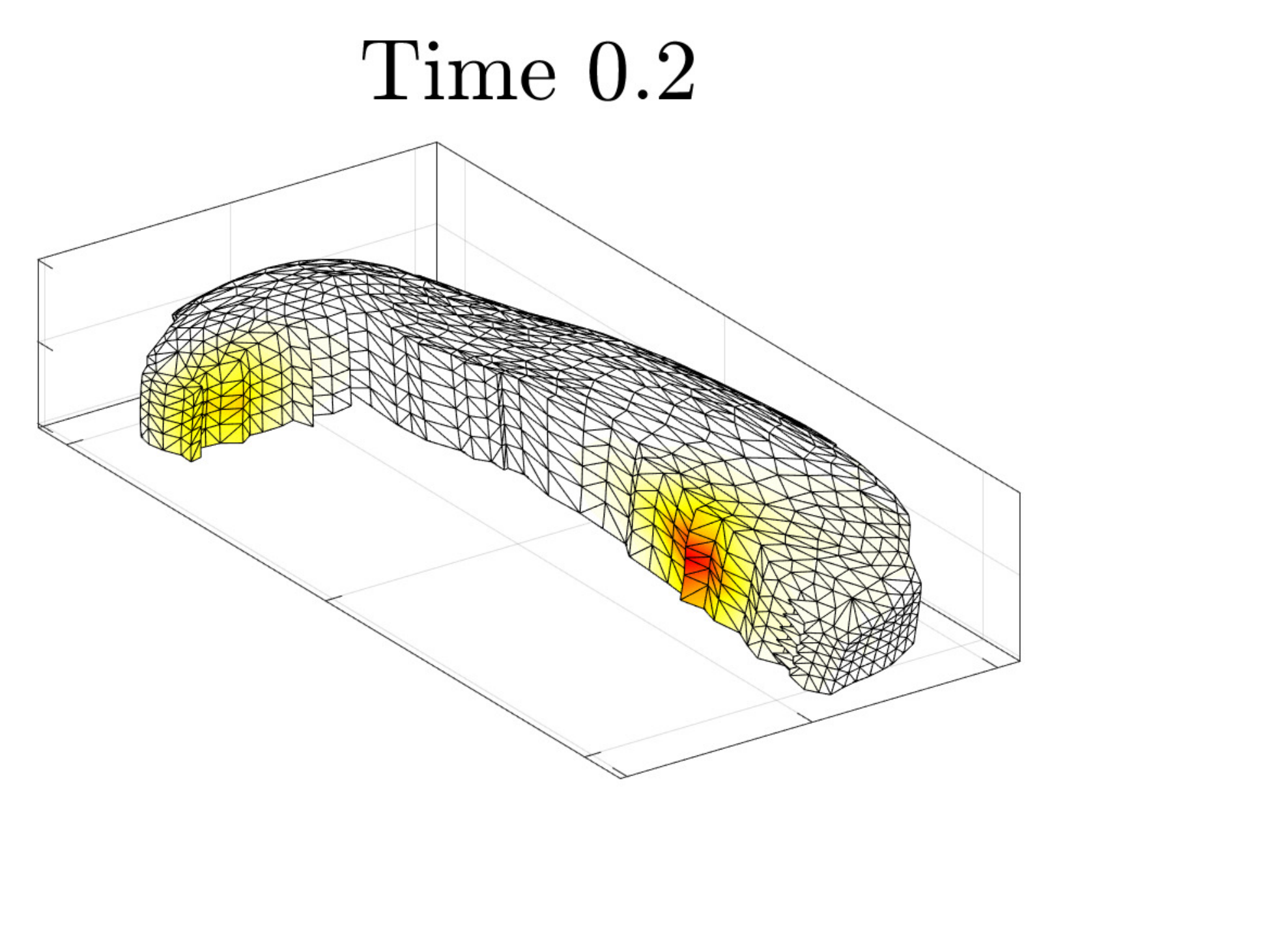}
		\includegraphics[trim = 0 40  20 0, clip, width = 0.32\textwidth]{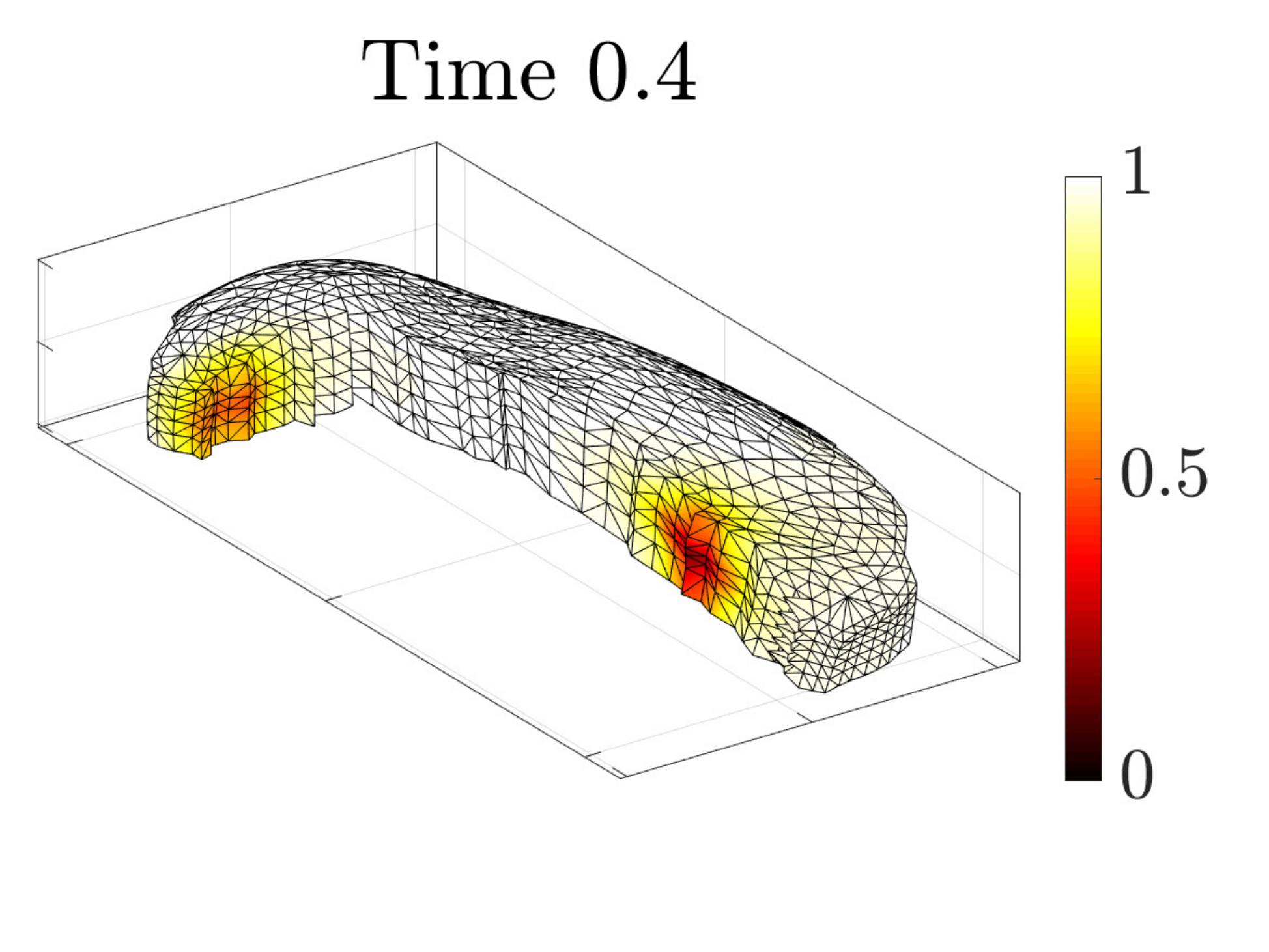}
		\\
		\includegraphics[trim = 0 40 110 0, clip, width = 0.27\textwidth]{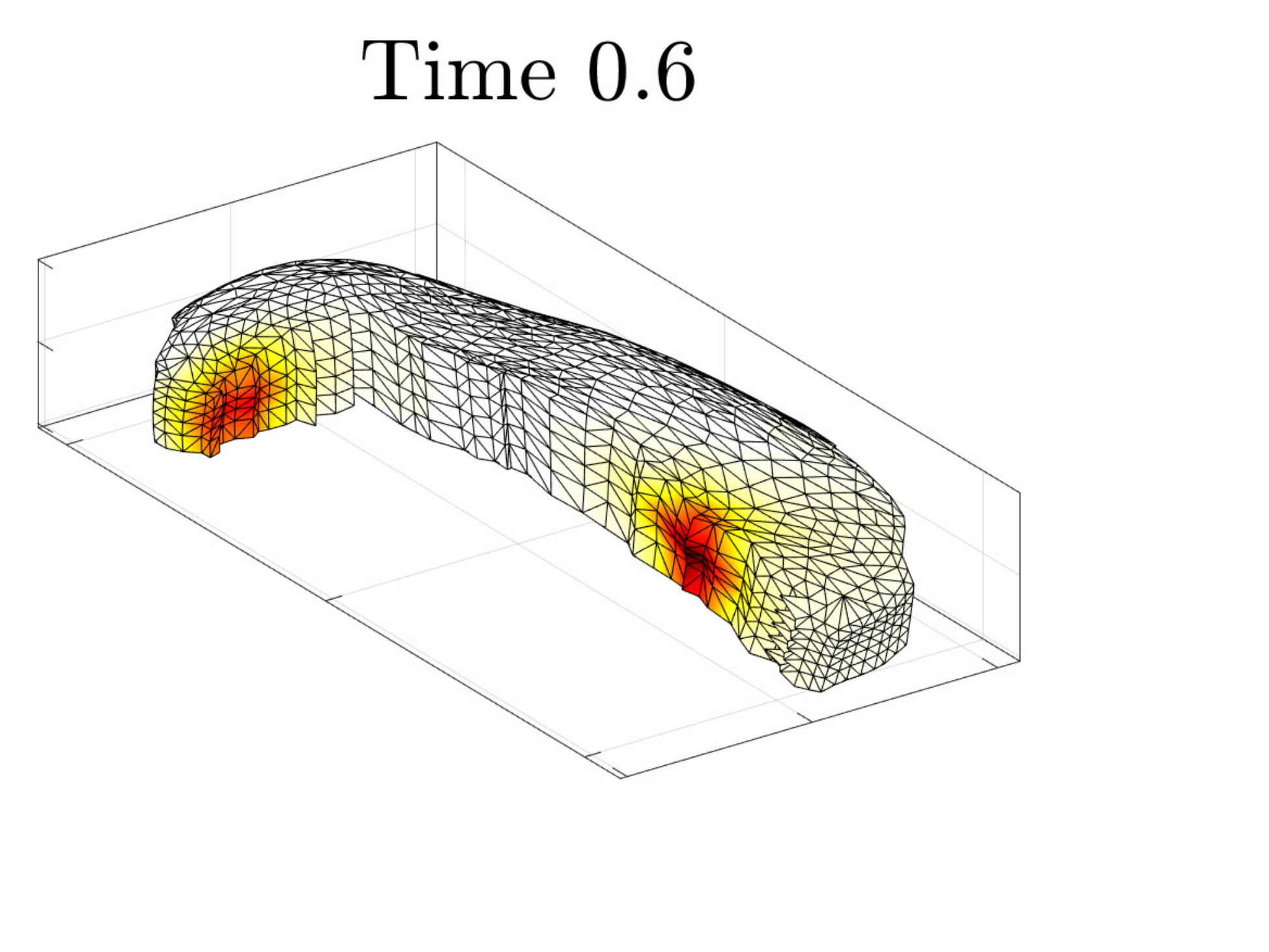}
		\includegraphics[trim = 0 40 110 0, clip, width = 0.27\textwidth]{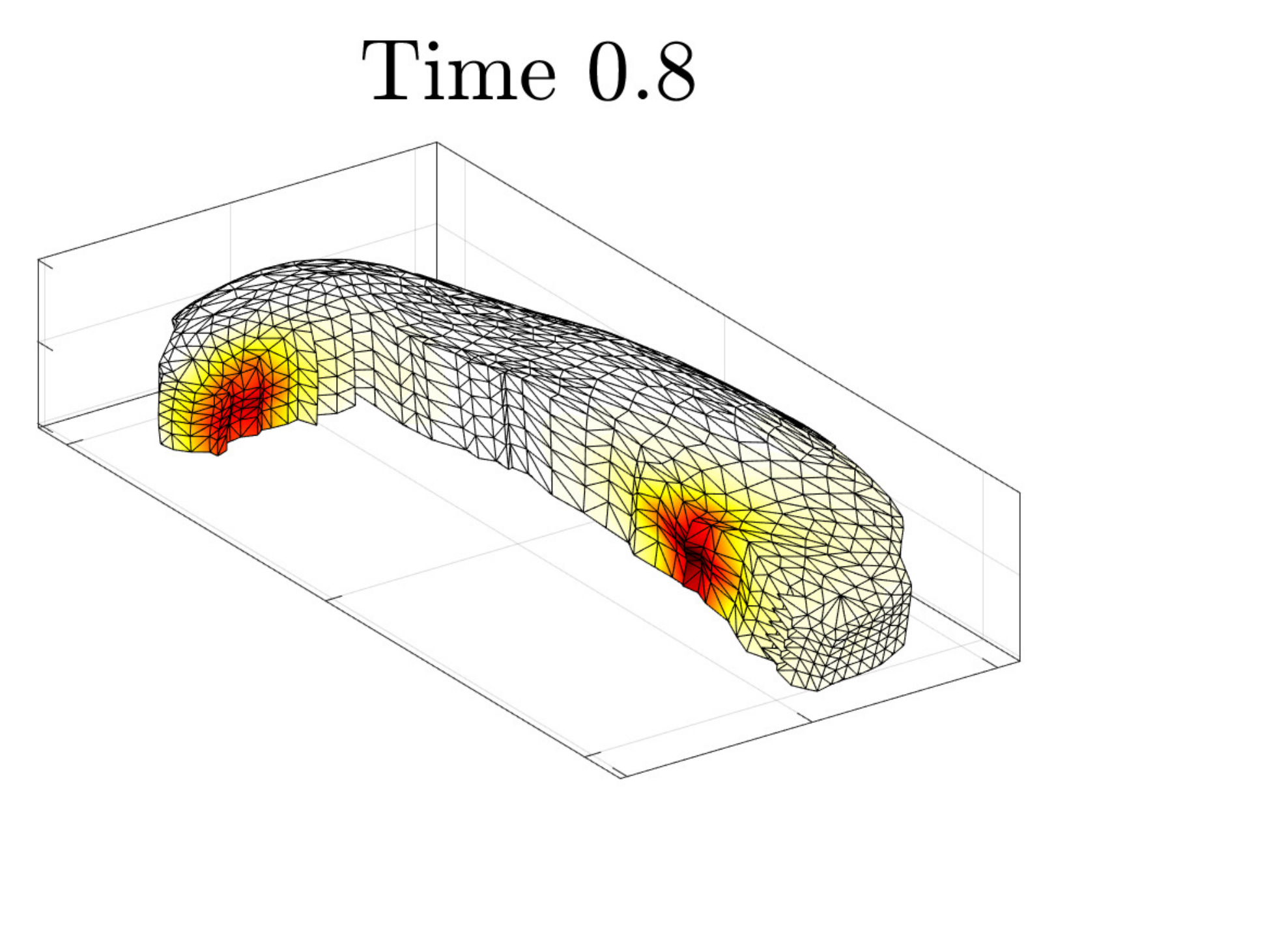}
		\includegraphics[trim = 0 40  20 0, clip, width = 0.32\textwidth]{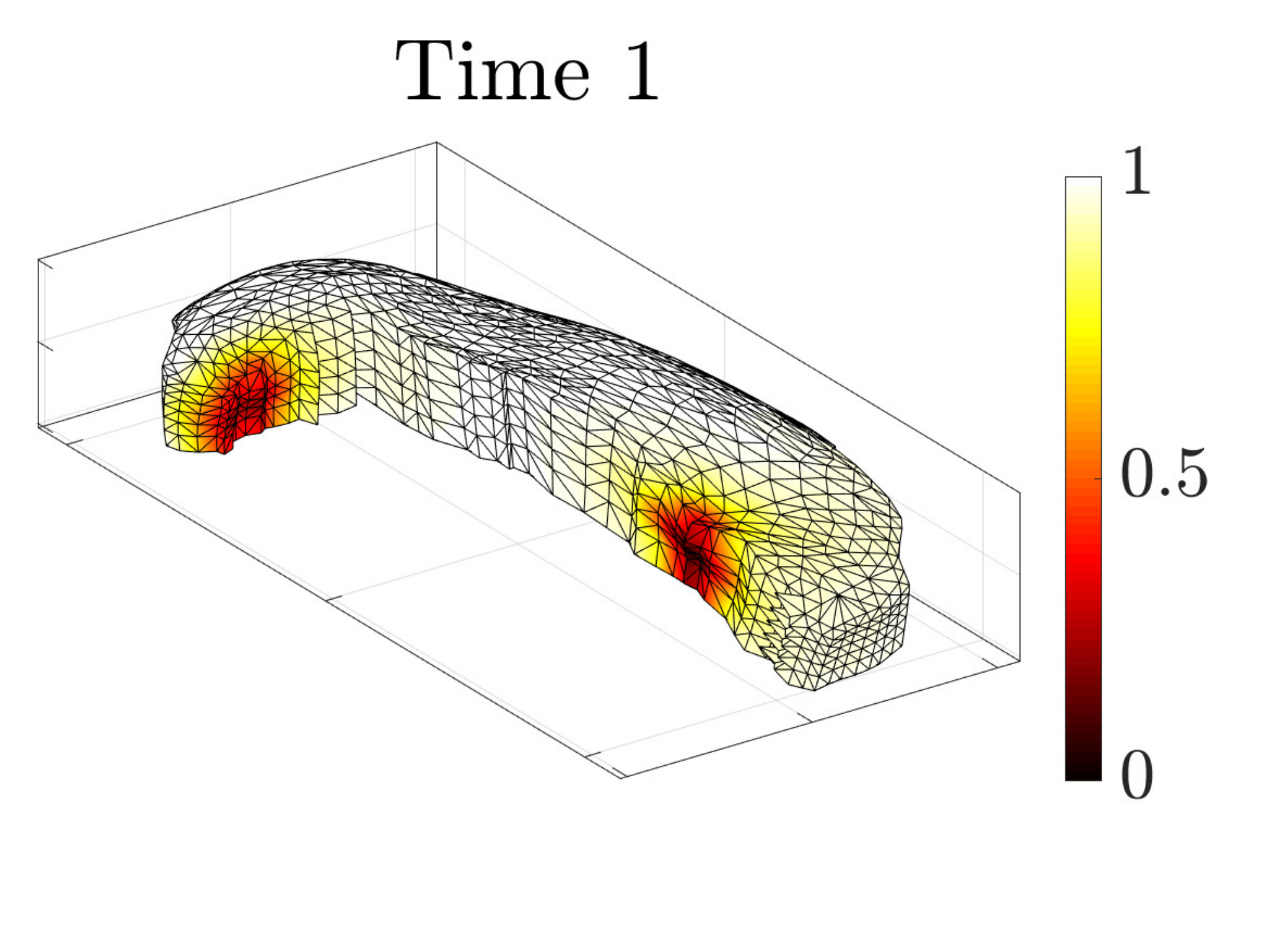}
		\caption{Deformation. The color represents the Jacobian $\det D\varphi(t)$.}
	\end{subfigure}
	\\[15pt]
	\begin{subfigure}{\textwidth}
		\centering
		\includegraphics[trim = 0 20 0 10, clip, width = 0.48\textwidth]{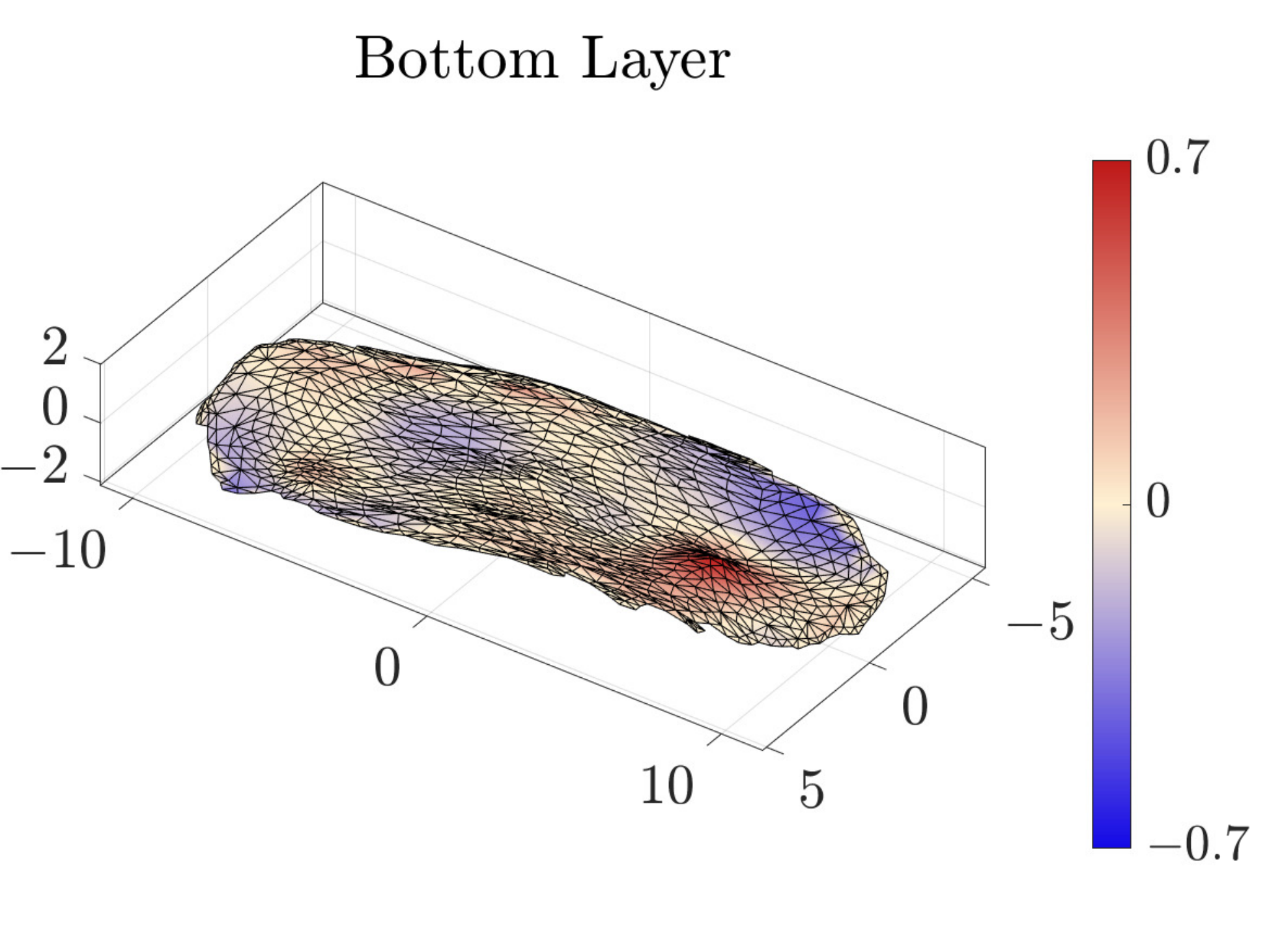}
		\includegraphics[trim = 0 20 0 10, clip, width = 0.48\textwidth]{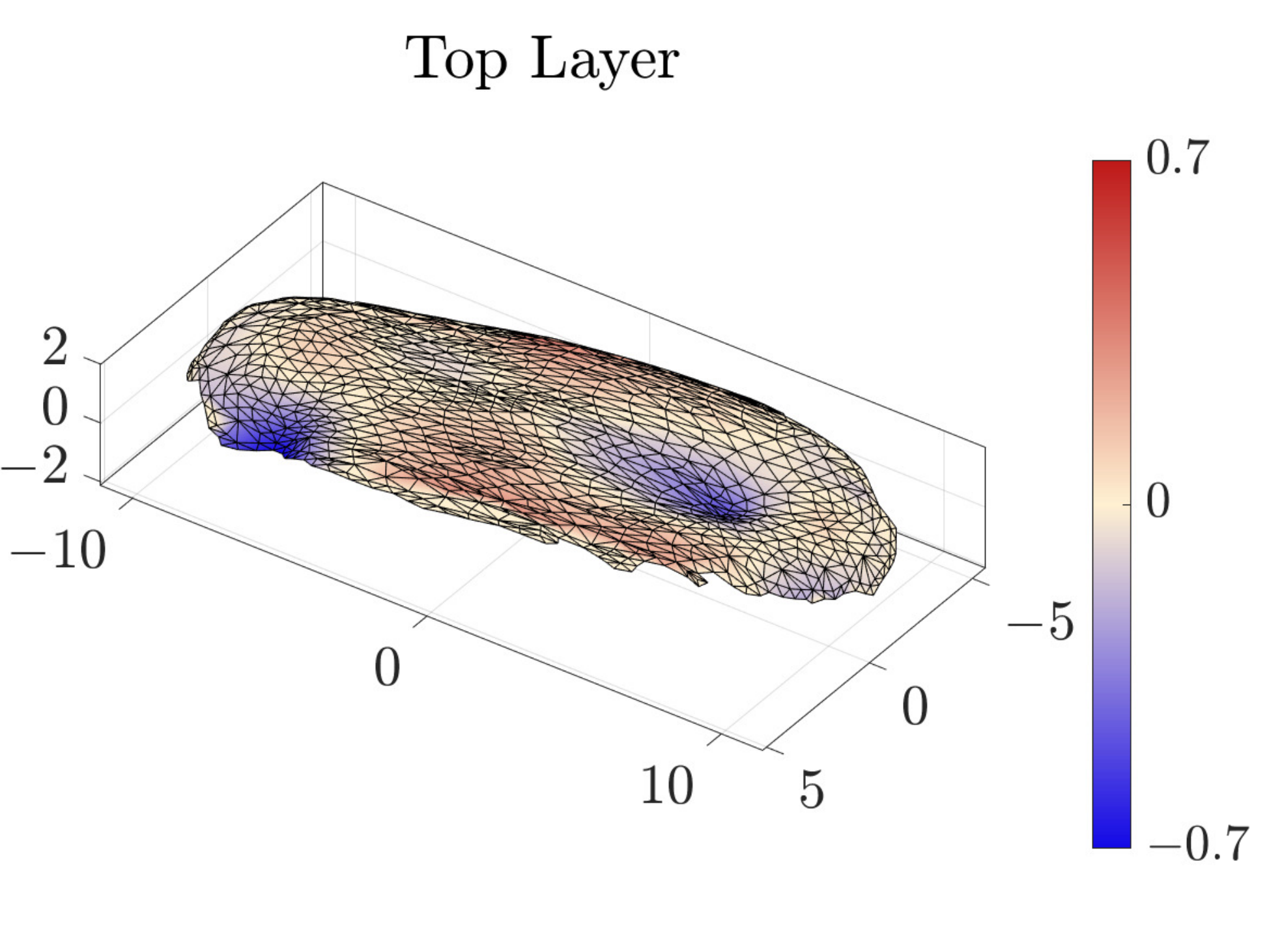}
		\caption{Registration result. The figures show the bottom and top layer of the deformed template, and the color represents the $z$-coordinate of the deformed template minus the one of the target.}
	\end{subfigure}
	\caption{\label{fig:real.data.results} Experimental results on averaged entorhinal cortex volumes.}
\end{figure}

\section{Proofs}
\label{sec:proofs}
We now prove  \cref{thm:ctrl_problem,thm:opt_problem}.
The two proofs being similar, it will be convenient to address together parts (i) of both theorems, and then parts (ii) together.

\subsection{Existence and uniqueness of solutions of controlled ODEs}
\label{sec:proofs.ode}
We will use the following version of local existence and uniqueness for controlled ordinary differential equations. In the following, we will say that a function $u$ defined on $[0,T]$ with values in a metric space $U$ is admissible if there exists a sequence of simple functions $u_n: [0,T] \to U$ such that $u_n(t)$ converges to $u(t)$ for almost every $t \in [0, T]$.
\begin{theorem}
\label{th:ode.control}
Let $U$ be a metric space, and let $O$ be an open subset of a Banach space $\mB$.
Let $F: U\times O \to \mB$ be a continuous function that satisfies the following properties:
\begin{enumerate}[label=(\arabic*),ref = Theorem \ref{th:ode.control}(\arabic*)]
\item
\label{th:ode.control:1}
There exists a function $\gamma: U \to (0, \infty)$ such that, for all $u\in U$, $y\in O$:
\[
\|F(u,y)\|_\mB \leq \gamma(u) \, (1 + \|y\|_\mB)\,.
\]
\item 
\label{th:ode.control:2}
For all $y_0\in O$, there exists $r_{y_0}>0$ and a function $\gamma_{y_0}: U \to (0, +\infty)$ such that for all $y,y'\in O$ with $\max(\|y-y_0\|_\mB, \|y'-y_0\|_\mB)\leq r_{y_0}$ and all $u\in U$
\begin{equation}
    \label{eq:loc.lip}
\|F(u,y) - F(u, y')\|_\mB \leq \gamma_{y_0}(u) \, \|y-y'\|_\mB\,.
\end{equation}

\end{enumerate}
Fix $T>0$ and
let $u: [0,T] \to U$ be admissible such that (i) $\gamma(u(t))$ is integrable on $[0,T]$ and (ii) for all $y_0\in O$, $\gamma_{y_0}(u(t))$ is integrable on $[0,T]$. 
Then, for all $y_0\in O$, there exists a largest $T_0\leq T$ such that the ODE
\[
\partial_t y(t) = F(u(t), y(t))
\]
has a unique solution on $[0,T_0)$ satisfying $y(0) = y_0$.\\
In addition, one has
\begin{equation}
    \label{eq:ode.control.bound}
\sup_{t \,\in\, [0, T_0)} \|y(t)\|_\mB \leq
C \exp\left(\int_0^{T} \gamma(u(s)) ds\right).
\end{equation}
where $C = \|y_0\|_{\mB} + \int_0^T \gamma(u(s)) ds$ is a constant that only depends on $y_0$ and on the control $u$.  Moreover, 
$y(t)$ has a limit in $\mB$ when $t$ tends to $T_0$ and, if $T_0<T$, then
$\lim_{t\to T_0} y(t) \not\in O$.

\end{theorem}

\begin{remark}
By a solution of $\partial_t y(t) = F(u(t), y(t))$ over $[0, T_{0})$ we mean a continuous function $y: [0. T_0) \to O$ satisfying
\[
y(t) = y_{0} + \int_{0}^{t} F(u(s), y(s)) ds
\]
for all $t\in [0, T_{0})$. The integral on the right-hand side is the Bochner integral. The fact that  $s \mapsto F(u(s), y(s))$ is Bochner integrable  is always true under the assumptions of \cref{th:ode.control}, the proof being left to the reader.

Note that, letting $y(T_0) = \lim_{t\to T_0} y(t)$, we have, taking limits on both sides,  
\[
y(T_0) = y_{0} + \int_{0}^{T_0} F(u(s), y(s)) ds
\]
\end{remark}

\begin{proof}
The proof of the theorem is given here for completeness, and because the statement slightly deviates from that usually found in standard textbooks.

Fix $y_0\in O$ and $r_{y_0}, \gamma_{y_0}$ such that \cref{eq:loc.lip} is true. Let $\Omega$ denote the closed ball of center $y_0$ and radius $r_{y_0}$ in $\mB$. Reducing $r_{y_0}$ if needed, assume that $\Omega\subset O$. Introduce a small $\eta>0$, let $I_\eta = [0,\eta]$ and let $\mathcal B_{\eta}$ denote the complete metric space of continuous functions $y: I_\eta \to \Omega$ equipped with the metric $d(y, y') = \sup_{t \,\in\, [0, \eta]} \|y(t) - y'(t)\|_{\mB}$. On this space, define, for $t\in I_\eta$,
\[
\Gamma(y)(t) = y_0 + \int_0^t F(u(s), y(s)) ds.
\]
Then, if $y\in \mathcal{B}_{\eta}$,
\[
\|\Gamma(y)(t) - y_0\|_\mB \leq (1+ \|y_0\|_{\mB} + r_{y_0}) \int_0^\eta \gamma(u(s)) ds
\]
so that $\Gamma$ maps $\mathcal{B}_{\eta}$ onto itself as soon as $\int_0^\eta \gamma_0(u(s)) ds\leq r_{y_0}/(1+ \|y_0\|_{\mB} + r_{y_0})$. Moreover, for all $t\in I_\eta$ and all $y,y' \in \mathcal B_\eta$,
\[
d(\Gamma(y) - \Gamma(y')) \leq \left(\int_0^\eta \gamma_{y_0}(u(s))ds\right) d(y, y'),
\]
which shows that $\Gamma$ is a contraction if  $\int_0^\eta \gamma_{y_0}(u(s))ds<1$. Taking $\eta$ small enough, this shows that $\Gamma$ has a unique fixed point in $\mathcal B_\eta$ and therefore proves local existence and uniqueness.

Now consider a solution $y$ defined on $[0, T_0)$ with $T_0\leq T$. For all $0\leq t<T_0$:
\begin{align*}
 \|y(t)\|_\mB &\leq \|y_0\|_\mB +  \int_0^{t} \gamma(u(s)) (1+\|y(s)\|_\mB) ds \\
 &\leq \|y_0\|_\mB + \int_0^T \gamma(u(s)) ds + \int_0^t \gamma(u(s)) \|y(s)\|_\mB ds ,
\end{align*}
which by Gronwall's inequality leads to 
\[
 \|y(t)\|_\mB \leq \left(\|y_0\|_\mB + \int_0^T \gamma(u(s)) ds \right) \exp\left(\int_0^{T} \gamma(u(s)) ds\right) 
\]
and \cref{eq:ode.control.bound} follows by taking the supremum over all $t\in[0,T_0)$.
Using this bound of solutions, we get that for all $t,t'\in [0, T_0)$:
\[
\|y(t') - y(t)\|_\mB \leq \int_t^{t'} \gamma(u(s)) (1+\|y(s)\|_\mB) ds
\leq \widetilde{C} \int_t^{t'} \gamma(u(s)) ds
\]
for some constant $\widetilde{C}$ that depends on $y_0$ and $u$. This shows 
that $y$ can be extended by continuity to $[0, T_0]$. Denote the limit by $y(T_0)$: if $T_0<T$ and $y(T_0) \in O$, the solution can be extended further and we get a contradiction to the assumption that $T_0$ defines the largest interval.
This concludes the proof of \cref{th:ode.control}.
\end{proof}
\bigskip

\subsection{Proof of \hyperref[thm:ctrl_problem_ode]{Theorem}~\ref{thm:ctrl_problem_ode}}
\label{subsec:proof_ctrl_problem_ode}
We now return to our original problem and first note, by stating the following lemma, that $v(t)$ in \cref{eq:ctrl_problem_system} is well defined.
\begin{lemma}
\label{lemma:minimizer}
If $j \in V^*$ and $\varphi \in \Diffid{1}{\mathbb{R}^3}$, then
\[
	f(v)
	=
	\frac{\reg}{2} \, \|v\|_V^2
	+
	\frac{1}{2} \, (A_\varphi \, v \mid v)
	-
	(j \mid v)
\]
has a unique minimizer $v_{\varphi, \hspace{1pt} j}$ in $V$ given by $v_{\varphi, \hspace{1pt} j} = L_\varphi^{-1} j$, where $L_\varphi = \reg K_V^{-1} + A_\varphi \in \mathscr{L}(V, V^*)$.
\end{lemma}
\begin{proof}
Since $f$ is strictly convex and differentiable with
\[
	Df(v) = (\reg K_V^{-1} + A_\varphi) \, v - j = L_\varphi v - j ,
\]
the unique minimizer is characterized by $Df(v_{\varphi, \hspace{1pt} j}) = 0$, i.e., $L_\varphi$ is invertible and $v_{\varphi, \hspace{1pt} j} = L_\varphi^{-1} j$.
\end{proof}

System \cref{eq:ctrl_problem_system} can be rewritten as 
\[
\partial_t \varphi = F(j, \varphi) := v_{\varphi,j} \circ \varphi
\]
and, letting $\varphi = \mathit{id} + h$ we want to apply \cref{th:ode.control} to the equation $\partial_t h = F(j,\mathit{id}+h)$, where we will take $U = V^*$, $\mB = C_0^1(\mR^d, \mR^d)$ and  $O = \Diffid{1}{\mR^d} - \mathit{id}$.  We will show that the conditions for local existence are satisfied, and that, for any $T_{0} \leq T$ such that a solution exists over $[0, T_{0})$, the limit $\varphi(T_{0})$ belongs to $\Diffid{1}{\mR^{d}}$, which will prove existence on the full interval. Note that, by assumption, $j$ is Bochner integrable on $[0,T]$ and therefore admissible in the sense of \cref{th:ode.control}.

The following lemma summarizes some useful inequalities.
\begin{lemma}
\leavevmode
\label{lemma:inequalities}
\begin{enumerate}[label = (\roman*), ref = \ref{lemma:inequalities}(\roman*)]
\item
\label{lemma:inequalities_1}
If $v \in C_0^1(\mathbb{R}^3, \mathbb{R}^3)$ and 
$\varphi \in \Diffid{1}{\mR^3}$,
then $v \circ \varphi \in C_0^1(\mathbb{R}^3, \mathbb{R}^3)$ and
\[
	\|v \circ \varphi\|_{1, \infty} \leq (2 + \|\varphi - id\|_{1, \infty}) \, \|v\|_{1, \infty} .
\]

\item
\label{lemma:inequalities_2}
If $v \in C_0^2(\mathbb{R}^3, \mathbb{R}^3)$ and 
$\varphi, \psi \in \Diffid{1}{\mR^3}$, 
then 
\[
	\|v \circ \varphi - v \circ \psi\|_{1, \infty}
	\leq
	(2 + \|\psi - id\|_{1, \infty}) \, \|v\|_{2, \infty} \, \|\varphi - \psi\|_{1, \infty} .
\]

\item 
\label{lemma:inequalities_3}
If $j \in V^*$ and $\varphi \in \Diffid{1}{\mathbb{R}^3}$, then
\[
	\|v_{\varphi, \hspace{1pt} j}\|_V
	=
	\|L_\varphi^{-1} j\|_V \leq \frac{1}{\reg} \left\|j \right\|_{V^*} ,
\]
i.e., $L_\varphi^{-1} \in \mathscr{L}(V^*\!, V)$ with $\|L_\varphi^{-1}\|_{\mathscr{L}(V^*\!, \, V)} \leq \frac{1}{\reg}$.

\item
\label{lemma:inequalities_4}
If $j \in V^*$ and $\varphi, \, \psi \in \Diffid{1}{\mathbb{R}^3}$, then
\[
	\|L_\varphi^{-1} - L_\psi^{-1}\|_{\mathscr{L}(V^*\!, \, V)}
	\leq
	\frac{1}{\reg^2} \, \|A_\varphi - A_\psi\|_{\mathscr{L}(V, \, V^*)} ,
\]
which implies
\begin{equation*}
	\|v_{\varphi, \hspace{1pt} j} - v_{\psi, \hspace{1pt} j}\|_V
	\leq
	\frac{1}{\reg^2} \, \|A_\varphi - A_\psi\|_{\mathscr{L}(V, \, V^*)} \, \|j\|_{V^*} .
\end{equation*}

\end{enumerate}
\end{lemma}
\begin{proof}
\leavevmode
The proofs of (i) and (ii) are straightforward applications of the chain rule and  left to the reader.  
To show (iii), it is equivalent to prove that, for all $v \in V$,
$\|v\|_V \leq (1/\reg)\left\| L_\varphi \, v \right\|_{V^*}$ .
We have, letting $\mathit{Id}_V$ denote the identity operator on $V$,
\begin{align*}
	\left( \frac{1}{\reg} \left\| L_\varphi \, v \right\|_{V^*} \right)^2
	&=
	\left( \frac{1}{\reg} \left\| K_V \left( \reg K_V^{-1} + A_\varphi \right) v \right\|_V \right)^2
	\\
	&=
	\frac{1}{\reg^2} \left\| \left( \reg \, \mathit{Id}_V + K_V A_\varphi \right) v \right\|_V^2
	\\
	&=
	\|v\|_V^2 + \frac{1}{\reg^2} \, \|K_V A_\varphi \hspace{1pt} v\|_V^2 + \frac{2}{\reg} \, \langle v, K_V A_\varphi \hspace{1pt} v \rangle_V
	\\
	&=
	\|v\|_V^2 + \frac{1}{\reg^2} \, \|K_V A_\varphi \hspace{1pt} v\|_V^2 + \frac{2}{\reg} \, (A_\varphi \hspace{1pt} v \mid v)
	\, \geq \,
	\|v\|_V^2 ,
\end{align*}
where the last inequality follows from $(A_\varphi \hspace{1pt} v \mid v) \geq 0$.

\medskip
We now prove (iv), writing:
\begin{align*}
	\|L_\varphi^{-1} - L_\psi^{-1}\|_{\mathscr{L}(V^*\!, \, V)}
	&=
	\left\| L_\varphi^{-1} \left( L_\psi - L_\varphi \right) L_\psi^{-1} \right\|_{\mathscr{L}(V^*\!, \, V)}
	\\
	&=
	\left\| L_\varphi^{-1} \left( A_\psi - A_\varphi \right) L_\psi^{-1} \right\|_{\mathscr{L}(V^*\!, \, V)}
	\\
	&\leq
	\|L_\varphi^{-1}\|_{\mathscr{L}(V^*\!, \, V)} \ 
	\|A_\varphi - A_\psi\|_{\mathscr{L}(V, \, V^*)} \ 
	\|L_\psi^{-1}\|_{\mathscr{L}(V^*\!, \, V)} \ 
	\\
	&\leq
	\frac{1}{\reg^2} \, \|A_\varphi - A_\psi\|_{\mathscr{L}(V, \, V^*)},
\end{align*}
since (iii)  implies that $\|L_\varphi^{-1}\|_{\mathscr{L}(V^*\!, \, V)} \leq ({1}/{\reg})$.
\end{proof}

\begin{corollary}
\label{cor:lipschitz.F}
Let $F(j, \varphi) = v_{\varphi,j} \circ \varphi$ be defined on $V^* \times \Diffid{1}{\mR^d}$. Then, $F$ is continuous and for all $j\in V^*$, and $\varphi \in \Diffid{1}{\mR^{d}}$, 
\[
\|F(j, \varphi)\|_{1, \infty} \leq \frac{2c_{V}}{\reg} (1+\|\varphi-\mathit{id}\|_{1, \infty}) \|j\|_{V^{*}}
\]
so that the assumption \ref{th:ode.control:1} holds for $\tilde F: (j, h) \mapsto F(j, \id +h)$  with $\gamma(j) = (2c_{V}/\reg) \|j\|_{V^{*}}$. 

Moreover, for all $j\in V^*$, and $\varphi, \psi \in \Diffid{1}{\mR^{d}}$
\[
\|F(j, \varphi)- F(j, \psi)\|_{1, \infty} \leq C_\psi \, (\|\varphi - \psi\|_{1, \infty} +  \|A_\varphi - A_\psi\|_{\mathscr{L}(V, \, V^*)}) \|j\|_{V^{*}} .
\] 
Thus, if $A_{\varphi}$ is locally Lipschitz, so is $F$ and $\tilde F$, and the assumption \ref{th:ode.control:2} holds.
\end{corollary}
\begin{proof}
We have
\begin{align*}
\|F(j, \varphi)\|_{1, \infty} &\leq (2+\|\varphi-\mathit{id}\|_{1, \infty}) \|v_{\varphi, j}\|_{1, \infty}\\
&\leq c_{V} (2+\|\varphi-\mathit{id}\|_{1, \infty}) \|v_{\varphi, j}\|_{V}\\
&\leq \frac{c_{V}}{\reg} (2+\|\varphi-\mathit{id}\|_{1, \infty}) \|j\|_{V^{*}}\,.
\end{align*}
This inequality also shows the continuity of $F$ in $j$, since $F(j, \varphi) = (L_\varphi^{-1} j) \circ \varphi$ is linear in $j$.

Similarly
\begin{align*}
\|F(j, \varphi)- F(j, \psi)\|_{1, \infty} &\leq \|v_{\varphi,j} \circ \varphi - v_{\varphi,j} \circ \psi\|_{1, \infty} + \|v_{\varphi,j} \circ \psi - v_{\psi,j} \circ \psi\|_{1, \infty}\\
&\leq (2+ \|\psi - \mathit{id}\|_{1, \infty}) \|\varphi - \psi\|_{1, \infty} \|v_{\varphi,j}\|_{2, \infty}\\
&  + (2+ \|\psi - \mathit{id}\|_{1, \infty}) \|v_{\varphi,j} - v_{\psi,j}\|_{1, \infty} \\
&\leq \frac{c_{V}}{\reg} (2+ \|\psi - \mathit{id}\|_{1, \infty}) \|\varphi - \psi\|_{1, \infty} \|j\|_{V^{*}} \\
& + \frac{1}{\reg^2} (2+ \|\psi - \mathit{id}\|_{1, \infty})\, \|A_\varphi - A_\psi\|_{\mathscr{L}(V, \, V^*)} \|j\|_{V^{*}}.
 \end{align*}

Since $F(j, \varphi)$ is continuous and linear in $j$ and continuous in $\varphi$, we know that $F$ is continuous.
\end{proof}

\Cref{cor:lipschitz.F} therefore implies that \cref{eq:ctrl_problem_system} has a unique local solution as soon as $j(t)$ is integrable. To conclude the proof of  (i) in \cref{thm:ctrl_problem}, it suffices to show that any solution on an interval $[0, T_{0})$ is such that $\varphi(T_{0}) \colonequals \lim_{t \,\uparrow\, T_0} \varphi(t) \in \Diffid{1}{\mR^{d}}$. This is true because, if one lets $w(t) = v_{\varphi(t), j(t)}$, then $w \in L^1([0,T_0], V)$ and $\varphi$ is, by construction, the flow associated to the ordinary differential equation $\partial_t z = w(t,z)$, i.e., it satisfies $\partial_t \varphi(t,x) = w(t, \varphi(t,x))$. This shows that (see, e.g., \cite{younes2010shapes} Chap.\ 8) $\varphi \in C([0, T_0], \Diffid{2}{\mathbb{R}^3})$, which completes the proof of part (i) of \cref{thm:ctrl_problem}. 

In particular, for $t\in [0, T_0]$, one has $\varphi(t, \cdot)^{-1} = \psi(t, \cdot)$ where $\psi$ is the flow associated with the equation $\partial_s z = \tilde w^{(t)}(s,z)$, with $\tilde w^{(t)}(s) = -v_{\varphi(t-s), j(t-s)}$. Note that we have not defined $w$ at time $T_0$ (or $\tilde w$ at time 0), which is not a problem since these time-dependent vector fields only need to be defined almost everywhere (in $t$) for the result to hold. (One can actually show using Cauchy sequences that $w(t)$ has a limit when $t$ tends to $T_0$ anyway). 

We also point out that, using the inequalities of \cref{lemma:inequalities}, the transcription of \cref{eq:ode.control.bound} to the present case becomes $\|\varphi(t) - id\|_{1, \infty} \leq B_j$
with
\begin{equation}
    \label{eq:bj}
	B_j
	=
	 \frac{2c_v}{\reg} \|j\|_{\mathcal X^1_{V^*\!,\,T}}  \exp\!\left( \frac{c_V}{\reg} \, \|j\|_{\mathcal{X}_{V^*\!, \, T}^1} \right).
\end{equation}
(here, the initial condition is always $\varphi(0,\cdot)=\id$) and thus $B_j$ only depends on $j$. 

We also have $\|\varphi^{-1}(t) - id\|_{1, \infty} \leq B_j$. This can be proved for each $t$ by applying  \cref{th:ode.control} to $F(\tilde w, \psi) = \tilde w \circ \psi$, with $U = V$, which satisfies the hypotheses with $\gamma(\tilde w) = c_V \|\tilde w\|_V$. To obtain the same constant $B_j$, one only needs to notice that, for all $s$,  $\|\tilde w^{(t)}(s)\|_V = \|w(t-s)\|_V$

We summarize this discussion in the next lemma.

\begin{lemma}
\label{lemma:bounds}
Suppose that $j \in \mathcal{X}_{V^*\!, \, T}^1$. If $\varphi \in C([0, T_0], \Diffid{1}{\mathbb{R}^3})$ is a local solution to  system~\cref{eq:ctrl_problem_system} for some $T_0 \leq T$, then $\|\varphi(t) - id\|_{1, \infty} \leq B_j$ and $\|\varphi^{-1}(t) - id\|_{1, \infty} \leq B_j$ for all $t \in [0, T_0]$, where $B_j$ is defined in equation \cref{eq:bj}.

Moreover, the local solution $\varphi$ is actually in $C([0, T_0], \Diffid{2}{\mathbb{R}^3})$ and there exists a constant $C_j$ such that
\begin{equation}
	\label{eq:proof_sol_D2phi_bound}
\|D^2\varphi(t)\|_\infty \leq C_j 	\exp\!\left( \frac{c_V}{\reg} \, \|j\|_{\mathcal{X}_{V^*\!, \, T}^1} \right) 
\end{equation}
for all $t\in [0,T_0]$.
\end{lemma}
Only the last inequality remains to be shown. We have
\begin{align}
	\label{eq:proof_sol_D2phi}
	\begin{split}
	|D^2\varphi(t, x)|
	&\leq
	\int_0^t
	\left(
		\vphantom{\sum}
		|D^2 v_{\varphi(s), \hspace{1pt} j(s)}(\varphi(s, x))| \, |D\varphi(s, x)|^2
	\right.
	\\
	&\hspace{30pt}
	\left.
		\vphantom{\sum}
		\phantom{}
		+ 
		|Dv_{\varphi(s), \hspace{1pt} j(s)}(\varphi(s, x))| \, |D^2\varphi(s, x)|
	\right)
	ds ,\\
	&\leq  \frac{c_V}{\reg} \|j\|_{\chi^1_{V^*\!,\,T}} \|D\varphi(t)\|^2_{\infty} + \frac{c_V}{\reg} \int_0^t \|j(s)\|_{V^*} 	|D^2\varphi(s, x)|\, ds
	\end{split}
\end{align}
and inequality \cref{eq:proof_sol_D2phi_bound} follows from  Gronwall's lemma, since $\|D\varphi(t)\|^2_{\infty} \leq (1+B_j)^2$.

 \subsection{Proof of \hyperref[thm:opt_problem_ode]{Theorem}~\ref{thm:opt_problem_ode}}
\label{subsec:proof_opt_problem_ode}
Denote $v_{\varphi, \theta} = v_{\varphi, j(\varphi, \theta)}$ and let
\[
G_\theta(\varphi) = v_{\varphi, \theta} \circ \varphi.
\]
We only need to apply the standard existence theorem for the ODE $\partial_t \varphi = G_\theta(\varphi)$. For this,  it is sufficient to prove
that, for any $\theta\in \Theta$,
\begin{enumerate}[label=(\arabic*)]
\item
There exists $\gamma >0$ such that, for all $\varphi\in \Diffid{1}{\mR^d}$:
\[
\|G_\theta(\varphi)\|_{1, \infty} \leq \gamma (1 + \|\varphi - \id\|_{1, \infty})\,. 
\]
\item 
For all $\varphi_0\in \Diffid{1}{\mR^d}$, there exists $r_{\varphi_0}>0$ and $\gamma_{\varphi_0} >0$ such that for all $\varphi,\psi\in \Diffid{1}{\mR^d}$ with $\max(\|\varphi-\varphi_0\|_\mB, \|\psi-\varphi_0\|_\mB)\leq r_{\varphi_0}$
\[
\|G_\theta(\varphi) - G_\theta(\psi)\|_{1, \infty} \leq \gamma_{\varphi_0} \|\varphi-\psi\|_{1, \infty}\,.
\]
\end{enumerate}
These properties are easily deduced from \cref{lemma:inequalities} and the boundedness and Lipschitz assumptions made on $j(\varphi, \theta)$, and we skip the argument.

\subsection{Proof of \hyperref[thm:ctrl_problem_min]{Theorem}~\ref{thm:ctrl_problem_min}}
\label{subsec:proof_ctrl_problem_min}

Given $j \in \mathcal{X}_{V^*\!, \, T}^2 \subset \mathcal{X}_{V^*\!, \, T}^1$, we have proved that system~\cref{eq:ctrl_problem_system} has a unique solution $\varphi_j \in C([0, T], \Diffid{2}{\mathbb{R}^3})$. Denote the well-defined objective function by
\[
	f(j) = \int_0^T (j(t) \mid v_{\varphi_j(t), \hspace{1pt} j(t)}) \, dt + \rho(\varphi_j(T, M_0), M_\targ) . 
\]
To prove the existence of minimizers of $f$, we show that minimizing sequences of $f$ are bounded and $f$ is weakly sequentially lower semicontinuous. Since $\mathcal{X}_{V^*\!, \, T}^2$ is a Hilbert space, the existence of minimizers will follow by applying the direct method of the calculus of variations. 

\begin{lemma}
Minimizing sequences of $f$ are bounded.
\end{lemma}
\begin{proof}
Let $(j_n)_{n = 1}^\infty$ be a minimizing sequence of $f$. 
We denote for short $v_n(t) \colonequals v_{\varphi_{j_n}\!(t), \hspace{1pt} j_n(t)}$.
Note that the solution of $j \equiv 0$ is $\varphi_j \equiv \mathit{id}$, so we can assume $f(j_n) \leq f(0) = \rho(M_0, M_\targ)$ without loss of generality. Using \cref{lemma:minimizer}, we have $j_n(t) = (\reg K_V^{-1} + A_{\varphi_{j_n}\!(t)}) \, v_n(t)$. It follows that
\begin{align}
	\begin{split}
	\label{eq:claim_bounded}
	\rho(M_0, M_\targ)
	&\geq
	f(j_n)
	\geq
	\int_0^T (j_n(t) \mid v_n(t)) \, dt
	\\
	&\geq
	\int_0^T \reg \, (K_V^{-1} v_n(t) \mid v_n(t)) \, dt\\
	&=
	\reg \int_0^T \|v_n(t)\|_V^2 \, dt \geq \frac \reg T \left( \int_0^T \|v_n(t)\|_V \, dt \right)^2
	\end{split}
\end{align}
From
\Cref{th:ode.control} (applied to $F(v, h) = v\circ (\id + h)$ and $U=V$), we know that the  boundedness of $\int_0^T \|v_n(t)\|_V \, dt$ implies the boundedness of solutions, i.e., there exists a constant $B$ such that $\|\varphi_{j_n}\!(t) - \mathit{id}\|_{1, \infty} \leq B$ for all $t \in [0, T]$ and $n \in \mathbb{N}$. Since $\varphi \mapsto A_\varphi$ is Lipschitz with respect to the seminorm $\|\cdot\|_{1, \infty}^{M_0}$ on $\mathfrak{S}_B$, denoting the Lipschitz constant by $\ell_A(B)$ leads to
\[
	\|A_{\varphi_{j_n}\!(t)}\|_{\mathscr{L}(V, \, V^*)}
	\leq
	\|A_{\mathit{id}}\|_{\mathscr{L}(V, \, V^*)} + \ell_A \, \|\varphi_{j_n}\!(t) - \mathit{id}\|_{1, \infty}^{M_0}
	\leq
	C_B .
\]
Now we return to inequality \cref{eq:claim_bounded} and write
\begin{align*}
	\rho(M_0, M_\targ)
	&\geq
	\reg \int_0^T \|v_n(t)\|_V^2 \, dt
	\\
	&\geq
	\reg \int_0^T \left( \frac{\|j_n(t)\|_{V^*}}{\|\reg K_V^{-1} + A_{\varphi_{j_n}\!(t)}\|_{\mathscr{L}(V, \, V^*)}} \right)^2 dt
	\\
	&\geq
	\reg \int_0^T \left( \frac{\|j_n(t)\|_{V^*}}{\reg  + C_B}  \right)^2 dt
	=
	C_B \, \|j_n\|_{\mathcal{X}_{V^*\!, \, T}^2}^2 ,
\end{align*}
where we have used the fact that $K_V^{-1}$ is an isometry from $V$ to $V^*$ in the third inequality. The above inequality shows that a minimizing sequence $(j_n)_{n = 1}^\infty$ is bounded.
\end{proof}

To prove that $f$ is weakly sequentially lower semicontinuous, we separate the two terms in $f$ and show in 
\cref{claim:discrepancy,claim:running_cost} respectively that the first term $j \mapsto \int_0^T (j(t) \mid v_{\varphi_j(t),j(t)}) \, dt$ and the second term $j \mapsto \rho(\varphi_j(T, M_0), M_\targ)$ are both weakly sequentially lower semicontinuous.
\smallskip

Let $j_n \rightharpoonup j$ in $\mathcal{X}_{V^*\!, \, T}^2$. We first derive some preliminary bounds that will be used in the following 
\cref{claim:seminorm_continuous,claim:discrepancy,claim:running_cost}. Since $j_n \rightharpoonup j$, there exists $J > 0$ such that $\|j_n\|_{\mathcal{X}_{V^*\!, \, T}^2} \leq J$ for all $n \in \mathbb{N}$. It follows that
\[
	\|j\|_{\mathcal{X}_{V^*\!, \, T}^2} \leq \liminf_{n \rightarrow \infty} \|j_n\|_{\mathcal{X}_{V^*\!, \, T}^2} \leq J .
\]
For the solution $\varphi_{j_n}$,  \cref{lemma:bounds} shows that for every $t \in [0, T]$
\[
	\|\varphi_{j_n}\!(t) - \mathit{id}\|_{1, \infty}
	\leq
	\frac{2c_V}{\reg} \, \|j_n\|_{\mathcal{X}_{V^*\!, \, T}^1}
	\exp\!\left( {\frac{c_V}{\reg} \|j_n\|_{\mathcal{X}_{V^*\!, \, T}^1}} \right) ,
\]
so
\[
	\|\varphi_{j_n}\!(t) - \mathit{id}\|_{1, \infty}
	\leq
	\frac{2c_V}{\reg} \, J \sqrt{T} \ \exp\!\left( {\frac{c_V}{\reg} J \sqrt{T}} \right)
	\equalscolon
	B_J .
\]
Similarly, we have
\begin{align*}
	\|\varphi_j(t) - \mathit{id}\|_{1, \infty}
	&\leq
	\frac{2c_V}{\reg} \, \|j\|_{\mathcal{X}_{V^*\!, \, T}^1}
	\exp\!\left( {\frac{c_V}{\reg} \|j\|_{\mathcal{X}_{V^*\!, \, T}^1}} \right)
	\\
	&\leq
	\frac{2c_V}{\reg} \, J \sqrt{T} \ \exp\!\left( {\frac{c_V}{\reg} J \sqrt{T}} \right)
	=
	B_J .
\end{align*}
Still from \cref{lemma:bounds}, it also holds that
\[
	\|\varphi_{j_n}^{-1}(t) - \mathit{id}\|_{1, \infty} \leq B_J
	\ \ \mbox{ and } \ \ 
	\|\varphi_j^{-1}(t) - \mathit{id}\|_{1, \infty} \leq B_J .
\]
We then have $(\varphi_{j_n})_{n = 1}^\infty \subset \mathfrak{S}_{B_J}$ and $\varphi_j \in \mathfrak{S}_{B_J}$, where $\mathfrak{S}_{B_J}$ is defined in the statement of the theorem.
\smallskip

The following lemma is a preliminary step towards 
\cref{claim:discrepancy,claim:running_cost}, which prove the weakly sequentially lower semicontinuity of $f$.

\begin{lemma}
\label{claim:seminorm_continuous}
If $j_n \rightharpoonup j$ in $\mathcal{X}_{V^*\!, \, T}^2$, then $\|\varphi_{j_n}\!(t) - \varphi_j(t)\|_{1, \infty}^{M_0} \rightarrow 0$ for all $t \in [0, T]$.
\end{lemma}
\begin{proof}
Note that
\begin{align*}
	&\hspace{14pt}
	\varphi_{j_n}\!(t, x) - \varphi_j(t, x)
	\\
	&=
	\int_0^t
	\left(
		v_{\varphi_{j_n}\!(s), \hspace{1pt} j_n(s)}(\varphi_{j_n}\!(s, x))
		-
		v_{\varphi_j(s), \hspace{1pt} j(s)}(\varphi_j(s, x))
	\right)
	ds
	\\
	&=
	\int_0^t
	\left(
		\left( L_{\varphi_{j_n}\!(s)}^{-1} \, j_n(s) \right)(\varphi_{j_n}\!(s, x))
		-
		\left( L_{\varphi_j(s)}^{-1} \, j_n(s) \right)(\varphi_{j_n}\!(s, x))
	\right)
	ds
	\\
	&
	\hspace{12pt}
	\phantom{}
	+
	\int_0^t
	\left(
		\left( L_{\varphi_j(s)}^{-1} \, j_n(s) \right)(\varphi_{j_n}\!(s, x))
		-
		\left( L_{\varphi_j(s)}^{-1} \, j_n(s) \right)(\varphi_j(s, x))
	\right)
	ds
	\\
	&
	\hspace{12pt}
	\phantom{}
	+
	\int_0^t
	\left(
		\left( L_{\varphi_j(s)}^{-1} \, j_n(s) \right)(\varphi_j(s, x))
		-
		\left( L_{\varphi_j(s)}^{-1} \, j(s) \right)(\varphi_j(s, x))
	\right)
	ds
	\\
	&\equalscolon
	I_{1, n}(t)+ I_{2, n}(t) + \varLambda_n(t)
\end{align*}
Taking $\|\cdot\|_{1, \infty}^{M_0}$ on both sides, we will show that
\begin{equation}
	\label{eq:existence_minimizer_ineq}
	\|I_{1, n}(t)\|_{1, \infty}^{M_0} + \|I_{2, n}(t)\|_{1, \infty}^{M_0}
	\leq
	C_J \left( \int_0^t \left( \|\varphi_{j_n}\!(s) - \varphi_j(s) \|_{1, \infty}^{M_0} \right)^2 ds \right)^{\frac{1}{2}} ,
\end{equation}
and that
\begin{equation}
	\label{eq:existence_minimizer_limit}
	\lim_{n \rightarrow \infty}
	\left\| \varLambda_n(t) \right\|_{1, \infty}^{M_0} = 0
	\ \ \mbox{ for all } t \in [0, T] .
\end{equation}
Identities \cref{eq:existence_minimizer_ineq,eq:existence_minimizer_limit} combined with Gronwall's lemma will then lead to $\|\varphi_{j_n}\!(t) - \varphi_j(t)\|_{1, \infty}^{M_0} \rightarrow 0$ for all $t \in [0, T]$.

We estimate $\|I_{1, n}(t)\|_{1, \infty}^{M_0}$ as follows.
\begin{align*}
	&\hspace{13pt}
	\|I_{1, n}(t)\|_{1, \infty}^{M_0}
	\\
	&\leq
	\int_0^t
		\left\|
			\left( L_{\varphi_{j_n}\!(s)}^{-1} \, j_n(s) \right) \circ \varphi_{j_n}\!(s)
			-
			\left( L_{\varphi_j(s)}^{-1} \, j_n(s) \right) \circ \varphi_{j_n}\!(s)
		\right\|_{1, \infty}
	ds
	\\
	&\leq
	\int_0^t \,
	(2 + \|\varphi_{j_n}\!(s) - \mathit{id}\|_{1, \infty}) \,
	\left\|
		L_{\varphi_{j_n}\!(s)}^{-1} \, j_n(s)
		-
		L_{\varphi_j(s)}^{-1} \, j_n(s)
	\right\|_{1, \infty} \,
	ds
	&
	\hspace{-0.5cm}
	\mbox{(\hyperref[lemma:inequalities_1]{Lemma}~\ref{lemma:inequalities_1})}
	\\
	&\leq
	\int_0^t \,
	(2 + B_J) \ 
	\frac{c_V}{\reg^2} \, \|A_{\varphi_{j_n}\!(s)} - A_{\varphi_j(s)}\|_{\mathscr{L}(V, \, V^*)} \ 
	\|j_n(s)\|_{V^*} \,
	ds
	&
	\hspace{-0.5cm}
	\mbox{(\hyperref[lemma:inequalities_4]{Lemma}~\ref{lemma:inequalities_4})}
\end{align*}
Since $\varphi \mapsto A_\varphi$ is Lipschitz on $\mathfrak{S}_{B_J}$ with respect to $\|\cdot\|_{1, \infty}^{M_0}$, denoting the Lipschitz constant by $\ell_A(B_J)$ leads to
\begin{align}
	\|I_{1, n}(t)\|_{1, \infty}^{M_0}
	&\leq
	\frac{c_V}{\reg^2} \, (2 + B_J) 
	\int_0^t \, 
	\ell_A \ 
	\|\varphi_{j_n}\!(s) - \varphi_j(s)\|_{1, \infty}^{M_0} \ \,
	\|j_n(s)\|_{V^*} \,
	ds
	\label{eq:existence_minimizer_ineq1_middle}
	\\
	&\leq
	\frac{c_V}{\reg^2} \, (2 + B_J) \ 
	\ell_A
	\left( \int_0^t \left( \|\varphi_{j_n}\!(s) - \varphi_j(s)\|_{1, \infty}^{M_0} \right)^2 ds \right)^{\frac{1}{2}}
	\|j_n\|_{\mathcal{X}_{V^*\!, \, T}^2}
	\notag
	\\
	&\leq
	C_J \left( \int_0^t \left( \|\varphi_{j_n}\!(s) - \varphi_j(s) \|_{1, \infty}^{M_0} \right)^2 ds \right)^{\frac{1}{2}} .
	\notag
\end{align}
\medskip

We now pass to  $\|I_{2, n}(t)\|_{1, \infty}^{M_0}$. Similarly to \hyperref[lemma:inequalities_4]{Lemma}~\ref{lemma:inequalities_2}, we have
\[
	\|v \circ \varphi - v \circ \psi\|_{1, \infty}^{M_0}
	\leq
	(2 + \|\psi - \mathit{id}\|_{1, \infty}^{M_0}) \, \|v\|_{2, \infty} \, \|\varphi - \psi\|_{1, \infty}^{M_0}.
\]
It follows that
\begin{align}
	&\hspace{13pt}
	\|I_{2, n}(t)\|_{1, \infty}^{M_0}
	\notag
	\\
	&\leq
	\int_0^t \,
	\left\|
		\left( L_{\varphi_j(s)}^{-1} \, j_n(s) \right) \circ \varphi_{j_n}\!(s)
		-
		\left( L_{\varphi_j(s)}^{-1} \, j_n(s) \right) \circ \varphi_j(s)
	\right\|_{1, \infty}^{M_0}
	ds
	\notag
	\\
	&\leq
	\int_0^t \,
	\left( 2 + \|\varphi_j(s) - \mathit{id}\|_{1, \infty}^{M_0} \right) \,
	\|L_{\varphi_j(s)}^{-1} \, j_n(s)\|_{2, \infty} \,
	\|\varphi_{j_n}\!(s) - \varphi_j(s)\|_{1, \infty}^{M_0} \,
	ds
	\notag
	\\
	&\leq
	\int_0^t \,
	(2 + B_J) \ 
	\frac{c_V}{\reg} \, \|j_n(s)\|_{V^*} \,
	\|\varphi_{j_n}\!(s) - \varphi_j(s)\|_{1, \infty}^{M_0} \,
	ds
	\label{eq:existence_minimizer_ineq2_middle}
	\\
	&\leq
	C_J \left( \int_0^t \left( \|\varphi_{j_n}\!(s) - \varphi_j(s) \|_{1, \infty}^{M_0} \right)^2 ds \right)^{\frac{1}{2}} .
	\notag
\end{align}
\medskip

We proceed to show that \cref{eq:existence_minimizer_limit} holds.
To simplify notations, define $u_n: \mathbb{R}^3 \rightarrow \mathbb{R}^3$ and $u: \mathbb{R}^3 \rightarrow \mathbb{R}^3$ by
\[
	u_n(x)
	=
	\int_0^t \left( L_{\varphi_j(s)}^{-1} \, j_n(s) \right) \circ \varphi_j(s, x) \, ds
	\ \ \mbox{ and } \ \ 
	u(x)
	=
	\int_0^t \left( L_{\varphi_j(s)}^{-1} \, j(s) \right) \circ \varphi_j(s, x) \, ds 
\]
and prove that $\|u_n - u\|_\infty^{M_0}$ and $\|Du_n - Du\|_\infty^{M_0}$ converge to $0$ as $n \rightarrow \infty$. We aim to prove pointwise convergence, uniform boundedness, and equicontinuity of the two sequences, so as to invoke the Arzel{\`a}--Ascoli theorem. Given $x \in \mathbb{R}^3$, we define linear operators $\mathcal{L}_x: \mathcal{X}_{V^*\!, \, T}^2 \rightarrow \mathbb{R}^3$ and $\widetilde{\mathcal{L}}_x: \mathcal{X}_{V^*\!, \, T}^2 \rightarrow \mathbb{R}^{3 \times 3}$ by
\[
	\mathcal{L}_x \, j'
	=
	\int_0^t \left( L_{\varphi_j(s)}^{-1} \, j'(s) \right) \circ \varphi_j(s, x) \, ds
	\ \ \mbox{ and } \ \ 
	\widetilde{\mathcal{L}}_x \, j'
	=
	\int_0^t D\!\left( L_{\varphi_j(s)}^{-1} \, j'(s) \circ \varphi_j(s, x) \right) \, ds .
\]
Note that
\[
	|\mathcal{L}_x \, j'| + |\widetilde{\mathcal{L}}_x \, j'|
	\leq
	\int_0^t \left\| \left( L_{\varphi_j(s)}^{-1} \, j'(s) \right) \circ \varphi_j(s) \right\|_{1, \infty} ds
	\leq
	\frac{c_V \sqrt{T}}{\reg} \, (2 + B_J) \, \|j'\|_{\mathcal{X}_{V^*\!, \, T}^2} ,
\]
so both $\mathcal{L}_x$ and $\widetilde{\mathcal{L}}_x$ are bounded linear operators for all $x \in \mathbb{R}^3$. Since $j_n \rightharpoonup j$, pointwise convergence of $(u_n)_{n = 1}^\infty$ and $(D u_n)_{n = 1}^\infty$ now follows from
\[
	u_n(x) = \mathcal{L}_x \, j_n \rightarrow \mathcal{L}_x \, j = u(x)
	\ \ \mbox{ and } \ \ 
	D u_n(x) = \widetilde{\mathcal{L}}_x \, j_n \rightarrow \widetilde{\mathcal{L}}_x \, j = D u(x) .
\]
We also have
\[
	|u_n(x)| + |D u_n(x)|
	=
	|\mathcal{L}_x \, j_n| + |\widetilde{\mathcal{L}}_x \, j_n|
	\leq
	\frac{c_V \sqrt{T}}{\reg} \, (2 + B_J) \, \|j_n\|_{\mathcal{X}_{V^*\!, \, T}^2}
	\leq
	C_J ,
\]
which shows that the two sequences are uniformly bounded. The sequence $(u_n)_{n = 1}^\infty$ is equicontinuous on $\mathbb{R}^3$ because
\[
	|u_n(x) - u_n(y)|
	\leq
	\int_0^t
	\left\| D\!\left( L_{\varphi_j(s)}^{-1} \, j_n(s) \right) \right\|_\infty
	\|D\varphi_j(s)\|_\infty \, |x - y| \, 
	ds
	\leq
	C_J \, |x - y| .
\]
The sequence $(D u_n)_{n = 1}^\infty$ is equicontinuous on $\mathbb{R}^3$ because
\begin{align*}
	&\hspace{13pt}
	|Du_n(x) - Du_n(y)|
	\\
	&\leq
	\int_0^t
	\left(
		\vphantom{\int}
		\left|
			\left( D\!\left( L_{\varphi_j(s)}^{-1} \, j_n(s) \right) \circ \varphi_j(s, x) \right) D\varphi_j(s, x)
		\right.
	\right.
	\\
	&\hspace{38pt}
	\left.
		\left.
			\phantom{}
			-
			\left( D\!\left( L_{\varphi_j(s)}^{-1} \, j_n(s) \right) \circ \varphi_j(s, x) \right) D\varphi_j(s, y)
		\right|
	\right.
	\\
	&\hspace{22pt}
	\left.
		\phantom{}
		+
		\left|
			\left( D\!\left( L_{\varphi_j(s)}^{-1} \, j_n(s) \right) \circ \varphi_j(s, x) \right) D\varphi_j(s, y)
		\right.
	\right.
	\\
	&\hspace{37pt}
	\left.
		\vphantom{\int}
		\left.
			\phantom{}
			-
			\left( D\!\left( L_{\varphi_j(s)}^{-1} \, j_n(s) \right) \circ \varphi_j(s, y) \right) D\varphi_j(s, y)
		\right|
	\right)
	ds
	\\
	&\leq
	\int_0^t
	\left(
		\left\| D\!\left( L_{\varphi_j(s)}^{-1} \, j_n(s) \right) \right\|_\infty
		\|D^2 \varphi_j(s)\|_\infty \, |x - y|
	\right.
	\\
	&\hspace{33pt}
	\left.
		\phantom{}
		+
		\left\| D^2\!\left( L_{\varphi_j(s)}^{-1} \, j_n(s) \right) \right\|_\infty
		\|D\varphi_j(s)\|_\infty \, |x - y| \ \|D\varphi_j(s)\|_\infty
	\right)
	ds
	\\
	&\leq
	C_J \, |x - y| ,
\end{align*}
where we have used $\|D^2 \varphi_j(s)\|_\infty \leq C_J$ by \cref{eq:proof_sol_D2phi_bound}.
From the Arzel{\`a}--Ascoli theorem, we know that every subsequence of $(u_n)_{n = 1}^\infty$ has a further subsequence that converges uniformly to $u$ on $M_0$, which implies $\|u_n - u\|_\infty^{M_0} \rightarrow 0$. Applying the same argument to the sequence $(D u_n)_{n = 1}^\infty$ gives $\|Du_n - Du\|_\infty^{M_0} \rightarrow 0$.
\medskip

In summary, we have proved~\cref{eq:existence_minimizer_ineq,eq:existence_minimizer_limit}, which lead to
\begin{align}
	\begin{split}
		\|\varphi_{j_n}\!(t) - \varphi_j(t)\|_{1, \infty}^{M_0}
		&\leq
		\|I_{1, n}(t)\|_{1, \infty}^{M_0} + \|I_{2, n}(t)\|_{1, \infty}^{M_0} + \|\varLambda_n(t)\|_{1, \infty}^{M_0}
		\\
		&\leq
		\lambda_n(t)
		+
		C_J
		\left( \int_0^t \left( \|\varphi_{j_n}\!(s) - \varphi_j(s) \|_{1, \infty}^{M_0} \right)^2 ds \right)^{\frac{1}{2}} ,
	\end{split}
	\label{eq:existence_minimizer_summary}
\end{align}
where $\lambda_n(t) = \|\varLambda_n(t)\|_{1, \infty}^{M_0} \rightarrow 0$ as $n \rightarrow \infty$. This implies
\[
	\left( \|\varphi_{j_n}\!(t) - \varphi_j(t)\|_{1, \infty}^{M_0} \right)^2
	\leq
	2 \, \lambda_n^2(t)
	+
	\int_0^t
	2 \, C_J^2
	\left( \|\varphi_{j_n}\!(s) - \varphi_j(s) \|_{1, \infty}^{M_0} \right)^2
	ds.
\]
By Gronwall's lemma, we finally obtain
\begin{equation}
	\label{eq:existence_minimizer_final}
	\left( \|\varphi_{j_n}\!(t) - \varphi_j(t)\|_{1, \infty}^{M_0} \right)^2
	\leq
	2 \, \lambda_n^2(t)
	+
	\int_0^t 4 \, \lambda_n^2(s) \, C_J^2 \, \exp\!\left(2 \, C_J^2 (t-s) \right) \, ds .
\end{equation}
Note that
\begin{align}
	\begin{split}
		\lambda_n(t)
		&\leq
		\left\|
			\int_0^t
			\left(
				\left( L_{\varphi_j(s)}^{-1} \, j_n(s) \right) \circ \varphi_j(s)
				-
				\left( L_{\varphi_j(s)}^{-1} \, j(s) \right) \circ \varphi_j(s)
			\right)
			ds
		\right\|_{1, \infty}
		\\
		&\leq
		\int_0^t
		(2 + \|\varphi_j(s) - \mathit{id}\|_{1, \infty}) \ 
		\left\| L_{\varphi_j(s)}^{-1} \, j_n(s) - L_{\varphi_j(s)}^{-1} \, j(s) \right\|_{1, \infty} \,
		ds
		\\
		&\leq
		\int_0^t
		(2 + B_J) \ 
		\frac{c_V}{\reg} \, (\|j_n(s)\|_{V^*} + \|j(s)\|_{V^*}) \,
		ds
		\,\leq\,
		C_J ,
	\end{split}
	\label{eq:proof_lambda_bounded}
\end{align}
so $\lambda_n(t)$ is uniformly bounded in $n$ and $t$. The dominated convergence theorem then shows that the right hand integral of \cref{eq:existence_minimizer_final} converges to $0$ as $n\rightarrow \infty$ and thus
\[
	\|\varphi_{j_n}\!(t) - \varphi_j(t)\|_{1, \infty}^{M_0}
	\rightarrow 0 \ \mbox{ as } \ n \rightarrow \infty ,
\]
which completes the proof.
\end{proof}

We immediately have the following lemma as a corollary.

\begin{lemma}
\label{claim:discrepancy}
The second term $j' \mapsto \rho(\varphi_{j'}(T, M_0), M_\targ)$ in $f$ is weakly sequentially continuous; hence it is weakly sequentially lower semicontinuous.
\end{lemma}
\begin{proof}
\Cref{claim:seminorm_continuous} shows, in particular at $t = T$, that $j_n \rightharpoonup j$ implies $\|\varphi_{j_n}\!(T) - \varphi_j(T)\|_{1, \infty}^{M_0} \rightarrow 0$. Since the discrepancy function $\rho$ is continuous on $\mathscr M$ with respect to $\|\cdot\|_{1, \infty}$, we have $\rho(\varphi_{j_n}\!(T, M_0), M_\targ) \rightarrow \rho(\varphi_{j}(T, M_0), M_\targ)$, yielding that $j' \mapsto \rho(\varphi_{j'}(T, M_0), M_\targ)$ is weakly sequentially continuous.
\end{proof}

\begin{lemma}
\label{claim:running_cost}
The first term $j' \mapsto \int_0^T (j'(t) \mid v_{\varphi_{j'}(t), \hspace{1pt} j'(t)}) \, dt$ in $f$ is weakly sequentially lower semicontinuous.
\end{lemma}
\begin{proof}
Let $j_n \rightharpoonup j$ in $\mathcal{X}_{V^*\!, \, T}^2$. Denote $v_n(t) \colonequals v_{\varphi_{j_n}\!(t), \hspace{1pt} j_n(t)}$ and $v(t) \colonequals v_{\varphi_j(t), \hspace{1pt} j(t)}$.
We first show that $v_n, v \in \mathcal{X}_{V, \, T}^2$ for all $n \in \mathbb{N}$ and $v_n \rightharpoonup v$ in $\mathcal{X}_{V, \, T}^2$. \hyperref[lemma:inequalities_3]{Lemma}~\ref{lemma:inequalities_3} gives
\[
	\int_0^T \|v_n(t)\|_V^2 \, dt
	=
	\int_0^T \|L_{\varphi_{j_n}\!(t)}^{-1} \, j_n(t)\|_V^2 \, dt
	\leq
	\frac{1}{\reg^2} \, \|j_n\|_{\mathcal{X}_{V^*\!, \, T}}^2 ,
\]
so $v_n \in \mathcal{X}_{V, \, T}^2$. Moreover, 
\[
\|v_n\|_{\mathcal{X}_{V, \, T}^2} \leq \frac{1}{\reg} \, \|j_n\|_{\mathcal{X}_{V^*\!, \, T}^2} \leq \frac{J}{\reg} 
\]
for all $n \in \mathbb{N}$. The same argument shows that $v \in \mathcal{X}_{V, \, T}^2$. To see $v_n \rightharpoonup v$ in $\mathcal{X}_{V, \, T}^2$, consider an arbitrary $\mu \in (\mathcal{X}_{V, \, T}^2)^* \simeq \mathcal{X}_{V, \, T}^2 \simeq \mathcal{X}_{V^*\!, \, T}^2$ and observe that
\begin{align*}
	&\hspace{12pt}
	\left| (\mu \mid v_n) - (\mu \mid v) \right|
	\\
	&=
	\left| \int_0^T \left( \vphantom{\sum} (\mu(t) \mid v_n(t)) - (\mu(t) \mid v(t)) \right) dt \right|
	\\
	&=
	\left|
		\int_0^T
		\left(
			\big( \mu(t) \mid (L_{\varphi_{j_n}\!(t)}^{-1} - L_{\varphi_j(t)}^{-1}) \, j_n(t) \big)
			+
			\big( \mu(t) \mid L_{\varphi_j(t)}^{-1} \, (j_n(t) - j(t)) \big)
		\right)
		dt
	\right|
	\\
	&\leq
	\int_0^T \|\mu(t)\|_{V^*} \, \|L_{\varphi_{j_n}\!(t)}^{-1} - L_{\varphi_j(t)}^{-1}\|_{\mathscr{L}(V^*\!, \, V)} \, \|j_n(t)\|_{V^*} \, dt
	\\
	&\hspace{20pt}
	\phantom{}
	+
	\left| \int_0^T \big( \mu(t) \mid L_{\varphi_j(t)}^{-1} \, (j_n(t) - j(t)) \big) \, dt \right| .
\end{align*}
\hyperref[lemma:inequalities_4]{Lemma}~\ref{lemma:inequalities_4}, the Lipschitz condition on $\varphi \mapsto A_\varphi$, and \cref{claim:seminorm_continuous} imply
\begin{align}
	\label{eq:L_inv_converges}
	\begin{split}
	\|L_{\varphi_{j_n}\!(t)}^{-1} - L_{\varphi_j(t)}^{-1}\|_{\mathscr{L}(V^*\!, \, V)}
	&\leq
	\frac{1}{\reg^2} \, \|A_{\varphi_{j_n}\!(t)} - A_{\varphi_j(t)}\|_{\mathscr{L}(V, \, V^*)}
	\\
	&\leq
	\frac{\ell_A}{\reg^2} \, \|\varphi_{j_n}\!(t) - \varphi_j(t)\|_{1, \infty}^{M_0}
	\rightarrow 0 .
	\end{split}
\end{align}
In addition, note that $j' \mapsto \int_0^T (\mu(t) \mid L_{\varphi_j(t)}^{-1} \, j'(t)) \, dt$ is a bounded linear functional on $\mathcal{X}_{V^*\!, \, T}^2$. We conclude that $\left| (\mu \mid v_n) - (\mu \mid v) \right| \rightarrow 0$ by the dominated convergence theorem and $j_n \rightharpoonup j$ in $\mathcal{X}_{V^*\!, \, T}^2$.

Now we show that
\[
	\int_0^T (j(t) \mid v(t)) \, dt
	\leq
	\liminf_{n \rightarrow \infty} \int_0^T (j_n(t) \mid v_n(t)) \, dt . 
\]
The mapping $v' \mapsto \int_0^T (L_{\varphi_j}(t) \, v'(t) \mid v'(t)) \, dt$ is strongly continuous and convex on $\mathcal{X}_{V, \, T}^2$ since
\begin{align*}
	\|L_{\varphi_j(t)}\|_{\mathscr{L}(V, \, V^*)}
	&=
	\|\reg K_V^{-1} + A_{\varphi_j(t)}\|_{\mathscr{L}(V, \, V^*)}
	\\
	&\leq
	\reg  + \|A_{\mathit{id}}\|_{\mathscr{L}(V, \, V^*)} + \ell_A \, \|\varphi_{j_n}\!(t) - \varphi_j(t)\|_{1, \infty}^{M_0}
	\leq
	C_J
\end{align*}
and
$	(L_{\varphi_j(t)} \, v'' \mid v'') = \reg \, \|v''\|_V^2 + (A_{\varphi_j(t)} \, v'' \mid v'') \geq 0 $.

The mapping is therefore weakly lower semicontinuous. We then have
\begin{align*}
	&\hspace{12pt}
	\liminf_{n \rightarrow \infty} \int_0^T (j_n(t) \mid v_n(t)) \, dt 
	\\
	&=
	\liminf_{n \rightarrow \infty}
	\int_0^T
	\left(
		\reg \, \|v_n(t)\|_V^2
		+
		(A_{\varphi_{j_n}\!(t)} \, v_n(t) \mid v_n(t))
	\right)
	dt
	\\
	&=
	\liminf_{n \rightarrow \infty}
	\int_0^T
	\left(
		\vphantom{\sum}
		(L_{\varphi_j}(t) \, v_n(t) \mid v_n(t))
		+
		((A_{\varphi_{j_n}\!(t)} - A_{\varphi_j(t)}) \, v_n(t) \mid v_n(t))
	\right)
	dt
	\\
	&\geq
	\int_0^T (L_{\varphi_j}(t) \, v(t) \mid v(t)) dt
	=
	\int_0^T (j(t) \mid v(t)) \, dt ,
\end{align*}
where the second to last inequality follows from $v_n \rightharpoonup v$, the weak lower semicontinuity of $v' \mapsto \int_0^T (L_{\varphi_j}(t) \, v'(t) \mid v'(t)) \, dt$, the inequality $\|v_n\|_{\mathcal{X}_{V, \, T}^2} \leq \frac{J}{\reg}$, and $\|A_{\varphi_{j_n}\!(t)} - A_{\varphi_j(t)}\|_{\mathscr{L}(V, \, V^*)} \rightarrow 0$ by \cref{eq:L_inv_converges}.
\end{proof}


\subsection{Proof of \hyperref[thm:opt_problem_min]{Theorem}~\ref{thm:opt_problem_min}}
\label{subsec:proof_opt_problem_min}

The analysis is similar. It suffices to show that for a convergent sequence $\theta_n \rightarrow \theta$ in $\Theta$, one has $\|\varphi_{\theta_n}\!(T) - \varphi_\theta(T)\|_{1, \infty}^{M_0} \rightarrow 0$, where $\varphi_{\theta_n}$ and $\varphi_\theta$ are the solutions to system~\cref{eq:opt_problem_system} corresponding to $\theta_n$ and $\theta$. 

For all $\varphi \in C([0, T], \Diffid{1}{\mathbb{R}^3})$ and 
$\theta' \in \Theta$, the boundedness assumption  on $j(\cdot, \theta')$ gives $\int_0^T \|j(\varphi(t), \theta')\|_{V^*} \, dt \leq J_\Theta \, T$, so \cref{lemma:bounds} implies $(\varphi_{\theta_n})_{n = 1}^\infty \subset \mathfrak{S}_{B_\Theta}$ and $\varphi_\theta \in \mathfrak{S}_{B_\Theta}$, where
\[
	B_\Theta
	=
	\frac{2c_V}{\reg} \, (J_\Theta \, T) \, \exp\!\left( {\frac{c_V}{\reg} (J_\Theta \, T)} \right) .
\]
We denote the Lipschitz constants of $\varphi \mapsto A_\varphi$ and $\{j(\cdot, \theta): \theta \in \Theta\}$
on $\mathfrak{S}_{B_\Theta}$ by $\ell_A$ and $\ell_j$ respectively. 

Observe that
\begin{align*}
	&\hspace{15pt}
	\varphi_{\theta_n}\!(t, x) - \varphi_\theta(t, x)
	\\
	&=
	\int_0^t
	\left(
		v_{\varphi_{\theta_n}\!(s), \hspace{1pt} j(\varphi_{\theta_n}\!(s), \, \theta_n)} \circ \varphi_{\theta_n}\!(s, x)
		-
		v_{\varphi_{\theta}(s), \hspace{1pt} j(\varphi_{\theta}(s), \, \theta)} \circ \varphi_\theta(s, x)
	\right)
	ds
	\\
	&=
	\int_0^t
	\left(
		\left( L_{\varphi_{\theta_n}\!(s)}^{-1} \, j(\varphi_{\theta_n}\!(s), \theta_n) \right) \circ \varphi_{\theta_n}\!(s, x)
		-
		\left( L_{\varphi_\theta(s)}^{-1} \, j(\varphi_{\theta_n}\!(s), \theta_n) \right) \circ \varphi_{\theta_n}\!(s, x)
	\right)
	ds
	\\
	&\hspace{10pt}
	\phantom{}
	+
	\int_0^t
	\left(
		\left( L_{\varphi_\theta(s)}^{-1} \, j(\varphi_{\theta_n}\!(s), \theta_n) \right) \circ \varphi_{\theta_n}\!(s, x)
		-
		\left( L_{\varphi_\theta(s)}^{-1} \, j(\varphi_{\theta_n}\!(s), \theta_n) \right) \circ \varphi_\theta(s, x)
	\right)
	ds
	\\
	&\hspace{10pt}
	\phantom{}
	+
	\int_0^t
	\left(
		\left( L_{\varphi_\theta(s)}^{-1} \, j(\varphi_{\theta_n}\!(s), \theta_n) \right) \circ \varphi_\theta(s, x)
		-
		\left( L_{\varphi_\theta(s)}^{-1} \, j(\varphi_\theta(s), \theta_n) \right) \circ \varphi_\theta(s, x)
	\right)
	ds
	\\
	&\hspace{10pt}
	\phantom{}
	+
	\int_0^t
	\left(
		\left( L_{\varphi_\theta(s)}^{-1} \, j(\varphi_\theta(s), \theta_n) \right) \circ \varphi_\theta(s, x)
		-
		\left( L_{\varphi_\theta(s)}^{-1} \, j(\varphi_\theta(s), \theta) \right) \circ \varphi_\theta(s, x)
	\right)
	ds
	\\
	&\equalscolon
	I_{1, n}(t) + I_{2, n}(t) + I_{3, n}(t) + \varLambda_n(t)
\end{align*}
With the same reasoning as in \hyperref[subsec:proof_ctrl_problem_min]{the proof of \hyperref[thm:ctrl_problem_min]{Theorem}~\ref{thm:ctrl_problem_min}}, we bound $\|I_{1, n}(t)\|_{1, \infty}^{M_0}$ (see~\cref{eq:existence_minimizer_ineq1_middle}) by
\begin{align*}
	&\hspace{13pt}
	\|I_{1, n}(t)\|_{1, \infty}^{M_0}
	\\
	&\leq
	\frac{c_V}{\reg^2} \, (2 + B_\Theta) \, \ell_A
	\int_0^t \|\varphi_{\theta_n}\!(s) - \varphi_\theta(s)\|_{1, \infty}^{M_0} \ \|j(\varphi_{\theta_n}\!(s), \theta_n)\|_{V^*} \, ds
	\\
	&\leq
	C_\Theta \int_0^t \|\varphi_{\theta_n}\!(s) - \varphi_\theta(s)\|_{1, \infty}^{M_0} \, ds ,
\end{align*}
and $\|I_{2, n}(t)\|_{1, \infty}^{M_0}$ (see~\cref{eq:existence_minimizer_ineq2_middle}) by
\begin{align*}
	&\hspace{13pt}
	\|I_{2, n}(t)\|_{1, \infty}^{M_0}
	\\
	&\leq
	\frac{c_V}{\reg} \, (2 + B_\Theta)
	\int_0^t \|j(\varphi_{\theta_n}\!(s), \theta_n)\|_{V^*} \, \|\varphi_{\theta_n}\!(s) - \varphi_\theta(s)\|_{1, \infty}^{M_0} \, ds
	\\
	&\leq
	C_\Theta \int_0^t \|\varphi_{\theta_n}\!(s) - \varphi_\theta(s)\|_{1, \infty}^{M_0} \, ds .
\end{align*}
The assumption that $\{j(\cdot, \theta): \theta \in \Theta\}$ is equi-Lipschitz with respect to $\|\cdot\|_{1, \infty}^{M_0}$ gives the following estimate of $\|I_{3, n}(t)\|_{1, \infty}^{M_0}$:
\begin{align*}
	&\hspace{13pt}
	\|I_{3, n}(t)\|_{1, \infty}^{M_0}
	\\
	&=
	\left\|
		\int_0^t
		\left(
			\left( L_{\varphi_\theta(s)}^{-1} \, j(\varphi_{\theta_n}\!(s), \theta_n) \right) \circ \varphi_\theta(s)
			-
			\left( L_{\varphi_\theta(s)}^{-1} \, j(\varphi_\theta(s), \theta_n) \right) \circ \varphi_\theta(s)
		\right)
		ds
	\right\|_{1, \infty}^{M_0}
	\\
	&\leq
	\int_0^t
	\left\|
		\left( L_{\varphi_\theta(s)}^{-1} \, j(\varphi_{\theta_n}\!(s), \theta_n) \right) \circ \varphi_\theta(s)
		-
		\left( L_{\varphi_\theta(s)}^{-1} \, j(\varphi_\theta(s), \theta_n) \right) \circ \varphi_\theta(s)
	\right\|_{1, \infty}
	ds
	\\
	&\leq
	\int_0^t
	(2 + \|\varphi_\theta(s) - \mathit{id}\|_{1, \infty})
	\left\|
		L_{\varphi_\theta(s)}^{-1} \, j(\varphi_{\theta_n}\!(s), \theta_n)
		-
		L_{\varphi_\theta(s)}^{-1} \, j(\varphi_\theta(s), \theta_n)
	\right\|_{1, \infty}
	ds
	\\
	&\leq
	(2 + B_\Theta) \ \frac{c_V}{\reg} \ \ell_j
	\int_0^t \|\varphi_{\theta_n}\!(s) - \varphi_\theta(s)\|_{1, \infty}^{M_0} \, ds
	\\
	&=
	C_\Theta \int_0^t \|\varphi_{\theta_n}\!(s) - \varphi_\theta(s)\|_{1, \infty}^{M_0} \, ds .
\end{align*}
We now work on $\|\varLambda_n(t)\|_{1, \infty}^{M_0}$.
Recall that
\[
	\varLambda_n(t)
	=
	\int_0^t
	\left(
		\left( L_{\varphi_\theta(s)}^{-1} \, j(\varphi_\theta(s), \theta_n) \right) \circ \varphi_\theta(s)
		-
		\left( L_{\varphi_\theta(s)}^{-1} \, j(\varphi_\theta(s), \theta) \right) \circ \varphi_\theta(s)
	\right)
	ds .
\]
To carry over the last set of arguments in \hyperref[subsec:proof_ctrl_problem_min]{the proof of \hyperref[thm:ctrl_problem_min]{Theorem}~\ref{thm:ctrl_problem_min}}, it remains to show that $\lim_{n \rightarrow \infty} \|\varLambda_n(t)\|_{1, \infty}^{M_0} = 0$ for each $t \in [0, T]$ and $\|\varLambda_n(t)\|_{1, \infty}^{M_0}$ is uniformly bounded in $n$ and $t$. Following inequality~\cref{eq:proof_lambda_bounded}, the uniform boundedness of $\|\varLambda_n(t)\|_{1, \infty}^{M_0}$ in $n$ and $t$ is given by
\[
	\|\varLambda_n(t)\|_{1, \infty}^{M_0}
	\leq
	\int_0^t
	(2 + B_\Theta) \,
	\frac{c_V}{\reg} \, \left( \vphantom{\sum} \|j(\varphi_\theta(s), \theta_n)\|_{V^*} + \|j(\varphi_\theta(s), \theta)\|_{V^*} \right) \,
	ds
	\leq
	C_\Theta .
\]
To prove $\lim_{n \rightarrow \infty} \|\varLambda_n(t)\|_{1, \infty}^{M_0} = 0$ for each $t$,  we define the {integrands} of the two terms in $\varLambda_n(t)$ as $u_n: [0, T] \times \mathbb{R}^3 \rightarrow \mathbb{R}^3$ and $u: [0, T] \times \mathbb{R}^3 \rightarrow \mathbb{R}^3$ by
\[
	u_n(s, x)
	=
	\left( L_{\varphi_\theta(s)}^{-1} \, j(\varphi_\theta(s), \theta_n) \right) \circ \varphi_\theta(s, x)
\]
and
\[ 
	u(s, x)
	=
	\left( L_{\varphi_\theta(s)}^{-1} \, j(\varphi_\theta(s), \theta) \right) \circ \varphi_\theta(s, x) .
\]
We show that the two sequences $(\int_0^t u_n(s) \, ds)_{n = 1}^\infty$ and $(\int_0^t Du_n(s) \, ds)_{n = 1}^\infty$ converge pointwise, are uniformly bounded, and are equicontinuous. Given $s \in [0, T]$ and $x \in \mathbb{R}^3$, define linear operators $\mathcal{L}_{s, x}: V^* \rightarrow \mathbb{R}^3$ and $\widetilde{\mathcal{L}}_{s, x}: V^* \rightarrow \mathbb{R}^{3 \times 3}$ by
\[
	\mathcal{L}_{s, x} \, j'
	=
	\left( L_{\varphi_\theta(s)}^{-1} \, j' \right) \circ \varphi_\theta(s, x)
	\ \ \mbox{ and } \ \ 
	\widetilde{\mathcal{L}}_{s, x} \, j'
	=
	D\!\left( L_{\varphi_\theta(s)}^{-1} \, j' \circ \varphi_\theta(s, x) \right),
\]
which can be estimated by
\begin{equation}
	\label{eq:proof_operators_bounded}
	|\mathcal{L}_{s, x} \, j'| + |\widetilde{\mathcal{L}}_{s, x} \, j'|
	\leq
	\left\| \left( L_{\varphi_\theta(s)}^{-1} \, j' \right) \circ \varphi_\theta(s) \right\|_{1, \infty}
	\leq
	\frac{c_V}{\reg} \, (2 + B_\Theta) \, \|j'\|_{V^*} ,
\end{equation}
so both $\mathcal{L}_{s, x}$ and $\widetilde{\mathcal{L}}_{s, x}$ are bounded linear operators for all $s \in [0, T]$ and $x \in \mathbb{R}^3$. Since $j(\varphi_\theta(s), \theta_n) \rightharpoonup j(\varphi_\theta(s), \theta)$ when $\theta_n \rightarrow \theta$, we obtain pointwise convergence of $(u_n(s))_{n = 1}^\infty$ and $(Du_n(s))_{n = 1}^\infty$ for each $s \in [0, T]$. Uniform boundedness and pointwise convergence of $(\int_0^t u_n(s) \, ds)_{n = 1}^\infty$ and $(\int_0^t Du_n(s) \, ds)_{n = 1}^\infty$ follow from inequality \cref{eq:proof_operators_bounded}, $\int_0^T \|j(\varphi_\theta(s), \theta_n)\|_{V^*} \, ds \leq J_\Theta \, T$, and the dominated convergence theorem. The same process as the one in \hyperref[subsec:proof_ctrl_problem_min]{the proof of \hyperref[thm:ctrl_problem_min]{Theorem}~\ref{thm:ctrl_problem_min}} shows equicontinuity under different constants. Invoking the Arzel{\`a}--Ascoli theorem gives us $\lim_{n \rightarrow \infty} \|\int_0^t (u_n(s) - u(s)) \, ds\|_\infty^{M_0} = 0$ and $\lim_{n \rightarrow \infty} \|\int_0^t (Du_n(s) - Du(s)) \, ds\|_\infty^{M_0} = 0$, which in turn implies $\|\varLambda_n(t)\|_{1, \infty}^{M_0} \rightarrow 0$ for each $t \in [0, T]$.

We have thus proved that
\[
	\|\varphi_{\theta_n}\!(t) - \varphi_\theta(t)\|_{1, \infty}^{M_0}
	\leq
	\|\varLambda_n(t)\|_{1, \infty}^{M_0} + C_\Theta \int_0^t \|\varphi_{\theta_n}\!(s) - \varphi_\theta(s)\|_{1, \infty}^{M_0} \, ds ,
\]
where $\lim_{n \rightarrow \infty} \|\varLambda_n(t)\|_{1, \infty}^{M_0} = 0$ for each $t \in [0, T]$ and $\|\varLambda_n(t)\|_{1, \infty}^{M_0}$ is uniformly bounded in $n$ and $t$. Applying Gronwall's lemma and the dominated convergence theorem as in \hyperref[subsec:proof_ctrl_problem_min]{the proof of \hyperref[thm:ctrl_problem_min]{Theorem}~\ref{thm:ctrl_problem_min}}, we conclude that for each $t \in [0, T]$,
\[
	\|\varphi_{\theta_n}\!(t) - \varphi_\theta(t)\|_{1, \infty}^{M_0} \rightarrow 0
	\ \mbox{ as } \ n \rightarrow \infty .
\]

\subsection{Proof of \cref{prop:elastic_operator}}
\label{subsec:proof_elastic_operator}
We need to check that: (1) $A_\varphi \in \sym(V, V^*)$ for all $\varphi \in \Diffid{1}{\mathbb{R}^3}$; (2) $(A_\varphi \hspace{1pt} v \mid v) \geq 0$ for all $\varphi \in \Diffid{1}{\mathbb{R}^3}$ and $v \in V$; (3) The mapping $\varphi \mapsto A_\varphi$ is Lipschitz with respect to $\|\cdot\|_{1, \infty}^{M_0}$ on $\mathfrak{S}_\gamma$. Point (2) is obvious from the assumption that $E_\varphi(x)$ is positive definite for all $x \in \varphi(M_0)$. Since $id \in \mathfrak{S}_\gamma$ for all $\gamma > 0$, we can derive point (1) from the inequality
\begin{align*}
	|(A_\varphi \hspace{1pt} u \mid v)|
	&\leq
	\int_{M_0}
	\left|
		\left( \vphantom{\sum} E_\varphi(\varepsilon_u, \varepsilon_v) \right) \circ \varphi \ |\det D\varphi|
	\right|
	dx
	\\
	&\leq
	\|Du\|_{\infty} \, \|Dv\|_{\infty} \, \|\det D\varphi\|_\infty
	\left( \int_{M_0} |E_{\mathit{id}} \circ \mathit{id}| \, dx + C_\varphi \right)
	\\
	&
	\leq
	(c_V^2 \, C_\varphi) \ \|u\|_V \, \|v\|_V\,.
\end{align*}

We now proceed to proving point (3).
For $\varphi, \psi \in \mathfrak{S}_\gamma$, we show that there exists $C_\gamma$ such that
\[
	|(A_\varphi \hspace{1pt} u \mid v) - (A_\psi \hspace{1pt} u \mid v)|
	\leq
	C_\gamma \, \|\varphi - \psi\|_{1, \infty}^{M_0} \ \|u\|_V \, \|v\|_V .
\]
We make a change of variables and write
\begin{align}
	&\hspace{13pt}
	|(A_\varphi \hspace{1pt} u \mid v) - (A_\psi \hspace{1pt} u \mid v)|
	\notag
	\\
	&\leq
	\int_{M_0}
	\left|
		\left( \vphantom{\sum} E_\varphi(\varepsilon_u, \varepsilon_v) \right) \circ \varphi \ |\det D\varphi|
		-
		\left( \vphantom{\sum} E_\psi(\varepsilon_u, \varepsilon_v) \right) \circ \psi \ |\det D\psi| \,
	\right|
	dx
	\notag
	\\
	\begin{split}
	\label{eq:proof_elastic_operator_terms}
	&\leq
	\int_{M_0}
	\left| \vphantom{\sum} (E_\varphi \circ \varphi - E_\psi \circ \psi)(\varepsilon_u \circ \varphi, \, \varepsilon_v \circ \varphi) \right|
	|\det D\varphi| \, 
	dx
	\\
	&\hspace{13pt}
	\phantom{}
	+
	\int_{M_0}
	\left| \vphantom{\sum} (E_\psi \circ \psi)(\varepsilon_u \circ \varphi - \varepsilon_u \circ \psi, \, \varepsilon_v \circ \varphi) \right|
	|\det D\varphi| \,
	dx
	\\
	&\hspace{13pt}
	\phantom{}
	+
	\int_{M_0}
	\left| \vphantom{\sum} (E_\psi \circ \psi)(\varepsilon_u \circ \psi, \, \varepsilon_v \circ \varphi - \varepsilon_v \circ \psi) \right|
	|\det D\varphi| \,
	dx
	\\
	&\hspace{13pt}
	\phantom{}
	+
	\int_{M_0}
	\left| \vphantom{\sum} (E_\psi \circ \psi)(\varepsilon_u \circ \psi, \, \varepsilon_v \circ \psi) \right|
	\left| \vphantom{\sum} |\det D\varphi| - |\det D\psi| \right|
	dx .
	\end{split}
\end{align}

For the first term in \cref{eq:proof_elastic_operator_terms}, the assumption on $E$ yields
\begin{align}
	\label{eq:proof_elastic_operator_term1}
	\begin{split}
	&\hspace{12pt}
	\int_{M_0}
	\left| \vphantom{\sum} (E_\varphi \circ \varphi - E_\psi \circ \psi)(\varepsilon_u \circ \varphi, \, \varepsilon_v \circ \varphi) \right|
	|\det D\varphi| \, 
	dx
	\\
	&\leq 
	\|Du\|_\infty \, \|Dv\|_\infty \, \|\det D\varphi\|_\infty
	\int_{M_0}
	\left| \vphantom{\sum} (E_\varphi \circ \varphi - E_\psi \circ \psi) \right|
	dx
	\\
	&\leq
	C_\gamma \, \|\varphi - \psi\|_{1, \infty}^{M_0} \ \|u\|_V \, \|v\|_V ,
	\end{split}
\end{align}
where we have used $\|\det D\varphi\|_\infty \leq C_\gamma$ for all $\varphi \in \mathfrak{S}_\gamma$.

We estimate the second and third terms together by symmetry. Again note that $\mathit{id} \in \mathfrak{S}_\gamma$ for all $\gamma > 0$, so
\[
	\int_{M_0} |E_\psi \circ \psi| \, dx
	\leq
	\int_{M_0} |E_{\mathit{id}} \circ \mathit{id}| \, dx + \alpha_\gamma \, \|\psi - \mathit{id}\|_{1, \infty}^{M_0}
	\leq
	\int_{M_0} |E_{\mathit{id}}| \, dx + \alpha_\gamma \, \gamma
	=
	C_\gamma .
\]
It follows that
\begin{align}
	\label{eq:proof_elastic_operator_term23}
	\begin{split}
	&\hspace{13pt}
	\int_{M_0}
	\left| \vphantom{\sum} (E_\psi \circ \psi)(\varepsilon_u \circ \varphi - \varepsilon_u \circ \psi, \, \varepsilon_v \circ \varphi) \right|
	|\det D\varphi| \,
	dx
	\\
	&\hspace{11pt}
	\phantom{}
	+
	\int_{M_0}
	\left| \vphantom{\sum} (E_\psi \circ \psi)(\varepsilon_u \circ \psi, \, \varepsilon_v \circ \varphi - \varepsilon_v \circ \psi) \right|
	|\det D\varphi| \,
	dx
	\\
	&\leq
	\|D^2 u\|_\infty \, \|\varphi - \psi\|_\infty^{M_0} \, \|Dv\|_\infty \, \|\det D\varphi\|_\infty \, \int_{M_0} |E_\psi \circ \psi| \, dx
	\\
	&\hspace{11pt}
	\phantom{}
	+
	\|Du\|_\infty \, \|D^2 v\|_\infty \, \|\varphi - \psi\|_\infty^{M_0} \, \|\det D\varphi\|_\infty \, \int_{M_0} |E_\psi \circ \psi| \, dx
	\\
	&\leq
	C_\gamma \, \|\varphi - \psi\|_{1, \infty}^{M_0} \ \|u\|_V \, \|v\|_V .
	\end{split}
\end{align}

For the fourth term, we use inequality \cref{eq:ex4.det} and write
\begin{align}
	\label{eq:proof_elastic_operator_term4}
	\begin{split}
	&\hspace{13pt}
	\int_{M_0}
	\left| \vphantom{\sum} (E_\psi \circ \psi)(\varepsilon_u \circ \psi, \, \varepsilon_v \circ \psi) \right|
	\left| \vphantom{\sum} |\det D\varphi| - |\det D\psi| \right|
	dx
	\\
	&\leq
	C \, (\|D\varphi\|_\infty + \|D\psi\|_\infty)^2 \, \|D\varphi - D\psi\|_\infty^{M_0} \ 
	\|Du\|_\infty \, \|Dv\|_\infty \, \int_{M_0} |E_\psi \circ \psi| \, dx
	\\
	&\leq
	C_\gamma \, \|\varphi - \psi\|_{1, \infty}^{M_0} \ \|u\|_V \, \|v\|_V .
	\end{split}
\end{align}

Combining inequality~\cref{eq:proof_elastic_operator_terms} with estimates~\cref{eq:proof_elastic_operator_term1,eq:proof_elastic_operator_term23,eq:proof_elastic_operator_term4}, we conclude that
\[
	|(A_\varphi \hspace{1pt} u \mid v) - (A_\psi \hspace{1pt} u \mid v)|
	\leq
	C_\gamma \, \|\varphi - \psi\|_{1, \infty}^{M_0} \ \|u\|_V \, \|v\|_V .
\]

To justify that the operator in equation \cref{eq:bottom.penalty} also satisfies the hypotheses, we note that the penalty term can be rewritten as
\[
\beta\int_{\mathcal{M}_{\mathrm{bottom}}}
	((u \circ \varphi)^\top D\varphi^{-\top}n_0)((w \circ \varphi)^\top D\varphi^{-\top}n_0)
	\, \frac{|\det D\varphi|}{|D\varphi^{-\top} n_0|} \, d\sigma
\]
(where $n_0$ is a unit normal to $\mathcal{M}_{\mathrm{bottom}}$). One can then work on the  terms in the integral using similar arguments to those made in the previous proof.


\section{Conclusion}
\label{sec:conclusion}

In this paper, we first examined the existence and uniqueness of solutions to general systems $\partial_t \hspace{1pt} \varphi(t, x) =  v(t, \varphi(t, x)), \ \varphi(0, x) = x$, where the vector field is a function of a yank $j$ of the form $v(t) = L_{\varphi(t)}^{-1} \, j(t)$ or $v(t) = L_{\varphi(t)}^{-1} \, j(\varphi(t), \theta)$. We then extended the analysis to prove the existence of solutions to the corresponding inverse problems in which one attempts to recover the yank or the parameters from the observed initial and final volumes.

Although we have focused on the specific operator $L_\varphi^{-1} = (\reg K_V^{-1} + A_\varphi)^{-1} \in \mathscr{L}(V^*, V)$, our theorems can be generalized to an operator in $\mathscr{L}(V^*, V)$ satisfying similar conditions of boundedness and regularity. 

We have presented results of simulated inverse problems assuming shapes are hyperelastic materials. 
Our results indicate that the elasticity assumption together with the data from the boundary of target are not enough to determine the internal yank. Additional information such as the internal structure of the target or a parametric model for $j$ is necessary to tackle these inverse problems. As a proof of concept, we have considered a simple form of yank whose density is the gradient of a parametrized potential advected by deformation and demonstrated the retrievability of the potential function parameters under this setting. A more sophisticated model should likely involve propagation of the potential in addition to advection as a way to account for, e.g., the progression of pathology along with  morphological changes. A possible approach could be to combine the shape evolution equations discussed in this paper with, for example, a reaction-diffusion PDE on the potential function to model its dynamics. We are currently investigating a model of this kind, which comes with the extra technicality of dealing with such PDEs on varying domains, and hope to publish relevant results in the near future.



\appendix
\section{Implementation details}
\label{sec:implementation}
\Cref{subsec:tetrahedralization} covers discretization specific to templates with layered structures. (One may use any discretization procedure if layered structures are not of concern.) In \cref{subsec:gradient}, we include the computation of the gradient of the parametric yank problem with the elastic operators and yank presented in \cref{sec:elasticity}. The computation of the gradient of the free yank problem can be adapted from \cref{subsec:gradient}.

\subsection{Tetrahedralization of layered templates}
\label{subsec:tetrahedralization}

We use the notation of \cref{ex:layers}.
The template shape $M_0$ is discretized into a set of points $\bigcup_{\ell = 1}^L\{q_i^\ell\}_{i = 1}^N$ according to its layered structure $\varPhi$. Points $\{q_i^\ell\}_{i = 1}^N$ are on the same layer $\nu_\ell$, and the vectors $q_i^{\ell + 1} - q_i^\ell$ are parallel to the transversal vector $\partial_\nu \varPhi(q_i^1, \nu_\ell)$ at $q_i^\ell$ for all $i$ and $\ell$ (\cref{fig:tetrahedralization_layers,fig:tetrahedralization_points}). Note that points $\{q_i^1\}_{i = 1}^N$ are on the bottom layer, points $\{q_i^L\}_{i = 1}^N$ are on the top layer, and each discretized layer has the same number of discretized points. Since $\varPhi$ is a diffeomorphism, the same triangulation structure can be applied to each layer (\cref{fig:tetrahedralization_triangles}). It follows that $\{q_{i_1}^\ell, q_{i_2}^\ell, q_{i_3}^\ell, q_{i_1}^{\ell + 1}, q_{i_2}^{\ell + 1}, q_{i_3}^{\ell + 1}\}$ forms a triangular prism for any triangular face $(i_1, i_2, i_3)$ of one layer. Those prisms between the first and second layers are further split into tetrahedra without adding vertices using the procedure introduced in \cite{Dompierre1999}, which guarantees consistent triangular faces across adjacent prisms. To ensure the same tetrahedralization structure between consecutive layers, the tetrahedralization between the first and second layers is then replicated to prisms between consecutive upper layers (\cref{fig:tetrahedralization_mesh}).

\begin{figure}[hbt!]
	\centering
	\begin{subfigure}[t]{0.48\textwidth}
		\centering
		\includegraphics[width = 0.9\textwidth]{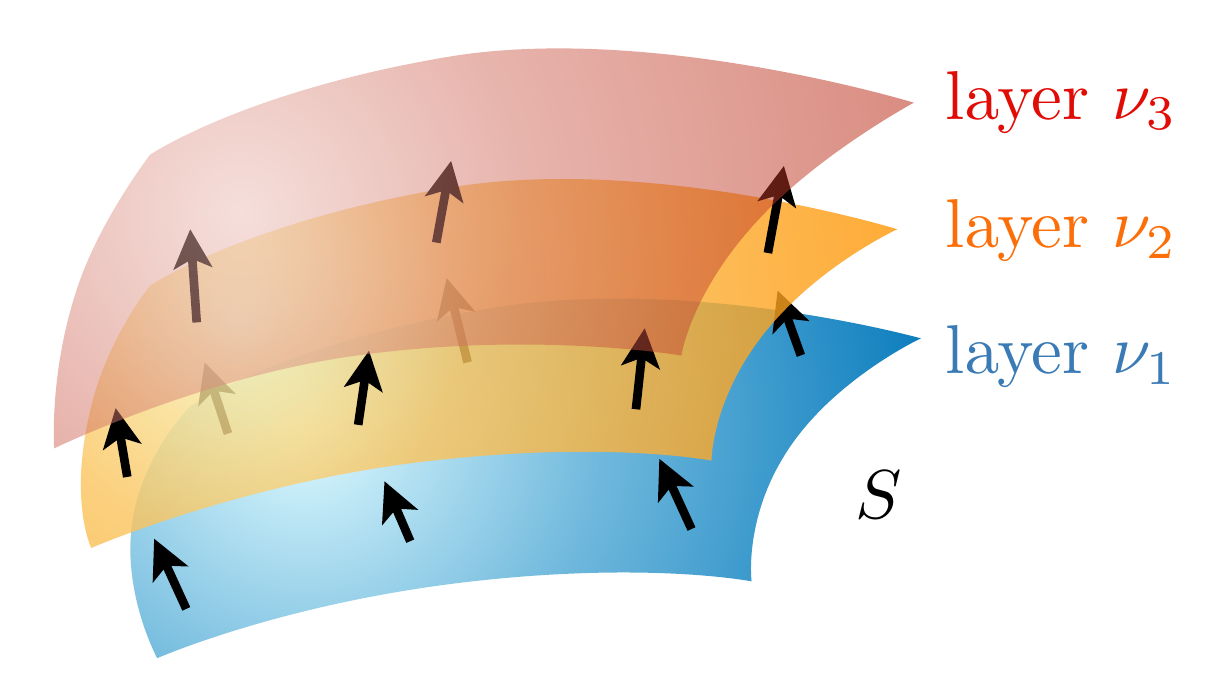}
		\caption{\footnotesize Layers and transversal vector field given by the layered structure.}
		\label{fig:tetrahedralization_layers}
	\end{subfigure}
	~~~
	\begin{subfigure}[t]{0.48\textwidth}
		\centering
		\includegraphics[width = 0.9\textwidth]{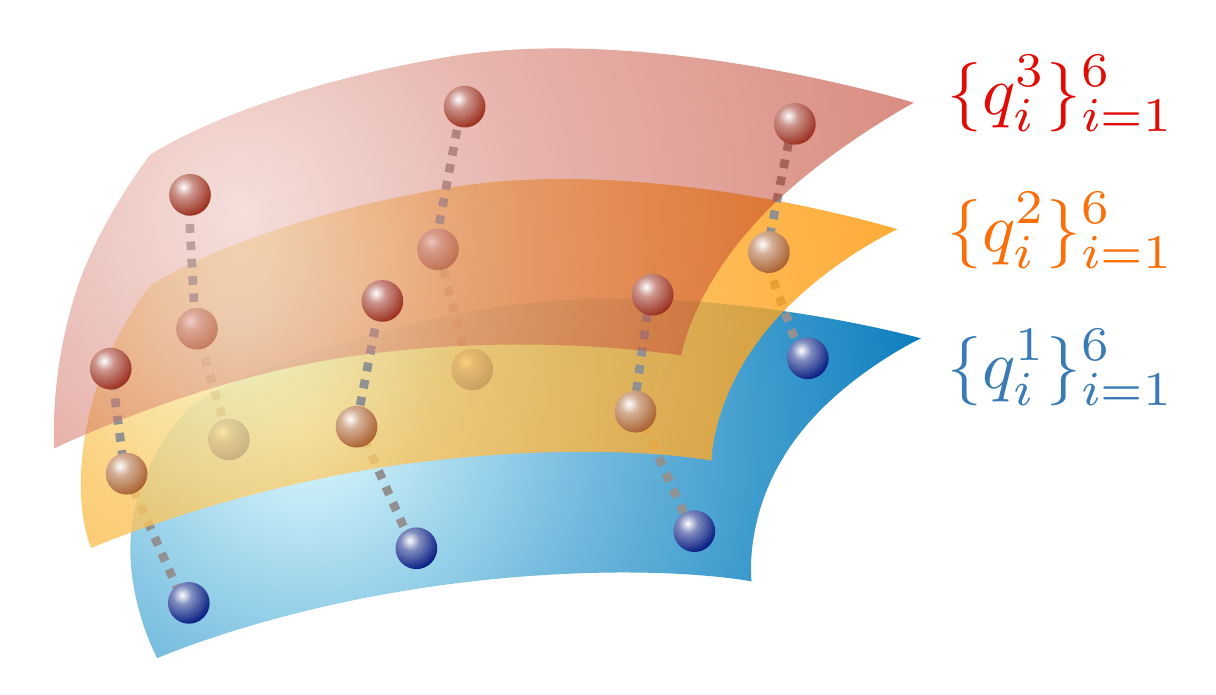}
		\caption{\footnotesize Discretized points according to the layers and the transversal vector field.}
		\label{fig:tetrahedralization_points}
	\end{subfigure}
	\\[10pt]
	\begin{subfigure}[t]{0.48\textwidth}
		\centering
		\includegraphics[width = 0.9\textwidth]{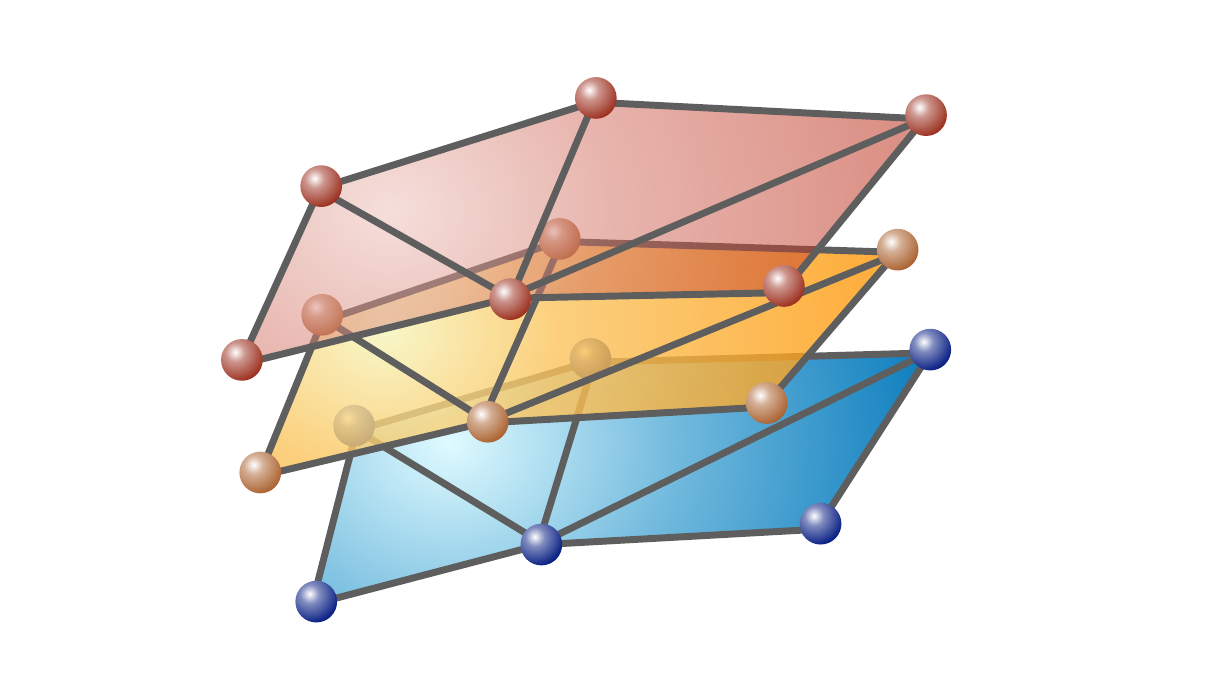}
		\caption{\footnotesize The same triangulation structure applied to each layer.}
		\label{fig:tetrahedralization_triangles}
	\end{subfigure}
	~~~
	\begin{subfigure}[t]{0.48\textwidth}
		\centering
		\includegraphics[width = 0.9\textwidth]{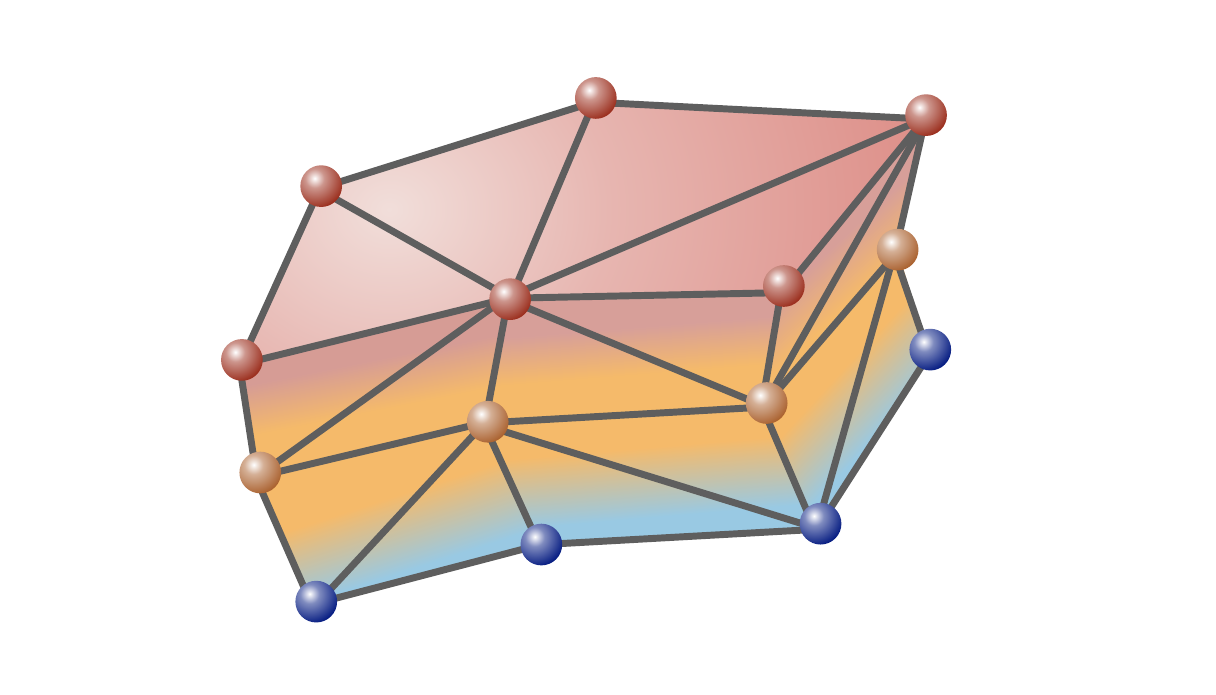}
		\caption{\footnotesize The same tetrahedralization structure applied to volumes between consecutive layers.}
		\label{fig:tetrahedralization_mesh}
	\end{subfigure}
	\caption{Tetrahedralization of layered templates.}
\end{figure}


\subsection{Gradient computation}
\label{subsec:gradient}

We write down the gradient of our optimization problem assuming a continuous time variable, which can be easily discretized in time once an integration scheme for ODEs is selected. Denote the discretized $M_0$ by $q_0 \in \mathbb{R}^{3n}$ and the discretized $M_\targ$ by $q_\targ \in \mathbb{R}^{3n'}$ ($n = NL$ for layered templates). Moreover, denote the kernel matrix of the RKHS $V$ by $\bm{K}_q \in \mathbb{R}^{3n \times 3n}$, namely for $q=(q_i)_{i=1,\ldots,n}$, $\bm{K}_q = (K(q_i,q_j))_{i,j=1,\ldots,n}$ where $K:\mathbb{R}^3 \times \mathbb{R}^3 \rightarrow \mathbb{R}^{3\times 3}$ is the kernel function associated to the vector RKHS $V$. Finally, we write the discretized operator $(A_\varphi \hspace{1pt} u \mid w)$ as $u^\top \bm{A}_q \, w$, and the discretized work $(j(\varphi, \theta) \mid w)$ as $j_{q, \, \theta}^\top \, w$. It follows that the optimal velocity expressed in \cref{lemma:minimizer} becomes $v_{q, \hspace{1pt} \theta} = (\reg \bm{K}_q^{-1} + \bm{A}_q)^{-1} j_{q, \hspace{1pt} \theta} \equalscolon \bm{L}_q^{-1} j_{q, \hspace{1pt} \theta}$ and is obtained numerically by solving a $3n$-by-$3n$ symmetric positive-definite linear system. The discretized optimization problem now becomes
\[
	\min_{\theta \,\in\, \Theta} \ \rho(q(T), q_\targ)
\]
subject to $\dot q(t) = v_{q(t), \, \theta}$ and $q(0) = q_0$. Introduce the costate $p(\cdot)$, where $p(t) \in \mathbb{R}^{3n}$. We then form the Lagrangian
\[
	\mathcal{L}(q, p, \theta) = \rho(q(T), q_\targ) + \int_0^T p(t)^\top (\dot q(t) - v_{q(t), \, \theta}) \, dt.
\]
For each $\theta$, we look for $q_\theta$ and $p_\theta$ such that
\[
	\left\{
		\begin{array}{l}
			\dot q_\theta(t) = v_{q_\theta(t), \, \theta}, \ q_\theta(0) = q_0,
			\\[5pt]
			D_q \mathcal{L}(q_\theta, p_\theta, \theta) = 0,
			\\[5pt]
			D_p \mathcal{L}(q_\theta, p_\theta, \theta) = 0.
		\end{array}
	\right. 
\]
With such chosen $q_\theta$ and $p_\theta$, we deduce that $\left. \partial_{\theta'} \mathcal{L}(q_\theta, p_\theta, \theta') \right|_{\theta' = \theta}$ is the gradient of our discretized optimization problem. We now show how to obtain $q_\theta$ and $p_\theta$. Derivatives of $\mathcal{L}$ with respect to $q$ and $p$ are given by
\begin{align*}
	\left\{
		\begin{array}{l}
			\displaystyle (D_q \mathcal{L}(q, p, \theta) \mid \delta q)
			=
			\left( \partial_{q(T)} \rho(q(T), q_\targ) \right)^\top \delta q(T)
			+
			p(T)^\top \delta q(T)
			\\[5pt]
			\hspace{90pt}
			\displaystyle
			\phantom{}
			-
			\int_0^T \left( \dot p(t) + \partial_{q(t)} (p(t)^\top v_{q(t), \, \theta}) \right)^\top \delta q(t) \, dt,
			\\[10pt]
			\displaystyle (D_p \mathcal{L}(q, p, \theta) \mid \delta p) = \int_0^T \delta p(t)^\top (\dot q(t) - v_{q(t), \, \theta}) \, dt\,.
		\end{array}
	\right. 
\end{align*}
Note that $D_q \mathcal{L}(q, p, \theta) = 0$ is equivalent to $\dot p(t) = -\partial_{q(t)} (p(t)^\top v_{q(t), \, \theta})$ and $p(T) = -\partial_{q(T)} \rho(q(T), q_\targ)$. In addition, $D_p \mathcal{L}(q, p, \theta) = 0$ is equivalent to $\dot q(t) = v_{q(t), \, \theta}$. Hence we can compute the gradient as follows. First we compute $q_\theta$ as a solution of
\[
	\dot q(t) = v_{q(t), \, \theta}, \ q(0) = q_0 .
\]
Plugging $q_\theta$ into the ODE of $p$, next we solve
\[
	\dot p(t) = -\partial_{q(t)} (p(t)^\top v_{q_\theta(t), \, \theta}), \ p(T) = -\partial_{q(T)} \rho(q_\theta(T), q_\targ)
\]
for $p_\theta$. Since $q_\theta$ and $p_\theta$ satisfy the requirements, we can compute the gradient of our discretized optimization problem as
\[
	\left. \partial_{\theta'} \mathcal{L}(q_\theta, p_\theta, \theta') \right|_{\theta' = \theta}
	=
	-\int_0^T \left. \partial_{\theta'} (p_\theta(t)^\top v_{q_\theta(t), \, \theta'}) \right|_{\theta' = \theta} \, dt .
\]

Now the computation of gradient is broken down into three terms: $\partial_q (p^\top v_{q, \hspace{1pt} \theta})$, $\partial_{q(T)} \rho(q(T), q_\targ)$, and $\partial_\theta (p^\top v_{q, \, \theta})$. We will use the notation $u_q^\top (\partial_q \bm{A}_q) \, w_q$ in the following to mean the differentiation of $u_q^\top \bm{A}_q \, w_q$ with respect to $q$ while keeping $u_q$ and $w_q$ fixed. Similarly, the notation $(\partial_q \, j_{q, \hspace{1pt} \theta})^\top w_q$ means the differentiation of $j_{q, \hspace{1pt} \theta}^\top \, w_q$ with respect to $q$ when we fix $w_q$. With those conventions set, we can now formally compute
\begin{align*}
	\partial_q (p^\top v_{q, \hspace{1pt} \theta})
	&=
	\partial_q \left( p^\top \bm{L}_q^{-1} j_{q, \hspace{1pt} \theta} \right)
	\\
	&=
	-(\bm{L}_q^{-1} p)^\top (\partial_q \bm{L}_q ) \, (\bm{L}_q^{-1} j_{q, \hspace{1pt} \theta})
	+
	(\bm{L}_q^{-1} p)^\top (\partial_q \hspace{1pt} j_{q, \hspace{1pt} \theta})
	\\
	&=
	-(\bm{L}_q^{-1} p)^\top (\reg \, \partial_q \bm{K}_q^{-1} + \partial_q \bm{A}_q) \, (\bm{L}_q^{-1} j_{q, \hspace{1pt} \theta})
	+
	(\bm{L}_q^{-1} p)^\top (\partial_q \hspace{1pt} j_{q, \hspace{1pt} \theta})
	\\
	&=
	\reg \, (\bm{K}_q^{-1} \beta_{q, \hspace{1pt} p})^\top (\partial_q \bm{K}_q) \, (\bm{K}_{q}^{-1} v_{q, \hspace{1pt} \theta})
	-
	\beta_{q, \hspace{1pt} p}^\top \, (\partial_q \bm{A}_q) \, v_{q, \hspace{1pt} \theta}
	+
	\beta_{q, \hspace{1pt} p}^\top \, (\partial_q \hspace{1pt} j_{q, \hspace{1pt} \theta}) ,
\end{align*}
where $\beta_{q, \hspace{1pt} p} = \bm{L}_q^{-1} p$. 

We present computations of $u^\top \bm{A}_q \, w$ in \cref{subsubsec:computation_elastic_energy}, $\bm{A}_q \, w$ in \cref{subsubsec:computation_elastic_force}, $u^\top (\partial_q \bm{A}_q) \, w$ in \cref{subsubsec:computation_dq_elastic_energy}, $j_{q, \, \theta}^\top \, w$ in \cref{subsubsec:computation_work}, $j_{q, \hspace{1pt} \theta}$ in \cref{subsubsec:computation_external_force}, and $(\partial_q \hspace{1pt} j_{q, \, \theta})^\top w$  in \cref{subsubsec:computation_dq_work}, which are essential in the above computation of $\partial_q (p^\top v_{q, \hspace{1pt} \theta})$. The computation of $\partial_{q(T)} \rho(q(T), q_\targ)$ depends on the discrepancy function $\rho$. We refer to \cite{Charon2013} when $\rho$ is the varifold discrepancy between triangulated surfaces. Since
\[
	\partial_\theta (p^\top v_{q, \, \theta})
	=
	\partial_\theta (p^\top \bm{L}_q^{-1} j_{q, \theta})
	=
	\partial_\theta (\beta_{q, \hspace{1pt} p}^\top \, j_{q, \theta}) ,
\]
the computation of $\partial_\theta (p^\top v_{q, \, \theta})$ can be derived from $j_{q, \, \theta}^\top \, w$ (\cref{subsubsec:computation_work}). 



To make the presentation more concrete, we focus on the layered elastic operator and yank described in \cref{sec:elasticity}.

\subsubsection{Computation of {\boldmath $u^\top A_q \, w$}}
\label{subsubsec:computation_elastic_energy}

A little computation shows that the layered elastic operator \cref{eq:layered_elastic_operator} can be rewritten as
\begin{align}
	\label{eq:computation_elastic_energy}
	\begin{split}
		(A_\varphi \hspace{1pt} u \mid w)
		&=
		\int_{\varphi(M_0)}
		\left(
			\vphantom{\int}
			\lambda_{\mathrm{tan}}
			\left( \vphantom{\sum} \mathrm{tr}(\varepsilon_u) - N_\varphi^\top \varepsilon_u N_\varphi \right)
			\left( \vphantom{\sum} \mathrm{tr}(\varepsilon_w) - N_\varphi^\top \varepsilon_w N_\varphi \right)
		\right.
		\\
		&\hspace{35pt}
		\phantom{}	
		+
		\mu_{\mathrm{tan}}
		\left(
			\vphantom{\sum}
			\mathrm{tr}(\varepsilon_u \varepsilon_w)
			-
			2 \, N_\varphi^\top \varepsilon_u \varepsilon_w N_\varphi
			+
			(N_\varphi^\top \varepsilon_u N_\varphi)(N_\varphi^\top \varepsilon_w N_\varphi)
		\right)
		\\
		&\hspace{35pt}
		\phantom{}
		+
		\mu_{\mathrm{tsv}} \, (S_\varphi^\top \varepsilon_u S_\varphi)(S_\varphi^\top \varepsilon_w S_\varphi)
		\\
		&\hspace{32pt}
		\left.
			\vphantom{\int}
			\phantom{}
			+
			2 \, \mu_{\mathrm{ang}} \left( \vphantom{\sum} S_\varphi^\top \varepsilon_u \varepsilon_w S_\varphi - (N_\varphi^\top \varepsilon_u S_\varphi)(N_\varphi^\top \varepsilon_w S_\varphi) \right)
		\right) dx ,
	\end{split}
\end{align}
where $\varepsilon_u = \frac{1}{2} \left( Du + Du^\top \right)$ and $\varepsilon_w = \frac{1}{2} \left( Dw + Dw^\top \right)$ are linear strain tensors, $N_\varphi$ is a unit vector field normal to deformed layers $\{\varphi(\mathcal{M}_\nu): \nu \in [0, 1]\}$, and $S_\varphi = \frac{(D\varphi \hspace{1pt} S) \circ \varphi^{-1}}{|(D\varphi \hspace{1pt} S) \circ \varphi^{-1}|}$ is the unit transversal vector field according to the deformed layered structure. After discretizing $\varphi(M_0)$ into a union of tetrahedra, we compute the integral \cref{eq:computation_elastic_energy} by summing over these tetrahedra. Thus we can focus the computation on one single tetrahedron. Note that we need $N_\varphi$, $S_\varphi$, $Du$, and $Dw$ to evaluate~\cref{eq:computation_elastic_energy}. Recall that the tetrahedralization procedure (\cref{subsec:tetrahedralization}) splits one triangular prism into three tetrahedra. Given a tetrahedron, we compute $N_\varphi$ as the average of normals of the two bases of the corresponding prism, and $S_\varphi$ is computed as the average of three sides of the corresponding prism. To be more precise, let the ``upward-pointing" unit normals of two bases be $N_1$ and $N_2$, and let the unit transversals from three sides be $S_1$, $S_2$, and $S_3$. The vectors $N_\varphi$ and $S_\varphi$ of the three tetrahedra split from this prism are computed by
\[
	N_\varphi = \frac{N_1 + N_2}{|N_1 + N_2|}
	\ \ \mbox{ and } \ \ 
	S_\varphi = \frac{S_1 + S_2 + S_3}{|S_1 + S_2 + S_3|} .
\]
For the computation of $Du$, denote the positions at the four vertices of the tetrahedron by $q_0$, $q_1$, $q_2$, $q_3$, and denote $u$ at the four vertices by $u_0$, $u_1$, $u_2$, $u_3$. We approximate $Du(q_0)$ by
\begin{align*}
	Du(q_0)
	&=
	Du(q_0)
	\left[ \begin{array}{ccc} \displaystyle \vphantom{\sum} q_1 - q_0, & q_2 - q_0, & q_3 - q_0 \end{array} \right]
	\left[ \begin{array}{ccc} \displaystyle \vphantom{\sum} q_1 - q_0, & q_2 - q_0, & q_3 - q_0 \end{array} \right]^{-1}
	\\
	&\approx
	\left[ \begin{array}{ccc} \displaystyle \vphantom{\sum} u_1 - u_0, & u_2 - u_0, & u_3 - u_0 \end{array} \right]
	\left[ \begin{array}{ccc} \displaystyle \vphantom{\sum} q_1 - q_0, & q_2 - q_0, & q_3 - q_0 \end{array} \right]^{-1} .
\end{align*}
The approximated $Du(q_0)$ within a tetrahedron $\mathcal{T}$, denoted by $(Du)_\mathcal{T}$, is characterized by
\[
	\left\{
		\begin{array}{l}
			(Du)_\mathcal{T} \, (q_1 - q_0) = u_1 - u_0 \\
			(Du)_\mathcal{T} \, (q_2 - q_0) = u_2 - u_0 \\
			(Du)_\mathcal{T} \, (q_3 - q_0) = u_3 - u_0 \\
		\end{array}
	\right. ,
\]
which is equivalent to
\[
	\left\{
		\begin{array}{l}
			(Du)_\mathcal{T} \, (q_0 - q_1) = u_0 - u_1 \\
			(Du)_\mathcal{T} \, (q_2 - q_1) = u_2 - u_1 \\
			(Du)_\mathcal{T} \, (q_3 - q_1) = u_3 - u_1 \\
		\end{array}
	\right. .
\]
The same pattern holds if we change the anchor position to $q_2$ and $q_3$. In other words, the approximated $(Du)_\mathcal{T}$ only depends on tetrahedron, not on the anchor position, the ordering of vertices, or the choice of three edges from the tetrahedron. $Dw$ is computed in exactly the same way.



\subsubsection{Computation of {\boldmath $A_q \, w = \partial_u (u^\top A_q \, w)$}}
\label{subsubsec:computation_elastic_force}

We use the same notation as in \cref{subsubsec:computation_elastic_energy} and keep focusing on one single tetrahedron. Note that we still denote the discretized $\varepsilon_u$ by $\varepsilon_u$. Define
\[
	U = \left[ \begin{array}{ccc} \displaystyle \vphantom{\sum} u_1 - u_0, & u_2 - u_0, & u_3 - u_0 \end{array} \right]
	\ \ \mbox{ and } \ \ 
	Q = \left[ \begin{array}{ccc} \displaystyle \vphantom{\sum} q_1 - q_0, & q_2 - q_0, & q_3 - q_0 \end{array} \right] ,
\]
so $(Du)_\mathcal{T} = U Q^{-1}$. Since $\mathrm{tr}(\varepsilon_u) = \sum_{i = 1}^3 e_i^\top \varepsilon_u \, e_i$, where $e_i$ is the canonical basis of $\mathbb{R}^3$, we only need to have an expression of $\partial_{u_i} (a^\top \varepsilon_u \, b)$ for arbitrary $a, b \in \mathbb{R}^3$ in order to compute $\partial_u (u^\top \bm{A}_q \hspace{1pt} w)$ (see equation~\cref{eq:computation_elastic_energy}). Note that
\begin{equation}
	a^\top \varepsilon_u \, b
	=
	a^\top \left(\frac{1}{2} \, (U Q^{-1} + Q^{-\top} U^\top)\right) b
	=
	\frac{1}{2} \, \mathrm{tr}\!\left( Q^{-1} (b a^\top + a b^\top)\, U \right) ,
	\label{eq:computation_aeb}
\end{equation}
which gives
\[
	\partial_{u_0}\!\left( a^\top \varepsilon_u \, b \right) = \left( -\frac{1}{2} \, \mathbbm{1}_3^\top Q^{-1} (b a^\top + a b^\top) \right)^\top
\]
and
\[
	\partial_{u_i}\!\left( a^\top \varepsilon_u \, b \right) = \left( \frac{1}{2} \left( Q^{-1} (b a^\top + a b^\top) \right)_{i*} \right)^\top
	\ \ \mbox{ for } i = 1, 2, 3 ,
\]
where $\mathbbm{1}_3$ denotes the 3-by-1 all-one vector, and $(A)_{i*}$ denotes the $i$th row of a matrix $A$.

Let $k$ be the global index running through $n$ discretized points. Note that when we compute $\partial_{u_k} (u^\top \bm{A}_q \hspace{1pt} w)$ by summing over tetrahedra, we only need to take into account those tetrahedra having $q_k$ as a vertex. Other tetrahedra do not have $u_k$ involved in our computation of $u^\top \bm{A}_q \hspace{1pt} w$. This information can be precomputed when we generate the tetrahedralization.


\subsubsection{Computation of {\boldmath $u^\top (\partial_q A_q) \, w = \partial_q(u^\top \!A_q \, w)$}}
\label{subsubsec:computation_dq_elastic_energy}

Differentiating $N_\varphi$, $S_\varphi$, and volume with respect to $q$ is straightforward. Given a tetrahedron with $q_0$, $q_1$, $q_2$, $q_3$ as vertices, we look at $\partial_{q_i} (a^\top \varepsilon_u \, b)$ for arbitrary $a, b \in \mathbb{R}^3$. From~\cref{eq:computation_aeb}, we deduce that
\[
	\partial_{q_0}\!\left( a^\top \varepsilon_u \, b \right)
	=
	\left( \frac{1}{2} \, \mathbbm{1}_3^\top Q^{-1} (b a^\top + a b^\top) \, U Q^{-1} \right)^\top
\]
and
\[
	\partial_{q_i}\!\left( a^\top \varepsilon_u \, b \right)
	=
	\left( -\frac{1}{2} \left( Q^{-1} (b a^\top + a b^\top) \, U Q^{-1} \right)_{i*} \right)^\top
	\ \ \mbox{ for } i = 1, 2, 3 .
\]


\subsubsection{Computation of {\boldmath $j_{q, \, \theta}^\top \, w$}}
\label{subsubsec:computation_work}

Much of the work has been done in \cref{subsubsec:computation_elastic_energy,subsubsec:computation_elastic_force}. Recall that
\[
	(j(\varphi, \theta) \mid w)
	=
	-\int_{\varphi(M_0)} \chi \ g_\theta \circ \varphi^{-1} \, \mathrm{div}(w) \, dx
	=
	-\int_{\varphi(M_0)} \chi \ g_\theta \circ \varphi^{-1} \, \mathrm{tr}(Dw) \, dx .
\]
In a single transformed tetrahedron $\mathcal{T}$, we evaluate $\chi \ g_\theta \circ \varphi^{-1}$ at the transformed centroid to simplify the computation. Denote the evaluated value by $g_\mathcal{T}$. The derivative $Dw$ is approximated in the same way as in \cref{subsubsec:computation_elastic_energy}, that is, $(Dw)_{\mathcal{T}} = W Q^{-1}$, where
\[
	W = \left[ \begin{array}{ccc} \displaystyle \vphantom{\sum} w_1 - w_0, & w_2 - w_0, & w_3 - w_0 \end{array} \right]
	\ \ \mbox{ and } \ \ 
	Q = \left[ \begin{array}{ccc} \displaystyle \vphantom{\sum} q_1 - q_0, & q_2 - q_0, & q_3 - q_0 \end{array} \right] .
\]
The contribution of a single tetrahedron $\mathcal{T}$ in the full integral is then given by
\begin{equation}
	g_\mathcal{T} \, \mathrm{tr}\!\left((Dw)_{\vphantom{\sum}\mathcal{T}}\right) \mathrm{vol}(\mathcal{T})
	=
	g_\mathcal{T} \, \mathrm{tr}\!\left( W Q^{-1} \right) \frac{1}{6} \left|\det Q\right| .
	\label{eq:computation_partial_work}
\end{equation}


\subsubsection{Computation of {\boldmath $j_{q, \, \theta} = \partial_w (j_{q, \, \theta}^\top \, w)$}}
\label{subsubsec:computation_external_force}

From equation~\cref{eq:computation_partial_work}, we obtain the derivatives
\[
	\partial_{w_0}\!\left( g_\mathcal{T} \, \mathrm{tr}\!\left(W Q^{-1}\right) \mathrm{vol}(\mathcal{T}) \right)
	=
	-g_{\mathcal{T}} \, \mathrm{vol}(\mathcal{T}) \left( \mathbbm{1}_3^\top Q^{-1} \right)^\top
\]
and
\[
	\partial_{w_i}\!\left( g_\mathcal{T} \, \mathrm{tr}\!\left(W Q^{-1}\right) \mathrm{vol}(\mathcal{T}) \right)
	=
	g_{\mathcal{T}} \, \mathrm{vol}(\mathcal{T}) \left( (Q^{-1})_{i*} \right)^\top
	\ \ \mbox{ for } i = 1, 2, 3 .
\]


\subsubsection{Computation of {\boldmath $(\partial_q \hspace{1pt} j_{q, \, \theta})^\top w = \partial_q (j_{q, \, \theta}^\top \, w)$}}
\label{subsubsec:computation_dq_work}

Again from \cref{eq:computation_partial_work}, note that $g_\mathcal{T}$ is independent of $q$, so the derivatives from one tetrahedron are given by
\begin{align*}
	&\hspace{14pt}
	\partial_{q_i}\!\left( g_\mathcal{T} \, \mathrm{tr}\!\left(W Q^{-1}\right) \mathrm{vol}(\mathcal{T}) \right)
	\\
	&=
	g_{\mathcal{T}} \, \mathrm{vol}(\mathcal{T}) \, \partial_{q_i}\!\left( \mathrm{tr}\!\left(W Q^{-1}\right) \right)
	+
	g_\mathcal{T} \, \mathrm{tr}\!\left(W Q^{-1}\right) \, \partial_{q_i}\!\left( \mathrm{vol}(\mathcal{T}) \right)
	\ \ \mbox{ for } i = 0, \dots, 3 ,
\end{align*}
where
\[
	\partial_{q_0}\!\left( \mathrm{tr}\!\left(W Q^{-1}\right) \right) = \left( \mathbbm{1}_3^\top Q^{-1} W Q^{-1} \right)^\top
\]
and
\[
	\partial_{q_i}\!\left( \mathrm{tr}\!\left(W Q^{-1}\right) \right) = \left( -(Q^{-1} W Q^{-1})_{i*} \right)^\top
	\ \ \mbox{ for } i = 1, 2, 3 .
\]


\bibliographystyle{siamplain}
\bibliography{Elasticity}

\begin{thebibliography}{10}

\bibitem{adaszewski2013early}
{\sc S.~Adaszewski, J.~Dukart, F.~Kherif, R.~Frackowiak, B.~Draganski, A.~D.~N.
  Initiative, et~al.}, {\em How early can we predict alzheimer's disease using
  computational anatomy?}, Neurobiology of aging, 34 (2013), pp.~2815--2826.

\bibitem{amar2005growth}
{\sc M.~B. Amar and A.~Goriely}, {\em Growth and instability in elastic
  tissues}, Journal of the Mechanics and Physics of Solids, 53 (2005),
  pp.~2284--2319.

\bibitem{amunts2013bigbrain}
{\sc K.~Amunts, C.~Lepage, L.~Borgeat, H.~Mohlberg, T.~Dickscheid, M.-{\'E}.
  Rousseau, S.~Bludau, P.-L. Bazin, L.~B. Lewis, A.-M. Oros-Peusquens, et~al.},
  {\em Bigbrain: an ultrahigh-resolution 3d human brain model}, Science, 340
  (2013), pp.~1472--1475.

\bibitem{aronszajn1950theory}
{\sc N.~Aronszajn}, {\em Theory of reproducing kernels}, Transactions of the
  American mathematical society, 68 (1950), pp.~337--404.

\bibitem{aylward2004onset}
{\sc E.~H. Aylward, B.~Sparks, K.~Field, V.~Yallapragada, B.~Shpritz,
  A.~Rosenblatt, J.~Brandt, L.~Gourley, K.~Liang, H.~Zhou, et~al.}, {\em Onset
  and rate of striatal atrophy in preclinical huntington disease}, Neurology,
  63 (2004), pp.~66--72.

\bibitem{beg2005computing}
{\sc M.~F. Beg, M.~I. Miller, A.~Trouv{\'e}, and L.~Younes}, {\em Computing
  large deformation metric mappings via geodesic flows of diffeomorphisms},
  International journal of computer vision, 61 (2005), pp.~139--157.

\bibitem{braak1991neuropathological}
{\sc H.~Braak and E.~Braak}, {\em Neuropathological stageing of
  alzheimer-related changes}, Acta neuropathologica, 82 (1991), pp.~239--259.

\bibitem{braak1995staging}
{\sc H.~Braak and E.~Braak}, {\em Staging of alzheimer's disease-related
  neurofibrillary changes}, Neurobiology of aging, 16 (1995), pp.~271--278.

\bibitem{bressan2018model}
{\sc A.~Bressan and M.~Lewicka}, {\em A {Model} of {Controlled} {Growth}},
  Archive for Rational Mechanics and Analysis, 227 (2018), pp.~1223--1266.

\bibitem{Charlier2017}
{\sc B.~Charlier, N.~Charon, and A.~Trouv{\'{e}}}, {\em The fshape framework
  for the variability analysis of functional shapes}, J. Foundations of Comput.
  Math, 17 (2017), pp.~287--357.

\bibitem{Charon2013}
{\sc N.~Charon and A.~Trouv\'{e}}, {\em The varifold representation of
  non-oriented shapes for diffeomorphic registration}, SIAM journal of Imaging
  Science, 6 (2013), pp.~2547--2580.

\bibitem{damasio1995human}
{\sc H.~Damasio}, {\em Human brain anatomy in computerized images.}, Oxford
  university press, 1995.

\bibitem{dicarlo2002growth}
{\sc A.~DiCarlo and S.~Quiligotti}, {\em Growth and balance}, Mechanics
  Research Communications, 29 (2002), pp.~449--456.

\bibitem{Dompierre1999}
{\sc J.~Dompierre, P.~Labb{\'{e}}, M.~Vallet, and R.~Camarero}, {\em {How to
  Subdivide Pyramids, Prisms, and Hexahedra into Tetrahedra}}, in Proceedings
  of the 8th International Meshing Roundtable, 1999, pp.~195--204.

\bibitem{Durrleman2013}
{\sc S.~Durrleman, X.~Pennec, A.~Trouv{\'e}, J.~Braga, G.~Gerig, and
  N.~Ayache}, {\em Toward a comprehensive framework for the spatiotemporal
  statistical analysis of longitudinal shape data}, International Journal of
  Computer Vision, 103 (2013), pp.~22--59.

\bibitem{gerig2006computational}
{\sc G.~Gerig, B.~Davis, P.~Lorenzen, S.~Xu, M.~Jomier, J.~Piven, and
  S.~Joshi}, {\em Computational anatomy to assess longitudinal trajectory of
  brain growth}, in Third International Symposium on 3D Data Processing,
  Visualization, and Transmission (3DPVT'06), IEEE, 2006, pp.~1041--1047.

\bibitem{goriely2017mathematics}
{\sc A.~Goriely}, {\em The mathematics and mechanics of biological growth},
  vol.~45, Springer, 2017.

\bibitem{Hsieh2019}
{\sc D.-N. Hsieh, S.~Arguill{\`e}re, N.~Charon, M.~I. Miller, and L.~Younes},
  {\em A model for elastic evolution on foliated shapes}, in Information
  Processing in Medical Imaging, A.~C.~S. Chung, J.~C. Gee, P.~A. Yushkevich,
  and S.~Bao, eds., Springer International Publishing, 2019, pp.~644--655.

\bibitem{hua2011accurate}
{\sc X.~Hua, B.~Gutman, C.~P. Boyle, P.~Rajagopalan, A.~D. Leow, I.~Yanovsky,
  A.~R. Kumar, A.~W. Toga, C.~R. Jack~Jr, N.~Schuff, et~al.}, {\em Accurate
  measurement of brain changes in longitudinal mri scans using tensor-based
  morphometry}, Neuroimage, 57 (2011), pp.~5--14.

\bibitem{Linjeb180414}
{\sc D.~C. Lin, C.~P. McGowan, K.~P. Blum, and L.~H. Ting}, {\em Yank: the time
  derivative of force is an important biomechanical variable in sensorimotor
  systems}, Journal of Experimental Biology, 222 (2019).

\bibitem{lindberg2012hippocampal}
{\sc O.~Lindberg, M.~Walterfang, J.~C. Looi, N.~Malykhin, P.~{\"O}stberg,
  B.~Zandbelt, M.~Styner, B.~Paniagua, D.~Velakoulis, E.~{\"O}rndahl, et~al.},
  {\em Hippocampal shape analysis in alzheimer's disease and frontotemporal
  lobar degeneration subtypes}, Journal of Alzheimer's Disease, 30 (2012),
  pp.~355--365.

\bibitem{lubarda2002mechanics}
{\sc V.~A. Lubarda and A.~Hoger}, {\em On the mechanics of solids with a
  growing mass}, International journal of solids and structures, 39 (2002),
  pp.~4627--4664.

\bibitem{ma2010bayesian}
{\sc J.~Ma, M.~I. Miller, and L.~Younes}, {\em A bayesian generative model for
  surface template estimation}, International journal of biomedical imaging,
  2010 (2010).

\bibitem{marsden1994mathematical}
{\sc J.~E. Marsden and T.~J. Hughes}, {\em Mathematical foundations of
  elasticity}, Courier Corporation, 1994.

\bibitem{miller2015amygdalar}
{\sc M.~I. Miller, L.~Younes, J.~T. Ratnanather, T.~Brown, H.~Trinh, D.~S. Lee,
  D.~Tward, P.~B. Mahon, S.~Mori, M.~Albert, et~al.}, {\em Amygdalar atrophy in
  symptomatic alzheimer's disease based on diffeomorphometry: the biocard
  cohort}, Neurobiology of aging, 36 (2015), pp.~S3--S10.

\bibitem{QIU2009S51}
{\sc A.~Qiu, M.~Albert, L.~Younes, and M.~I. Miller}, {\em Time sequence
  diffeomorphic metric mapping and parallel transport track time-dependent
  shape changes}, NeuroImage, 45 (2009), pp.~S51 -- S60.

\bibitem{ratnanather20183d}
{\sc J.~T. Ratnanather, S.~Arguill{\`e}re, K.~S. Kutten, P.~Hubka, A.~Kral, and
  L.~Younes}, {\em 3d normal coordinate systems for cortical areas}, arXiv
  preprint arXiv:1806.11169,  (2018).

\bibitem{Simon1983}
{\sc L.~Simon}, {\em Lecture notes on geometric measure theory}, Australian
  national university, 1983.

\bibitem{singh2013hierarchical}
{\sc N.~Singh, J.~Hinkle, S.~Joshi, and P.~T. Fletcher}, {\em A hierarchical
  geodesic model for diffeomorphic longitudinal shape analysis}, in
  International Conference on Information Processing in Medical Imaging,
  Springer, 2013, pp.~560--571.

\bibitem{tallinen2016growth}
{\sc T.~Tallinen, J.~Y. Chung, F.~Rousseau, N.~Girard, J.~Lef{\`e}vre, and
  L.~Mahadevan}, {\em On the growth and form of cortical convolutions}, Nature
  Physics, 12 (2016), p.~588.

\bibitem{tang2019regional}
{\sc X.~Tang, C.~A. Ross, H.~Johnson, J.~S. Paulsen, L.~Younes, R.~L. Albin,
  J.~T. Ratnanather, and M.~I. Miller}, {\em Regional subcortical shape
  analysis in premanifest huntington's disease}, Human Brain Mapping, 40
  (2019), pp.~1419--1433.

\bibitem{ward1997mathematical}
{\sc J.~P. Ward and J.~King}, {\em Mathematical modelling of avascular-tumour
  growth}, Mathematical Medicine and Biology: A Journal of the IMA, 14 (1997),
  pp.~39--69.

\bibitem{younes2010shapes}
{\sc L.~Younes}, {\em Shapes and diffeomorphisms}, vol.~171, Springer Science
  \& Business Media, 2010.

\bibitem{younes2018hybrid}
{\sc L.~Younes}, {\em Hybrid riemannian metrics for diffeomorphic shape
  registration}, Annals of Mathematical Sciences and Applications, 3 (2018),
  pp.~189--210.

\bibitem{younes2019identifying}
{\sc L.~Younes, M.~Albert, A.~Moghekar, A.~Soldan, C.~Pettigrew, and M.~I.
  Miller}, {\em Identifying changepoints in biomarkers during the preclinical
  phase of alzheimer's disease}, Frontiers in Aging Neuroscience, 11 (2019),
  p.~74.

\bibitem{younes2019normal}
{\sc L.~Younes, K.~S. Kutten, and J.~T. Ratnanather}, {\em Normal and
  equivolumetric coordinate systems for cortical areas}, arXiv preprint
  arXiv:1911.07999,  (2019).

\bibitem{younes2014regionally}
{\sc L.~Younes, J.~T. Ratnanather, T.~Brown, E.~Aylward, P.~Nopoulos,
  H.~Johnson, V.~A. Magnotta, J.~S. Paulsen, R.~L. Margolis, R.~L. Albin,
  et~al.}, {\em Regionally selective atrophy of subcortical structures in
  prodromal hd as revealed by statistical shape analysis}, Human brain mapping,
  35 (2014), pp.~792--809.

\end{thebibliography}

\end{document}